\documentclass[11pt]{amsart}
\usepackage{amssymb,amsmath,amscd,amsfonts,amsthm,mathrsfs}
\usepackage{verbatim} 
\usepackage{tikz}
\usetikzlibrary{arrows,positioning,automata,shadows}
\input xy
\xyoption{all}
\usepackage[all]{xy}  
\usepackage{fullpage}
\usepackage{cases}
\newcommand*{\dosubeqns}{}
  \makeatletter
  \@ifdefinable \subeqnstoks {\newtoks \subeqnstoks}
  
  \DeclareRobustCommand*{\settheparentequation}{%
    \def \theparentequation
  }
  \newenvironment{subeqns*}[1][gather]{%
    \subequations
    \the \subeqnstoks
    \renewcommand*{\dosubeqns}[1]{\begin{#1}##1\end{#1}}%
    \collect@body \dosubeqns
  }{%
    \protected@edef \@tempa {%
      \global \subeqnstoks {%
        \protect \setcounter{parentequation}{\number\value{parentequation}}%
        \settheparentequation {\theparentequation}%
        \protect \setcounter{equation}{\number\value{equation}}%
      }%
    }%
    \@tempa
    \endsubequations
  }
  \makeatother

\usepackage{dsfont}
\usepackage{color}
\usepackage{pgf}
\usepackage{tikz}
\usepackage[T1]{fontenc}
\usepackage{lmodern}
\usetikzlibrary{arrows,automata}
\usepackage[francais]{babel}

\newtheorem{theorem}{Th\'eor\`eme}[section]%{\bf}{\it }
\newtheorem{proposition}[theorem]{Proposition}%{\bf}{\it }
\newtheorem{construction}[theorem]{Construction}
 
\newtheorem{lemma}[theorem]{Lemme}%{\bf}{\it }
\newtheorem{remark}[theorem]{Remarque}%{\bf}{\it }
\newcommand{\N}{\mathbb{N}}

\newcommand{\Z}{\mathbb{Z}}

\newcommand{\R}{\mathbb{R}}
\newcommand{\C}{\mathbb{C}}

\newcommand{\Dc}{\mathcal{D}}

\def \ds{\displaystyle}

\def \Dc{\text{D}_c}
\def \Dc1{\text{D}_d}
\def \Dc2{\mathcal{D}}

\title{Blocs de chiffres de taille croissante dans les nombres premiers}
\author{Gautier Hanna}
\address{1. Universit\'e de Lorraine, Institut Elie Cartan de Lorraine, UMR 7502, Vandoeuvre-l\`es-Nancy, F-54506, France;
2. CNRS, Institut Elie Cartan de Lorraine, UMR 7502, Vandoeuvre-l\`es-Nancy, F-54506, France}
\email{gautier.hanna@univ-lorraine.fr}



\begin{document}

\begin{abstract}
In this article, we prove a theorem à la Mauduit et Rivat (prime number theorem, Moebius randomness principle) for functions that count digital blocks whose length is a growing function tending to infinity. These sequences are not automatic. To obtain our results, we control sums of type I and II and use an adapted and refined version of the carry propagation property as well as standard methods from harmonic analysis.
\end{abstract}

\maketitle
%\section{Introduction}

\tableofcontents

\section{Introduction}

Dans \cite{gelfond}, Gelfond fournit l'estimation asymptotique du nombre d'entiers en progression arithmétique dont la somme des chiffres est dans une classe de congruence fix\'ee. Il termine son article par une série de questions, dont celle de l'estimation du nombre de nombres premiers inf\'erieurs \`a $x$ tels que $s_q(p) \equiv a \bmod m$, où ici $s_q(\cdot)$ désigne la somme des chiffres en base $q$ ($q$ étant, et sera pour tout cet article, un entier supérieur ou égal à $2$), et $a$ et $m$ deux entiers presque quelconques (une obstruction naturelle pour $a$ et $m$, conséquence d'une généralisation de la règle de $9$, apparaît), et dans \cite{MR:gelfond}, Mauduit et Rivat arrivent à estimer cette quantité près de $40$ ans après que cette question fut posée. 

Peu de temps après la parution de \cite{MR:gelfond}, Kalai, dans \cite{Kalai1}, a posé une question qui peut s'apparenter à une extension de la question de Gelfond : il s'agit de montrer un théorème des nombres premiers pour une suite liée à une somme partielle des chiffres d'un entier. Green \cite{Green} répond partiellement à cette question, puis Bourgain \cite{bourgain} en s'inspirant des travaux de \cite{MR:gelfond} complète la réponse de Green. Cette question possède une généralisation naturelle également proposée par Kalai \cite{Kalai2}. Cette généralisation consiste à démontrer, si $\mu$ est la fonction de Möbius, $P\in \mathbb{F}_2[X_1, \ldots , X_N]$ de degré inférieur à polylog ($N$) (le degré est ici le nombre maximal de variables entrant en compte dans un monôme), et si pour tout entier $n < 2^N$, $\epsilon_0(n), \ldots , \epsilon_{N-1}(n)$ désigne les $N$ premiers chiffres de $n$ en base $2$, $\epsilon_0$ étant le chiffre des unités (les chiffres au-delà de $N$ sont nécessairement nuls, il n'est pas utile de les marquer) que \begin{align}\label{PAM Kalai}
\ds\sum_{n < 2^N}\mu(n)(-1)^{P(n)} = o(2^N),
\end{align} avec ici \begin{equation}\label{Forme de P(n)}
P(n) := P(\epsilon_0(n), \ldots , \epsilon_{N-1}(n)).
\end{equation}

Avec ces notations, puisque $s_2(n) = P(n)$ avec $P(X_1,\ldots , X_N) = \sum_i X_i$, on voit bien que cette question peut être perçue comme une généralisation de celle de Gelfond. La première question de Kalai concernait les polynômes de degré au plus $1$, la seconde s'intéresse aux polynômes de degré strictement plus grand que $1$. Un exemple de polynôme de degré $2$ est donné par $P(X_1,\ldots , X_N) := \sum_i X_iX_{i+1}$ et $f(n) :=
(-1)^{P(n)}$ engendre alors la suite de Rudin-Shapiro. Dans ce cas, Tao \cite{Kalai2} a esquissé une démonstration de \eqref{PAM Kalai} sans terme d'erreur explicite, et Mauduit et Rivat \cite{RudinShapiro} en utilisant la méthode qu'ils ont développé dans \cite{MR:gelfond} ont fourni un terme d'erreur explicite. De plus, Mauduit et Rivat ne se contentent pas de regarder la suite de Rudin-Shapiro, mais une classe de suite $(e(\alpha a(n)))_{n \geq 0}$ où $e(x):= \exp (2i\pi x), \alpha \in \R \backslash \mathbb{Z}$ et $(a(n))_{n \geq 0}$ est une suite liée aux comptages de blocs en base $2$, ce que Mauduit et Rivat nomment suites de Rudin-Shapiro généralisées. 

Dans \cite{Hanna1}, les résultats de \cite{RudinShapiro} ont été étendus dans le cas où $(a(n))_{n \geq 0}$ est une suite $\beta$-récursive possédant certaines propriétés peu restrictives. Les suites $\beta$-récursives généralisent les suites bloc-additives (parfois nommées suites digitales) \cite[Section $3.3$]{AlloucheShallit}, \cite{Cateland} dans le sens où cette fois-ci il est permis de compter les blocs $0\cdots 0$. En particulier, dans \cite{Hanna1}, les résultats de \cite{RudinShapiro} ont été étendus à toutes les suites comptant les blocs, que ceux-ci soient connexes ou non (le bloc $1*1*00$ est non connexe) en base $q$ quelconque. Les résultats de \cite{RudinShapiro} et de \cite{Hanna1} concernent les suites qui sont automatiques, et de ce fait sont induits par \cite{MuellnerAutomate}.

Une explicitation des calculs effectués dans \cite{RudinShapiro} et \cite{Hanna1} permet de montrer qu'il existe deux constantes absolues strictement positives $C_1$ et $C_2$ telles que 
 \begin{equation}\label{Equation Poly Chap 1}
\ds\sum_{n < 2^N}\mu(n)(-1)^{P_N(n)} \ll  N^{C_1} 2^{N-
C_2N\frac{1}{\beta2^\beta}  +o\left(\frac{N}{\beta2^\beta}\right)} \qquad (N \rightarrow \infty),
\end{equation} avec $P_N(X_1,\ldots , X_N) := \sum_{i \leq N- \beta +1} \tilde{P}(X_i , \ldots X_{i+\beta - 1})$ et $\tilde{P}$ un polynôme à $\beta$-variables vérifiant une non-trivialité proche de celle réclamée aux suite $\beta$-récursives, pour plus de précision voir \cite[Théorème $ 0.4.1$]{HannaThese}. En particulier le terme majorant de \eqref{Equation Poly Chap 1} reste $o(2^N)$ pour $\beta < \log N / \log 2$, degré tout à fait acceptable dans l'optique de la conjecture de Kalai : dans la théorie de la complexité, le polylog désigne une puissance quelconque du logarithme. 

Une extension naturelle de ces résultats consiste à s'intéresser à la même question, mais pour des suites liées aux comptages de blocs de taille croissante. Si nous notons à présent $P : \N \rightarrow \N$ non plus un polynôme mais une application croissante, il convient de s'intéresser à \begin{equation}
a_P(n) := \ds\sum_{i \geq 0} \epsilon_{i+P(T_q(n))}(n)\cdots \epsilon_i(n)
\end{equation} où \begin{equation}\label{definition TQ}
T_q(\cdot) := \lfloor \log (\cdot) / \log q \rfloor
\end{equation}  indique l'indice du dernier chiffre non nul dans l'écriture en base $q$.

En base $2$, $a_P(n)$ compte les blocs composés de $P(T_2(n))$ `$1$' consécutifs; si $P$ est la fonction constante ceci correspond à la suite $(b_d(n))_{n \geq 0}$ introduite dans \cite{RudinShapiro}. S'il était possible d'utiliser la méthode de Mauduit et Rivat pour $P(x) = x$; puisque celle-ci fournit un résultat d'équirépartition \cite[Corollary $2$]{RudinShapiro}, une conséquence serait que la moiti\'e des nombres premiers seraient des nombres de Mersenne (voir Partie \ref{Optim} pour plus de précision). Dans cet article, nous montrons que la vitesse $P(y) < \log y / \log q$ est admissible en démontrant le théorème suivant  (ici et dans le reste de l'article, la constante implicite dépend de $q$) : 

\begin{theorem}\label{Thm VM}
Soit $P$ une fonction croissante, positive, et \`a valeurs enti\`eres pour laquelle il existe une constante strictement positive $c < 1/ \log q$ telle que $P(y) \leq c \log y$ pour tout r\'eel $y$ assez grand. Alors uniform\'ement en $\vartheta \in \R$ :
\begin{equation}\label{equa thm chap 2}
\left |\ds\sum_{ n \leq x}\Lambda(n)f_P(n)e(\vartheta n) \right| \ll c_1(q)(\log x)^{3+\frac{\omega(q)}{4}}xq^{-\frac{1}{64}\gamma_P(\frac{1}{120}\lfloor \frac{\log x}{\log q}\rfloor  , \, \lfloor \frac{\log x}{\log q}\rfloor)},
\end{equation}
avec  \begin{equation}\label{Forme gamma deuxieme prop fourier}
\gamma_P(l,k) = l\left(1-\frac{\log \left(q^{P(k)}-8\left(\sin  \frac{\pi |\!|\alpha|\!|}{4} \right)^2\right)}{P(k)\log q} \right),
\end{equation} $\Lambda$ la fonction de von Mangoldt, $f_P(n) = e(\alpha a_P(n))$ et $c_1(q) = q^{13/64}\max\left((\log q)^3,\tau(q)^{1/4}\right)(\log q)^{-3-\frac{\omega(q)}{4}}$. De plus ce théorème reste valable avec $\mu$ en lieu et place de $\Lambda$.
\end{theorem}

\begin{remark}
Si $P(x)$ v\'erifie les conditions du th\'eor\`eme, alors le majorant de l'\'equation \eqref{equa thm chap 2} est $o(x)$ si $x \rightarrow \infty$. %Par exemple si $\alpha = 1/2$ et $$P(x) = \left \lfloor \frac{2}{3} \frac{\log x}{ \log 2} \right \rfloor \leq \frac{2}{3} \frac{\log x}{ \log 2}, $$ 
Par exemple si $q=2, \alpha = 1/2$, comme $\log (1-u) \leq -u$ si $u \in [0,1[$ :
 \begin{align*}
\gamma_P \left(\frac{1}{120} \left\lfloor \frac{\log x}{\log 2} \right\rfloor  , \, \left\lfloor \frac{\log x}{\log 2} \right\rfloor \right) &= - \frac{1}{120} \left\lfloor \frac{\log x}{\log 2}\right\rfloor \log \left(1- \frac{8\left(\sin \pi /8 \right)^2}{2^{P\left(\left\lfloor \frac{\log x}{\log 2} \right\rfloor \right)}} \right) \cdot \frac{1}{P\left(\left\lfloor \frac{\log x}{\log 2} \right\rfloor \right) \log 2}
\\ & \geq  \frac{1}{120} \left\lfloor \frac{\log x}{\log 2}\right\rfloor \frac{8\left(\sin \pi /8 \right)^2}{2^{P\left(\left\lfloor \frac{\log x}{\log 2} \right\rfloor \right)}} \cdot \frac{1}{P\left(\left\lfloor \frac{\log x}{\log 2} \right\rfloor \right) \log 2}
\end{align*} et si $$P(x) = \left \lfloor \frac{2}{3} \frac{\log x}{ \log 2} \right \rfloor \leq \frac{2}{3} \frac{\log x}{ \log 2}, $$ nous obtenons :
\begin{align*}
\gamma_P \left(\frac{1}{120} \left\lfloor \frac{\log x}{\log 2} \right\rfloor  , \, \left\lfloor \frac{\log x}{\log 2} \right\rfloor \right) & \geq \frac{1}{120} \left\lfloor \frac{\log x}{\log 2}\right\rfloor \frac{8\left(\sin \pi /8 \right)^2}{\left\lfloor \log x /\log 2 \right\rfloor^{2/3}} \cdot \frac{1}{2/3 \log \left( \left\lfloor \log x / \log 2  \right\rfloor \right) }.
\end{align*} En posant $c_2' = \frac{3 \cdot 8\left(\sin \pi /8 \right)^2}{2 \cdot 64 \cdot 120} > 0$, nous trouvons que dans ce cas pr\'ecis, le membre de droite de l'\'equation \eqref{equa thm chap 2} est major\'e par \begin{align*}
x (\log x)^{c_1'}2^{-c_2' \frac{\left\lfloor \log x / \log 2 \right\rfloor ^{1/3} }{\log \left\lfloor \log x / \log 2\right\rfloor}}
\end{align*} qui est $o(x)$ si $x \rightarrow \infty$. Une d\'emonstration plus g\'en\'erale de ce fait est donn\'ee dans la Partie \ref{Sec Choix de P}.
\end{remark}
 
Si $\alpha = 1/2$, $f_P$ n'est pas automatique (voir Partie \ref{Annexe 4}), de sorte que les résultats de cet article ne sont pas induits par \cite{MuellnerAutomate}.

\section{Notations}

Pour des raisons techniques nous ne regarderons pas la {suite $(a_P(n))_{n\geq 0}$} mais une application $a_P : \N \times \N \rightarrow \N$ \`a deux variables d\'efinie par : \begin{equation}\label{def ap(x,y)}
a_P(x,y) := \ds\sum_{i \geq 0}\epsilon_{i+P(y)}(x)\cdots\epsilon_i(x).
\end{equation} %la premi\`ere variable nous donne des informations sur les valeurs des chiffres, tandis que la seconde porte sur la taille de l'entier. 
Nous d\'efinissons alors pour $\rho$ un entier, $a_P^{(\rho)}(x,y) := a_P(x \bmod q^\rho,y)$. Avec ces d\'efinitions, nous avons $a_P(n) = a_P(n,T_q(n))$ et $a_P^{(\rho)}(n) = a_P^{(\rho)}(n,T_q(n)) = a_P(n \bmod q^\rho, T_q(n))$, en rappelant que $T_q$ est définie par \eqref{definition TQ}. %Dans ce chapitre, pour des raisons de simplification d'\'ecriture, $T_q(n)$ ne voudra pas dire la taille de $n$ mais le $l$ tel que $n = \ds\sum_{i = 0}^l \epsilon_i q^i.$

 Si $\alpha$ est un nombre r\'eel, nous d\'efinissons  $f_P(x,y) := e(\alpha a_P(x,y))$, $f_P(n):= f_P(n,T_q(n))$, \textit{etc}. Nous d\'efinissons \'egalement pour $0 \leq \mu_1 \leq \mu_2$ des entiers $$f_P^{(\mu_1,\mu_2)}(x,y) := f_P^{(\mu_2)}(x,y)\overline{f_P^{(\mu_1)}(x,y)}.$$ Il s'agit d'une fonction doublement tronqu\'ee qui jouera un r\^ole crucial dans les estimations des sommes de type II.
 
Si $\mu_0, \mu_2$ et $n$ sont des entiers tels que $\mu_0 \leq \mu_2$, nous notons $r_{\mu_0, \mu_2}(n)$ l'entier $u$ tel que $n = kq^{\mu_2} + uq^{\mu_0} + (n \bmod q^{\mu_0})$ et $0 \leq u < q^{\mu_2-\mu_0}$.

\begin{remark}\label{Remarque sur la fonction gammaP}
La fonction $\gamma_P$ d\'efinie par \eqref{Forme gamma deuxieme prop fourier} intervient dans le contr\^ole de la transform\'ee de Fourier de $f_P$. Ainsi, tout comme dans \cite[equation $(25)$]{RudinShapiro} nous avons uniformément en $t \in \R$ $$1 = \ds\sum_{0 \leq h < q^{\lambda}}\left|\frac{1}{q^\lambda}\ds\sum_{0\leq u < q^\lambda}f_P(uq^\kappa,k)e(-u(h+t)) \right|^2 \leq q^{\lambda-2\gamma_P(\lambda,k)} $$ et donc $\gamma_P(\lambda,k) \leq \lambda/2$.
\end{remark}

\section{Panorama de la preuve}\label{Chap 2 Panorama}

%La preuve des Th\'eor\`emes \ref{Thm VM} et \ref{Thm mu} repose essentiellement sur la m\'ethode que Mauduit et Rivat ont d\'evelopp\'ee au cours d'une s\'erie d'articles (voir \cite{digitales, MR:gelfond, RudinShapiro}).

%Leur id\'ee consiste \`a montrer qu'un contr\^ole de la norme infinie de la transform\'ee de Fourier d'une fonction $f$, d\'efinie d'une certaine mani\`ere sur les chiffres, permet de montrer une orthogonalit\'e asymptotique entre $f$ et les fonctions $\Lambda$ et $\mu$, via l'identit\'e de Vaughan (ou son pendant). Ce contr\^ole a \'et\'e formalis\'e par la D\'efinition \ref{propriete fourier mauduit rivat}. 

%Afin de rendre cette connexion {avec la transform\'ee de Fourier} possible nous avons besoin de la D\'efinition \ref{MR anc def 1} qui est une formalisation : pour les fonctions {qui nous int\'eressent,} seuls les chiffres du milieu de la d\'ecomposition {en base $q$} sont significatifs.

%Une tentative naturelle pour montrer un principe d'al\'ea de M\"obius et un th\'eor\`eme des nombres premiers pour une fonction bas\'ee sur les chiffres consiste \`a essayer de v\'erifier l'une et l'autre d\'efinition. Dans le cas o\`u $P$ {est constante,} Mauduit et Rivat le v\'erifient  dans \cite{RudinShapiro} pour notre fonction $a_P$. En proc\'edant \`a un d\'ecoupage $q$-adique de la quantit\'e \`a contr\^oler, il est possible d'obtenir que l'objet que nous \'etudions satisfait la propri\'et\'e de Fourier.
Depuis \cite{RudinShapiro}, une stratégie naturelle pour prouver un principe d'aléa de Möbius pour une fonction définie sur les chiffres est de vérifier les \cite[Definition $1-2$]{RudinShapiro}. S'il est possible dans notre cas, en ajustant les idées de \cite{RudinShapiro}, d'obtenir \cite[Definition $2$]{RudinShapiro}, en revanche, d\`es que $P$ n'est plus constante, il est impossible d'obtenir l'autre condition.

En effet, si $\tau$ est tel que $P(\tau) \neq P(\tau-1)$; alors, pour $n$ et $\rho$ tels que $T_q(n) = \tau$ et $n\bmod q^\rho < n$, %\begin{equation}\label{Equa ap mod rho}
 %a_P(n \bmod q^\rho) = \ds\sum_{i \geq 0}\epsilon_{i+P(T_q(n \bmod q^\rho))}(n\bmod q^\rho)\ldots \epsilon_i(n \bmod q^\rho),
%\end{equation} 
nous avons alors $ P(T_q(n)) = P(\tau) \neq P(T_q(n \bmod q^\rho))$. {Ceci veut dire que le nombre de chiffres consid\'er\'es dans nos comptages de blocs pour $n$ et $n \bmod q^\rho$ n'est pas le m\^eme et on ne peut pas esp\'erer de simplification dans l'expression $a_P(n) - a_P(n \bmod q^\rho)$.}  Il n'y a donc \textit{a priori} aucune raison de pouvoir contr\^oler le nombre de $l$ tels que $$a_P(lq^\kappa + k_1)-  
a_P(lq^\kappa+k_1+k_2) \neq a_P((lq^\kappa + k_1) \bmod q^\rho)-  
a_P((lq^\kappa+k_1+k_2) \bmod q^\rho)$$ si $lq^\kappa$ est ``beaucoup plus grand'' que $q^\rho$. Il devient nécessaire de reprendre en profondeur les arguments de \cite{RudinShapiro} pour espérer obtenir un résultat.

\smallskip

%Comme dit pr\'ec\'edemment, la preuve de \cite{RudinShapiro} utilise l'identit\'e de Vaughan, qui relie $$\ds\sum_{n \leq x}\Lambda(n)f(n)$$ \`a des sommes plus faciles \`a estimer : $S_I$ d\'efinie en \eqref{def SI} et $S_{II}$ d\'efinie en \eqref{def SII}.
L'obtention du \cite[Theorem $1$]{RudinShapiro} se fait à l'aide de l'étude des sommes de type I et de type II.
L'adaptation de la technique utilis\'ee dans \cite{RudinShapiro} pour contr\^oler $S_I$, la somme de type I nécessite la définition suivante : pour $n$ un entier en consid\'erant l'\'ecriture $n = u+vq^\kappa$ {avec $0\leq u<q^\kappa$,} en posant $T_q(n) = l$, et $w = v \bmod q^{\rho}$, nous introduisons $n' = u+wq^{\kappa}+q^{\kappa+\rho}\lfloor q^{l-\rho}\rfloor $. Cette introduction permet de s\'eparer l'information digitale de multiplicative, ingrédient primordial dans la méthode de Mauduit et Rivat.

Pour $S_{II}$, l'introduction de la fonction doublement tronquée nécessite le Lemme \ref{Lemme Olivier Robert}, qui dit que si $m$ et $n$ sont des entiers, et $m'$ et $n'$ des petites perturbations de ces entiers, alors les produits respectifs auront le même nombre de chiffres. Il y a (au moins) deux manières différentes de démontrer ce lemme, l'une consistant en l'étude de la série de Fourier de la fonction $T_q(\cdot) := \left\lfloor \log (\cdot)/\log q \right\rfloor$, et elle nous a été suggérée par Olivier Robert, l'autre consiste à reprendre certaines idées évoquées dans \cite[Lemma $3.5$]{PrimesDigits}, et nous a été fournie par Thomas Stoll.

Ces introductions faites, plusieurs difficultés surgissent, mais leur résolution est plus de l'ordre technique et se fait à travers l'utilisation de résultats classiques (Kusmin-Landau, comportement en moyenne de la fonction $\tau$, \textit{etc.})

\medskip

Cet article est pr\'esent\'e comme suit. Les Parties \ref{Sec Prepa I} et \ref{Sec Prepa II} sont d\'edi\'ees \`a la collecte de r\'esultats utiles, respectivement, au traitement des sommes de type I et de type II. Ces r\'esultats sont de diff\'erentes natures (analytiques, digitaux et harmoniques) et leur d\'emonstration est assez diff\'erente du sch\'ema global de la preuve principale : les inclure nuirait \`a la compr\'ehension structurelle du traitement des sommes. Dans la Partie \ref{Section SI} nous exerçons un contr\^ole sur les sommes de type I, et dans la Partie \ref{Section SII}, sur les sommes de type II. Dans la Partie \ref{Sec Fin estim} nous r\'eunissons les r\'esultats des Parties \ref{Section SI} et \ref{Section SII} afin d'obtenir le Th\'eor\`eme \ref{Thm VM}. Dans la Partie \ref{Sec Choix de P}, nous r\'eunissons les diff\'erentes contraintes obtenues sur la fonction $P$ dans les parties pr\'ec\'edentes afin de d\'eterminer une croissance possible de $P$. Nous obtenons notamment que pour toute fonction \`a valeurs enti\`eres $P$, telle qu'il existe une constante strictement positive $c < 1/\log q$ satisfaisant à $P(x) \leq c\log x$ pour tout r\'eel $x$, la fonction $f_P(n) = e(\alpha a_P(n))$ v\'erifie un th\'eor\`eme des nombres premiers et un principe d'al\'ea de M\"obius pour tout $\alpha$ non entier. Dans la Partie \ref{Annexe 4}, nous montrons que les résultats présentés ici ne découlent pas de \cite{MuellnerAutomate} dans le sens où les suites traitées dans cet article ne sont pas automatiques. Enfin, dans la Partie \ref{Optim}, nous montrons que les résultats obtenus sont optimaux. Dans tout cet article, $P$ d\'esigne une fonction \`a valeurs enti\`eres.
\section{Travail pr\'eparatoire pour les sommes de type I}\label{Sec Prepa I}
L'essentiel des r\'esultats pr\'esent\'es ici servent \`a la Partie \ref{Section SI}. Le Lemme \ref{Lemme propa SI} sert aussi \`a la Partie \ref{Sec Prepa II}, il est de nature digitale. Le Lemme \ref{Lemme premiere estimation Fourier} revêt de l'analyse harmonique.

L'\'enonc\'e du lemme suivant est le r\'esultat final d'un (assez) long calcul vou\'e \`a apparaître plusieurs fois dans la Partie \ref{Section SI}.
\begin{lemma}\label{Lemme pour la somme de type I}
Soient $\mu$, $\nu$, $q \geq 2$ et $d$ des entiers. Soit $M$ un entier tel que $q^{\mu-1}\leq M < q^{\mu}$ et $\kappa_d$ un entier tel que $1 \leq \kappa_d \leq \frac{2}{3}(\mu+\nu)$  et $q^{\kappa_d-1} <M^2/d^2 \leq q^{\kappa_d}$. Consid\'erons enfin $\rho_1, h, u$ des entiers tels que  $0 \leq \rho_1 \leq \mu+\nu-\kappa_d$, $0 \leq h < q^{\rho_1}$ et $0 \leq u < q^{\kappa_d}$ ainsi que $\vartheta'$ un r\'eel. Nous posons pour $l$ un entier quelconque \begin{align*}
c_{\kappa_d,\rho_1,l}(u,h) = \frac{1}{q^{\rho_1}}\ds\sum_{w < q^{\rho_1}} f_P^{(\kappa_d+\rho_1)}&(u+wq^{\kappa_d}+q^{\kappa_d+\rho_1}\lfloor q^l/q^{\rho_1} \rfloor)\\ & \cdot \overline{f_P^{(\kappa_d+\rho_1)}(wq^{\kappa_d}+q^{\kappa_d+\rho_1}\lfloor q^l/q^{\rho_1} \rfloor)}e\left(- \frac{hw}{q^{\rho_1}}\right), 
\end{align*}
 et, si $(k,m)$ désigne le pgcd de $k$ et $m$ : $$S(M,d,l) = \ds\sum_{0 \leq h < q^{\rho_1}}\ds\sum_{\frac{M}{qd}\leq m' < \frac{M}{d}}\frac{1}{m'}\ds\sum_{\substack{0 \leq k' < m' \\ (k',m')=1}}\left|\ds\sum_{0 \leq u < q^{\kappa_d}}c_{\kappa_d,\rho_1,l}(u,h)e\left(-\frac{u\vartheta'}{q^{\mu+\nu}}+\frac{uk'}{m'}\right) \right|. $$ 
  Alors \begin{equation*}
|S(M,d,l)| \ll (\log q) q^{\rho_1/2+\kappa_d}.
\end{equation*}
\end{lemma}

\begin{proof}
Nous assemblons des parties \'eparses de la d\'emonstration des sommes de type I de \cite{RudinShapiro}.

Pour commencer \begin{equation}\label{Equa 1/M'}
\ds\sum_{\frac{M}{qd}\leq m' < \frac{M}{d}}\ds\sum_{\substack{0 \leq k' < m' \\ (k',m')=1}}\frac{1}{m'^2} \leq \ds\sum_{\frac{M}{qd}\leq m' < \frac{M}{d}}\frac{1}{m'} \ll \log q,
\end{equation} donc par l'in\'egalit\'e de Cauchy-Schwarz, nous avons $$|S(M,d,l)|^2 \ll (\log q)q^{\rho_1}\ds\sum_{0 \leq h < q^{\rho_1}}\ds\sum_{\frac{M}{qd}\leq m' < \frac{M}{d}}\ds\sum_{\substack{0 \leq k' < m' \\ (k',m')=1}}\left|\ds\sum_{0 \leq u < q^{\kappa_d}}c_{\kappa_d,\rho_1,l}(u,h)e\left(-\frac{u\vartheta'}{q^{\mu+\nu}}+\frac{uk'}{m'}\right) \right|^2. $$ Nous utilisons alors l'in\'egalit\'e du Grand Crible pour affirmer 
\begin{equation*}\label{equa 1 lemme 1}
|S(M,d,l)|^2 \ll (\log q)q^{\rho_1}\ds\sum_{0 \leq h < q^{\rho_1}}\left(q^{\kappa_d}+\frac{M^2}{d^2} \right)\ds\sum_{0 \leq u < q^{\kappa_d}}|c_{\kappa_d,\rho_1,l}(u,h)|^2.
\end{equation*} Cependant, par orthogonalité des caractères, %\begin{align*}
%& \ds\sum_{0 \leq h < q^{\rho_1}}|c_{\kappa_d,\rho_1,l}(u,h)|^2
%\\ &= \ds\sum_{0 \leq h < q^{\rho_1}}\frac{1}{q^{2\rho_1}}\ds\sum_{0 \leq w,w' < q^{\rho_1}}e\left(- \frac{h(w-w')}{q^{\rho_1}}\right)
%\\ & \qquad \quad   f_P^{(\kappa_d+\rho_1)}(u+wq^{\kappa_d}+q^{\kappa_d+\rho_1}\lfloor q^l/q^{\rho_1} \rfloor)\overline{f_P^{(\kappa_d+\rho_1)}(u+w'q^{\kappa_d}+q^{\kappa_d+\rho_1}\lfloor q^l/q^{\rho_1} \rfloor)}
%\\ & \qquad \quad \overline{f_P^{(\kappa_d+\rho_1)}(wq^{\kappa_d}+q^{\kappa_d+\rho_1}\lfloor q^l/q^{\rho_1} \rfloor)}f_P^{(\kappa_d+\rho_1)}(w'q^{\kappa_d}+q^{\kappa_d+\rho_1}\lfloor q^l/q^{\rho_1} \rfloor)
%\\ & = \frac{1}{q^{\rho_1}}\ds\sum_{0 \leq w < q^{\rho_1}}\Big|f_P^{(\kappa_d+\rho_1)}(wq^{\kappa_d}+q^{\kappa_d+\rho_1}\lfloor q^l/q^\rho_1 \rfloor)\Big|^2 \Big|f_P^{(\kappa_d+\rho_1)}(u+wq^{\kappa_d}+q^{\kappa_d+\rho_1}\lfloor q^l/q^{\rho_1} \rfloor ) \Big|^2
%\\ &= 1
%\end{align*} 
\begin{align*}
& \ds\sum_{0 \leq h < q^{\rho_1}}|c_{\kappa_d,\rho_1,l}(u,h)|^2
\\ & = \frac{1}{q^{\rho_1}}\ds\sum_{0 \leq w < q^{\rho_1}}\Big|f_P^{(\kappa_d+\rho_1)}(wq^{\kappa_d}+q^{\kappa_d+\rho_1}\lfloor q^l/q^{\rho_1} \rfloor)\Big|^2 \Big|f_P^{(\kappa_d+\rho_1)}(u+wq^{\kappa_d}+q^{\kappa_d+\rho_1}\lfloor q^l/q^{\rho_1} \rfloor ) \Big|^2
\end{align*}
qui vaut $1$ par d\'efinition de $f_P$. Ainsi \begin{align*}
|S(M,d,l)|^2 \ll (\log q)q^{\rho_1}\left(q^{\kappa_d}+\frac{M^2}{d^2} \right)q^{\kappa_d},
\end{align*}  et comme $q^{\kappa_d-1} < M^2/d^2 \leq q^{\kappa_d}$ le r\'esultat est obtenu.
\end{proof}

Nous d\'emontrons \`a pr\'esent un pendant de la propri\'et\'e de propagation \cite[Definition $1$]{RudinShapiro}.

\begin{lemma}\label{Lemme propa SI}
 Soient $\rho, \lambda$ et $\kappa$ des entiers tels que $P(\lambda+\kappa +1)\leq \rho$ et $\rho \leq \frac{3}{4}\lambda$. Soit  $B$ l'ensemble des $l$ tels que $0 \leq l < q^{\lambda}$ tels qu'il existe $0 \leq k_1,k_2 < q^{\kappa}$ tels que \begin{equation*}\label{Equation propagation taille changeante}
a_P(lq^{\kappa}+k_1+k_2)- a_P(lq^{\kappa}+k_1) \neq a_P^{(\kappa + \rho)}(lq^{\kappa}+k_1+k_2) - a_P^{(\kappa + \rho)}(lq^{\kappa}+k_1).
\end{equation*}
Alors \begin{equation*}
\# B \ll  q^{\lambda-\rho+P\left(\lambda+\kappa + 1\right)} .
\end{equation*}
\end{lemma}

\begin{proof}
Nous rappelons qu'ici \textit{a priori} \begin{align*}
a_P^{(\rho)}(n)& = \ds\sum_{i \geq 0}\epsilon_{i + T_q(n)}(n \bmod q^\rho)\cdots\epsilon_i(n\bmod q^\rho)
\\ &  \neq \ds\sum_{i \geq 0}\epsilon_{i + T_q(n \bmod q^{\rho})}(n \bmod q^\rho)\cdots\epsilon_i(n\bmod q^\rho).
\end{align*} 

Commen\c cons par montrer que le nombre de $l< q^{\lambda}$ tels que $T_q(lq^{\kappa}+k_1+k_2) \neq T_q(lq^{\kappa}+k_1)$ est major\'e par $\lambda + 1$. En effet, cela n'arrive que s'il existe un $m$ de sorte que $lq^{\kappa}+k_1 < q^m < lq^{\kappa}+k_1+k_2$. Ceci implique $0 \leq l < q^{m-\kappa} < l+2$ et donc $l = q^{m-\kappa}-1$ si $m \geq \kappa$ et $l = 0$ si $m \leq \kappa$. Les valeurs de $l$ possibles sont donc $\{0,q-1,q^2-1, \ldots , q^\lambda -1\}$ soient $\lambda + 1$ valeurs.

Nous pouvons \`a pr\'esent supposer que $T_q(lq^{\kappa}+k_1+k_2) = T_q(lq^{\kappa}+k_1)$ quels que soient $l, k_1,k_2$. 
Des techniques classiques sur les retenues (voir par exemple \cite[Section $11$]{RudinShapiro} ou \cite[Proposition $4.2$]{Hanna1}) montrent qu'alors, à $y$ fixé, le cardinal à estimer est $O(q^{\lambda - \rho + P(y)})$, nous exploitons la croissance de $P$ pour affirmer  \begin{align*}
\#\mathcal{B} \ll q^{\lambda - \rho + P(\lambda + \kappa + 1)} + \lambda + 1
\end{align*} et nous concluons la preuve en utilisant $\rho \leq 3 \lambda/4$.
\end{proof}

%Le lemme suivant fournit la propri\'et\'e de Fourier.

\begin{lemma}\label{Lemme premiere estimation Fourier} Soient $l \geq 0$ et $\kappa$ des entiers tel que $P(\kappa+l) \leq l$, alors uniform\'ement en $t \in \R$:
$$\left|\ds\sum_{q^{l-1}\leq u < q^l}f_P(uq^\kappa,T_q(uq^\kappa))e(-ut)\right| \leq q^{\gamma (l,\kappa)}, $$ avec \begin{equation}
\gamma (l,\kappa) := l\frac{\log \left(q^{P(\kappa+l)}-8\left(\sin  \frac{\pi |\!|\alpha|\!|}{4} \right)^2\right)}{P(\kappa+l)\log q} .
\end{equation}
\end{lemma} 

\begin{proof}

Nous commençons par remarquer que, \`a $\kappa$ fix\'e, $T_q(uq^{\kappa})$ est constante sur $q^{l-1}\leq u < q^l$ et donc $P(T_q(uq^{\kappa})) = P(\kappa+l-1)$. Soit donc $\beta = P(\kappa + l-1)$ et $V_u$ le $u$-i\`eme vecteur g\'en\'ealogique li\'e \`a la fonction $\beta$-r\'ecursive $a(n)=\sum \epsilon_{i+\beta}(n)\cdots \epsilon_i(n)$ (\cite[Définitions $2.1$ et $5.4$]{Hanna1}); alors, de mani\`ere similaire à \cite{Hanna1}, nous obtenons \begin{align*}
\left|\ds\sum_{q^{l-1}\leq u < q^l}f_P(uq^\kappa,T_q(uq^\kappa))e\left(-ut\right)\right|  \leq  \left(q^{P(l+\kappa)}-8\left(\sin  \frac{\pi |\!|\alpha|\!|}{4} \right)^2 \right)^{ \ds\frac{l}{P(\kappa+l)} },
\end{align*} ce qui ach\`eve la preuve.
\end{proof}

%\begin{proposition}\label{Croissance des fonctions}
%Si $P(0) \geq 3$, la fonction $(l,\kappa) \mapsto $ est une fonction croissante en $l$ et en $\kappa$ sur $[0, \infty)$. De plus la fonction $\gamma$ v\'erifie $\gamma (l,\kappa)/l = \gamma (l+k,\kappa-k)$ pour tout $k$ r\'eel.
%\end{proposition}

%\begin{proof}
%Nous avons, par d\'efinition : \begin{align*}
%\ds\frac{\gamma (l,\kappa)}{l} := \frac{\log \left(q^{P(\kappa+l)}-8\left(\sin  \frac{\pi |\!|\alpha|\!|}{4} \right)^2\right)}{P(\kappa+l)\log q},
%\end{align*} qui montre bien la seconde propriété. De plus, pour tout $K >0$, l'application \begin{align*}
%x \mapsto \frac{\log (x-K)}{\log x} 
%\end{align*} a sa d\'eriv\'ee %\begin{align*}
%\ds\frac{1}{(\log x)^2}\cdot \left(\ds\frac{\log x}{x-K}-\ds\frac{\log (x-K)}{x}\right) = \ds\frac{1}{(\log x)^2}\cdot \ds\frac{x\log x - (x-K)\log (x-K)}{x(x-K)}
%\end{align*} qui 
%positive si $x \geq K$ par croissance de la fonction $f(x) = x \log x$. Ici $x = q^{P(\kappa+l)}$ et $K = 8\left(\sin  \frac{\pi |\!|\alpha|\!|}{4} \right)^2$. La condition est donc v\'erifi\'ee si  $P(0) \geq 3$ par croissance de $P$.
%\end{proof}

%Nous sommes maintenant en mesure de contr\^oler les sommes de type I.
\section{Sommes de type I}\label{Section SI}

Soient $1 \leq M \leq N$ tels que 
\begin{equation}\label{Equa M SI}
M \leq (MN)^{1/3}
\end{equation} et $\mu$ et $\nu$ les uniques entiers tels que $q^{\mu-1} \leq M < q^{\mu}$ et $q^{\nu-1} \leq N < q^{\nu}$.% (ou encore $T_q(M) = \mu-1, T_q(N) = \nu-1$). 

Soient \'egalement $\vartheta \in \R$ et $I:=I(M,N) \subset [0,MN]$ un intervalle. Nous posons alors \begin{equation}\label{def SI}
S_I(\vartheta) := \ds\sum_{M/q < m \leq M}\left|\ds\sum_{ \substack{n : \\ mn \in I(M,N)}}f_P(mn)e(\vartheta mn) \right|.
\end{equation}

\begin{theorem}
Soit $P$ tel que 
\begin{equation}\label{Condition sur P SI}
P(\mu+\nu+1) \leq \gamma_P\left(\frac{1}{3}(\mu + \nu), \, \mu+\nu\right),
\end{equation}
alors nous avons uniform\'ement en $\vartheta \in \R$ :
\begin{equation*}
S_I(\vartheta) \ll (\mu+\nu)^3(\log q)^3q^{\mu+\nu- \frac{1}{4}\gamma_P(\frac{1}{3}(\mu + \nu), \, \mu+\nu)},
\end{equation*} où $\gamma_P(k,l)$ est d\'efinie en \eqref{Forme gamma deuxieme prop fourier}.
\end{theorem}

\begin{proof}
En suivant la preuve de \cite{RudinShapiro}, c'est à dire en réécrivant la somme intérieure et en posant $$S_I'(\vartheta) = \ds\sum_{M/q < m \leq M} \frac{1}{m} \ds\sum_{0 \leq k < m}\left|\ds\frac{1}{q^{\mu+\nu}}\ds\sum_{0 \leq l < q^{\mu+\nu}}f_P(l)e\left(- \frac{l\left( \vartheta - \frac{k}{m}q^{\mu + \nu}\right)}{q^{\mu+\nu}}\right) \right|, $$ nous obtenons \cite[equation ($31$)]{RudinShapiro} : \begin{equation}\label{lien SI et sommes sous jacentes}
S_I(\vartheta) \ll q^{\mu+\nu}\log q^{\mu+\nu} \left( \max_{\vartheta' \in \R}S_I'(\vartheta') \right).
\end{equation}
%Nous introduisons \`a pr\'esent la fonction tronqu\'ee. Pour ce faire, 
Pour tout \begin{equation}\label{Condi kappa SI}
1 \leq \kappa \leq \frac{2}{3}(\mu+\nu),
\end{equation} nous introduisons $1 \leq \rho \leq \mu + \nu - \kappa$, que l'on choisira explicitement par la suite. En \'ecrivant $l = u+vq^\kappa$, on obtient, si $\tilde{\mathcal{W}}_{\kappa} := \{u,v : f_P(u+vq^{\kappa})\overline{f_P(vq^{\kappa})}\neq f_P^{(\kappa + \rho)}(u+vq^{\kappa})\overline{f_P^{(\kappa + \rho)}(vq^{\kappa})} \}$  
\begin{align}
\notag g(t) := & \ds\frac{1}{q^{\mu+\nu}}\ds\sum_{l < q^{\mu+\nu}}f_P(l)e\left(- \frac{lt}{q^{\mu+\nu}} \right) \\ & \label{G1} = \frac{1}{q^{\mu+\nu-\kappa}}\ds\sum_{v < q^{\mu+\nu-\kappa}}f_P(vq^{\kappa})e\left(-\frac{vt}{q^{\mu+\nu-\kappa}} \right) \frac{1}{q^{\kappa}}\ds\sum_{u < q^{\kappa}}f_P^{(\kappa + \rho)}(u+vq^{\kappa})\overline{f_P^{(\kappa + \rho)}(vq^{\kappa})}e\left(- \frac{ut}{q^{\mu+\nu}}\right)
\\ \label{G2}& \quad + \frac{1}{q^{\mu+\nu}}\ds\sum_{(u,v)\in \tilde{\mathcal{W}}_{\kappa}}f_P(vq^{\kappa})e\left(-\frac{vt}{q^{\mu+\nu-\kappa}}- \frac{ut}{q^{\mu+\nu}} \right)
\\ \notag & \qquad \qquad \cdot\left[f_P(u+vq^{\kappa})\overline{f_P(vq^{\kappa})}-f_P^{(\kappa + \rho)}(u+vq^{\kappa})\overline{f_P^{(\kappa + \rho)}(vq^{\kappa})}\right].
\end{align}
On rappelle qu'ici \begin{equation*}\label{ecriture troncation SI}f_P^{(\kappa + \rho)}(n) = e\Bigg(\alpha a_P \Big(n \bmod q^{\kappa + \rho},T_q(n) \Big) \Bigg).
\end{equation*}
Nous incorporons \eqref{G1}, qu'on nomme $G_{\kappa,1}(t)$, et \eqref{G2}, qu'on nomme $G_{\kappa,2}(t)$, dans $S_I'(\vartheta)$ avec la valeur $\kappa = \kappa_d$ d\'efinie par \begin{equation}\label{choix de kappa SI}
q^{\kappa_d-1}< M^2/d^2 \leq q^{\kappa_d}.
\end{equation} Nous avons alors que la suite $(\kappa_d)_{d \geq 0}$ est décroissante, et donc satisfait à $1 \leq \kappa_d \leq \kappa_1 \leq 2 \mu$, de sorte que par \eqref{Equa M SI}, \eqref{Condi kappa SI} est satisfaite. En posant $t = \vartheta - \frac{k}{m}q^{\mu+\nu}$, nous avons la majoration suivante  \begin{equation}\label{lien SI' et autres}
 S_I'(\vartheta) \leq  \ds\sum_{1\leq d \leq M}\ds\sum_{\frac{M}{q}< m \leq M}\frac{1}{m}\ds\sum_{\substack{0 \leq k < m \\ (k,m) = d}}\left|G_{\kappa_d,1}\left(\vartheta - \frac{k}{m}q^{\mu+\nu}\right)+G_{\kappa_d,2}\left(\vartheta - \frac{k}{m}q^{\mu+\nu}\right)\right|,
\end{equation} et dans la somme de droite, le terme en $G_{\kappa_d,1}$ est nommé $S_{I,1}'$, quant au second, nous le désignons $S_{I,2}'$.
%ainsi $S_I'$ est s\'epar\'ee en deux sommes : 
%\begin{equation}\label{SI1' partie 1}
%S_{I,1}'(\vartheta)=\ds\sum_{1\leq d \leq M}\ds\sum_{\frac{M}{q}< m \leq M}\frac{1}{m}\ds\sum_{\substack{0 \leq k < m \\ (k,m) = d}}\left|G_{\kappa_d,1}\left(\vartheta - \frac{k}{m}q^{\mu+\nu}\right)\right|
%\end{equation} et
%\begin{equation}\label{SI2' partie 1}
%S_{I,2}'(\vartheta):= \ds\sum_{1\leq d \leq M}\ds\sum_{\frac{M}{q}< m \leq M}\frac{1}{m}\ds\sum_{\substack{0 \leq k < m \\ (k,m) = d}}\left|G_{\kappa_d,2}\left(\vartheta - \frac{k}{m}q^{\mu+\nu}\right)\right|.
%\end{equation}

Contr\^olons \`a pr\'esent le terme d'erreur $S_{I,2}'(\vartheta)$. Supposons
 \begin{equation}\label{Condition sur rho SI}
P(\mu+\nu+1) \leq \rho \leq \frac{\mu+\nu}{4}.
\end{equation}
 %\begin{equation}\label{Condition sur rho SI}
%P(\mu+\nu+1) \leq \rho,
%\end{equation} et \begin{equation}\label{Autre condition sur rho SI}
%\rho \leq \frac{\mu+\nu}{4}.
%\end{equation}
Par \eqref{Condi kappa SI} et \eqref{Condition sur rho SI}, nous avons $\rho \leq \frac{3}{4}(\mu+\nu-\kappa_d)$ de sorte que le Lemme \ref{Lemme propa SI} s'applique avec $\lambda = \mu+\nu-\kappa_d$, $\kappa = \kappa_d$ et $\rho = \rho$. Il vient donc que le nombre de $0 \leq v < q^{\mu+\nu-\kappa_d}$ tels qu'il existe $0 \leq u < q^{\kappa_d}$ avec \begin{equation*}
f_P(u+vq^{\kappa_d})\overline{f_P(vq^{\kappa_d})} \neq f_P^{(\kappa_d+\rho)}(u+vq^{\kappa_d})\overline{f_P^{(\kappa_d+\rho)}(vq^{\kappa_d})}
\end{equation*} est $O(q^{\mu+\nu-\kappa_d-\rho+P(\mu+\nu+1)})$. En nommant $\mathcal{\widetilde{W}}_{\kappa_d}$ l'ensemble des couples $(u,v)$ pour lesquels cette inégalité est vérifiée, on a \begin{equation}\label{petite propa pour SI}
\#\mathcal{\widetilde{W}}_{\kappa_d} \ll q^{\mu+\nu-\rho+P(\mu+\nu+1)}.
\end{equation}
Il est possible, en divisant par $(k,m)$ de montrer que
\begin{align*}
 S_{I,2}'(\vartheta) &\leq \sum_{1 \leq d \leq M}\frac{S_{I,2}''(M,d)}{dq^{\mu+\nu}},
\end{align*} avec
\begin{align*}
S_{I,2}''(M,d) = \ds\sum_{\frac{M}{qd} < m' \leq \frac{M}{d}}\ds\sum_{\substack{0 \leq k' < m'\\ (k',m') = 1}}\left| \ds\sum_{0 \leq w < q^{\mu+\nu}}c_{\kappa_d, \rho_1}'(w)e\left(-\frac{w \vartheta'}{q^{\mu+\nu}}+\frac{wk'}{m'} \right) \right|
\end{align*} où $|c_{\kappa_d, \rho_1}'(w)| \leq 2$ et vaut $0$ si $w \notin \mathcal{W}_{\kappa_d} := \{u+vq^{\kappa_d}, (u,v) \in  \tilde{\mathcal{W}}_{\kappa_d}\}$.
Puis, par l'inégalité de Cauchy-Schwarz et le Grand Crible, tout comme \cite{RudinShapiro}, nous avons en sommant sur tous les $w$ possibles :
\begin{align*}
\left|S_{I,2}''(M,d) \right|^2 & \ll (\log q) (q^{\mu+\nu} + M^2/d^2)\#\mathcal{\widetilde{W}}_{\kappa_d}.
\end{align*}
En réunissant ces deux équations avec \eqref{petite propa pour SI}, et le fait que $M^2 \leq q^{2 \mu} \leq q^{\mu+\nu}$ nous concluons :
\begin{equation}\label{Majoration finale SI2}
S_{I,2}'(\vartheta') \ll (\log q)^{1/2}\sum_{1 \leq d \leq M}\frac{1}{d}\, q^{-\rho/2+P(\mu+\nu+1)/2}\ll \mu (\log q)^{3/2}q^{-\rho/2+P(\mu+\nu+1)/2}.
\end{equation}

%Contr\^olons \`a pr\'esent le terme principal $S_{I,1}'(\vartheta)$. Pour ce faire nous rappelons l'\'el\'ement principal consituant cette somme est donné par \eqref{G1} avec $\kappa = \kappa_d$.

%Le fait d'avoir introduit la fonction tronqu\'ee nous facilite la tâche : \eqref{G_1 ligne 2} est de ce fait moins li\'ee \`a \eqref{G_1 ligne 1} (qui correspond \`a la transform\'ee de Fourier discr\`ete). Afin d'expliciter le lien entre ces \'equations, Mauduit et Rivat introduisent le reste de la division euclidienne de $v$ par $q^{\rho}$ de mani\`ere \`a rendre les deux lignes ind\'ependantes. Nous ne pouvons pas le faire directement : notre d\'efinition de $f_P^{(\kappa + \rho)}$ ne correspond pas exactement \`a la fonction bas\'ee sur les restes modulo $q^{\kappa + \rho}$, mais \`a \eqref{Equa ap mod rho}. Ainsi, si $w = v \bmod q^{\rho}$, on a clairement $T_q(w) \neq T_q(v)$ si $\rho < T_q(v)$. 
Pour estimer $S_{I,1}'(\vartheta ')$, il convient de séparer la partie multiplicative de la partie digitale de la question. Notre fonction $f_P^{(\rho)}$ ne correspond pas exactement à la fonction tronquée, telle que l'utilisent Mauduit et Rivat, leur technique ne s'applique pas ici.

Pour contourner ce probl\`eme on note $l = T_q(v)$ et on pose $0 \leq u < q^{\kappa_d}$. Alors on a, toujours si $w = v \bmod q^{\rho}$ : 

\smallskip

  \[
\left \{
\begin{array}{c  c}
    T_q(u+wq^{\kappa_d}+q^{\kappa_d+\rho}\lfloor q^{l-\rho}\rfloor ) & = T_q(vq^{\kappa_d}) \\
    \epsilon_i ((u+wq^{\kappa_d}+q^{\kappa_d+\rho}\lfloor q^{l-\rho}\rfloor) \bmod (q^{\kappa_d+\rho})) & = \epsilon_i (u+vq^{\kappa_d}) \,  , \\
\end{array}
\right.
\] et ce pour tout $0 \leq i < \kappa_d + \rho. $

\smallskip

En effet, si $\rho > l$, alors $w=v$ et $\lfloor q^{l-\rho}\rfloor  = 0$ et donc $u+wq^{\kappa_d}+q^{\kappa_d+\rho}\lfloor q^{l-\rho}\rfloor = u+vq^{\kappa_d}$. Inversement, si $\rho \leq l$, comme $0 \leq w \leq q^{\rho}-1$ et $0 \leq u \leq q^{\kappa_d}-1$, nous avons $u+wq^{\kappa_d} \leq q^{\kappa_d}-1+q^{\kappa_d+\rho}-q^{\kappa_d} = q^{\kappa_d+\rho}-1 $ et donc $T_q(u+wq^{\kappa_d}+q^{\kappa_d+\rho}\lfloor q^{l-\rho}\rfloor ) = T_q(q^{l+\kappa_d}) = l+\kappa_d = T_q(vq^{\kappa_d}) = T_q(u+vq^{\kappa_d}),$ et comme $q^{l + \kappa_d} \equiv 0  \bmod (q^{\rho+\kappa_d})$, on a la seconde ligne.

Avec ces considérations nous avons :

\begin{align}
\notag & G_{\kappa_d,1}(t) \\ \notag & = \frac{1}{q^{\mu+\nu-\kappa_d}}\ds\sum_{w < q^{\rho}}\ds\sum_{0 \leq l < \mu + \nu - \kappa_d}
 \ds\sum_{q^{l-1} \leq v < q^l}f_P(vq^{\kappa_d})e\left(-\frac{vt}{q^{\mu+\nu-\kappa_d}} \right)\frac{1}{q^{\rho}}\ds\sum_{h < q^{\rho}}e\left(h\frac{v-w}{q^{\rho}} \right)
\\ \notag & \quad \cdot \frac{1}{q^{\kappa_d}}\ds\sum_{u < q^{\kappa_d}}f_P^{(\kappa_d+\rho)}(u+wq^{\kappa_d}+q^{\kappa_d+\rho}\lfloor q^{l-\rho}\rfloor )\overline{f_P^{(\kappa_d+\rho)}(wq^{\kappa_d}+q^{\kappa_d+\rho}\lfloor q^{l-\rho}\rfloor)}e\left(- \frac{ut}{q^{\mu+\nu}}\right).
\\ \label{G2 transformee} & = \ds\sum_{0 \leq l < \mu + \nu - \kappa_d}\ds\sum_{h < q^{\rho}} \left[\frac{1}{q^{\kappa_d}}\ds\sum_{u < q^{\kappa_d}}c_{\kappa_d,\rho,l}(u,h)e\left(-\frac{ut}{q^{\mu+\nu}} \right) \right]
\\ \notag & \qquad \qquad \quad \cdot\left[\frac{1}{q^{\mu + \nu - \kappa_d}}\ds\sum_{q^{l-1} \leq v < q^l}f_P(vq^{\kappa_d})e\left(-\frac{vt}{q^{\mu + \nu - \kappa_d}}+\frac{hv}{q^\rho} \right) \right],
\end{align}
avec \begin{align*}
&c_{\kappa_d,\rho,l}(u,h) \\ &  = \frac{1}{q^{\rho}}\ds\sum_{w < q^{\rho}} f_P^{(\kappa_d+\rho)}(u+wq^{\kappa_d}+q^{\kappa_d+\rho}\lfloor q^l/q^\rho \rfloor)\overline{f_P^{(\kappa_d+\rho)}(wq^{\kappa_d}+q^{\kappa_d+\rho}\lfloor q^l/q^\rho \rfloor)}e\left(- \frac{hw}{q^{\rho}}\right).
\end{align*} 

%Il reste d\'esormais \`a isoler d\'efinitivement \eqref{G2 transformee} de \eqref{G2T}. On remarque d'abord que dans \eqref{G2 transformee}, certains $l$ doivent \^etre trait\'e \`a part. Si $P(l+\kappa_d) > l$ et $n$ tel que $T_q(n) = l$, alors $\epsilon_{i+P(l+\kappa_d)}(n)\ldots\epsilon_i(n) = 0$ car $l = T_q(n)$ est l'indice du dernier chiffre non nul de $n$, et la somme compos\'ee des $e(\alpha a(n))$ est alors triviale.

Soit $$M_d = \ds\max\{0 \leq l \leq \mu+\nu-\kappa_d : P(l+\kappa_d) > l\}.$$ Nous s\'eparons la somme dans \eqref{G2 transformee} selon que $l \leq M_d$ not\'ee $G_{\kappa_d,1}'(t)$) ou que $l > M_d$ not\'ee $G_{\kappa_d,1}''(t)$. 
En majorant la seconde somme trivialement dans $G_{\kappa_d,1}'(t)$, nous avons \begin{align*}
|G_{\kappa_d,1}'(t)| &\leq \ds\sum_{0 \leq l \leq M_d}q^{l-(\mu+\nu-\kappa_d)}\ds\sum_{h < q^{\rho}} \frac{1}{q^{\kappa_d}}\left|\ds\sum_{u < q^{\kappa_d}}c_{\kappa_d,\rho,l}(u,h)e\left(-\frac{ut}{q^{\mu+\nu}} \right) \right|
\\ & \ll q^{M_d-(\mu+\nu)}\ds\sum_{0 \leq l \leq M_d}\ds\sum_{h < q^{\rho}} \left|\ds\sum_{u < q^{\kappa_d}}c_{\kappa_d,\rho,l}(u,h)e\left(-\frac{ut}{q^{\mu+\nu}} \right) \right|,
\end{align*}
et par le Lemme \ref{Lemme pour la somme de type I}, nous obtenons \begin{align}
\label{SI deuxieme terme d'erreur, intermediaire}\ds\sum_{1 \leq d \leq M}&\frac{1}{d}\ds\sum_{\frac{M}{qd}\leq m' < \frac{M}{d}}\frac{1}{m'}\ds\sum_{\substack{0 \leq k' < m' \\ (k',m')=1}}\left|G_{\kappa_d,1}'\left(\vartheta'-\frac{k}{m}q^{\mu+\nu}\right)\right|
  \ll  \frac{\log q}{q^{\mu+\nu}} \ds\sum_{1 \leq d \leq M}\frac{1}{d}\, q^{M_d+ \kappa_d+ \rho/2}\, M_d.
\end{align}

Cependant, comme la fonction $P$ est croissante et que nous avons déjà remarqué que $1 \leq \kappa_d \leq \kappa_1 \leq \frac{2}{3}(\mu+\nu)$, nous avons \begin{align*}
M_d \leq  \ds\max\left\{0 \leq l \leq \mu+\nu-\kappa_d : P\left(l+\frac{2}{3}(\mu+\nu)\right) > l \right\}.
\end{align*}
Supposons maintenant que $P(x) \leq x/10$, ceci implique \begin{equation*}
\frac{1}{10}\left(l + \frac{2}{3}(\mu+\nu)\right) > l, \text{ c'est-\`a-dire : } l < \frac{2}{27}(\mu+\nu),
\end{equation*} et donc $M_d \leq \frac{2}{27}(\mu+\nu)$. Comme $\kappa_d \leq \frac{2}{3}(\mu+\nu)$, nous obtenons quel que soit $d$ : %$$M_d + \kappa_d \leq \frac{20}{27}(\mu+\nu), $$ et donc 
\begin{equation}\label{SI deuxieme terme d'erreur, final}
\ds\sum_{1 \leq d \leq M}\frac{1}{d}\ds\sum_{\frac{M}{qd}\leq m' < \frac{M}{d}}\frac{1}{m'}\ds\sum_{\substack{0 \leq k' < m' \\ (k',m')=1}}\left|G_{\kappa_d,1}'\left(\vartheta'-\frac{k}{m}q^{\mu+\nu}\right)\right| \ll (\mu+\nu)^2 (\log q)^2q^{-\frac{7}{27}(\mu+\nu)+\rho/2}.
\end{equation}

\`A pr\'esent contr\^olons le terme restant. %Rappelons que la somme \`a contr\^oler est essentiellement li\'ee \`a \begin{align*}
%G_{\kappa_d,1}''(t)& = \ds\sum_{M_d < l \leq \mu + \nu - \kappa_d}\ds\sum_{h < q^{\rho}} \left[\frac{1}{q^{\kappa_d}}\ds\sum_{u < q^{\kappa_d}}c_{\kappa_d,\rho,l}(u,h)e\left(-\frac{ut}{q^{\mu+\nu}} \right) \right]
%\\ \notag & \quad \quad \cdot \left[\frac{1}{q^{\mu + \nu - \kappa_d}}\ds\sum_{q^{l-1} \leq v < q^l}f_P(vq^{\kappa_d})e\left(-\frac{vt}{q^{\mu + \nu - \kappa_d}}+\frac{hv}{q^\rho} \right) \right].
%\end{align*} 
Comme $l > M_d$, nous pouvons \`a pr\'esent utiliser le Lemme \ref{Lemme premiere estimation Fourier} ce qui nous donne : 
\begin{align*}
|G_{\kappa_d,1}''(t)| & \leq \ds\sum_{M_d < l \leq \mu + \nu - \kappa_d}q^{\gamma(l,\kappa_d)-(\mu + \nu - \kappa_d)}\ds\sum_{h < q^{\rho}} \frac{1}{q^{\kappa_d}}\left|\ds\sum_{u < q^{\kappa_d}}c_{\kappa_d,\rho,l}(u,h)e\left(-\frac{ut}{q^{\mu+\nu}} \right) \right|. 
\end{align*}
Soit $$h(l,k) := \frac{\log \left(q^{P(\kappa+l)}-8\left(\sin  \frac{\pi |\!|\alpha|\!|}{4} \right)^2\right)}{P(\kappa+l)\log q}.$$ On a alors $$\gamma(l,\kappa_d) := lh(l,\kappa_d) = lh(l+\kappa_d,0) \leq lh(\mu+\nu,0) \leq (\mu+\nu-\kappa_d)h(\mu+\nu,0),$$ car quelque soit $k$, l'application $l \mapsto lh(k,0)$ est croissante quitte \`a supposer $P(0) \geq 4$. 
%Soit \`a pr\'esent $h(l,k) = \gamma(l,k)/l$. En nous rappelant la Proposition \ref{Croissance des fonctions},  et en remarquant que quelque soit $k$, l'application $l \mapsto lh(k,0)$ est croissante (quitte \`a supposer $P(0) \geq 4$), nous avons :\begin{align*}
%lh(\mu+\nu-\kappa_d,\kappa_d) - (\mu+\nu-\kappa_d) & = lh(\mu+\nu,0)-(\mu+\nu-\kappa_d)
%\end{align*}
Mais $m\left(1-h(k,0)\right)  = \gamma_P(m,k)$ %\begin{align*}
 %m\left(1-h(k,0)\right) &= m\left(1-\ds\frac{\log \left(q^{P(k)}-8 \sin \left(\frac{\pi|\!|\alpha|\!|}{4}\right)^2\right)}{P(k)\log q}\right)
% = \gamma_P(m,k).
%\end{align*}  %En effet, comme $l \leq \mu+\nu-\kappa_d$, nous avons $lh(\mu+\nu,0)-(\mu + \nu - \kappa_d) \leq (\mu+\nu-\kappa_d)(h(\mu+\nu,0)-1)$,
 et donc :
\begin{align*}
|G_{\kappa_d,1}''(t)|  \leq q^{-\gamma_P(\mu + \nu - \kappa_d, \, \mu+\nu)}\ds\sum_{M_d < l \leq \mu + \nu - \kappa_d}\ds\sum_{h < q^{\rho}} \frac{1}{q^{\kappa_d}}\left|\ds\sum_{u < q^{\kappa_d}}c_{\kappa_d,\rho,l}(u,h)e\left(-\frac{ut}{q^{\mu+\nu}} \right)\right|.
\end{align*} 

L'application $m \mapsto \gamma_P(m,k)$ est croissante (lin\'eaire de coefficient directeur strictement positif), et donc l'application $k \mapsto -\gamma_P(l-k,r)$ l'est également. Ainsi par le Lemme \ref{Lemme pour la somme de type I}, comme $1 \leq \kappa_d \leq \frac{2}{3}(\mu+\nu)$ nous obtenons \begin{align}
\label{SI terme principal, intermediaire}\ds\sum_{1 \leq d \leq M}&\frac{1}{d}\ds\sum_{\frac{M}{qd}\leq m' < \frac{M}{d}}\frac{1}{m'}\ds\sum_{\substack{0 \leq k' < m' \\ (k',m')=1}}\left|G_{\kappa_d,1}''\left(\vartheta'-\frac{k}{m}q^{\mu+\nu}\right)\right|
\\ \notag & \leq  \ds\sum_{1 \leq d \leq M}\frac{q^{-\gamma_P(\mu + \nu - \kappa_d, \, \mu+\nu)}}{dq^{\kappa_d}}
\\ \notag & \quad \cdot \ds\sum_{M_d < l \leq \mu+\nu-\kappa_d}\ds\sum_{\frac{M}{qd}\leq m' < \frac{M}{d}}\frac{1}{m'}\ds\sum_{\substack{0 \leq k' < m' \\ (k',m')=1}}\ds\sum_{h < q^{\rho}} \left|\ds\sum_{u < q^{\kappa_d}}c_{\kappa_d,\rho,l}(u,h)e\left(-\frac{ut}{q^{\mu+\nu}} \right) \right|
\\ \notag & \ll q^{-\gamma_P(\frac{1}{3}(\mu + \nu), \, \mu+\nu)}\ds\sum_{1 \leq d \leq M}\frac{q^{\rho/2}}{d} (\mu+\nu) \log q ,
\end{align} ce qui nous donne \begin{equation}\label{SI terme principal, final}
\eqref{SI terme principal, intermediaire} \ll q^{\rho/2-\gamma_P(\frac{1}{3}(\mu + \nu), \, \mu+\nu)} (\mu+\nu)^2 (\log q)^2.
\end{equation}

Nous avons donc par les \'equations \eqref{lien SI et sommes sous jacentes}, \eqref{lien SI' et autres},  \eqref{Majoration finale SI2}, \eqref{SI deuxieme terme d'erreur, final} et \eqref{SI terme principal, final} : 

\begin{align*}
S_I(\vartheta) \ll & (\log q^{\mu+\nu})\left( \mu (\log q)^{3/2}q^{\mu+\nu-\rho/2+P(\mu+\nu+1)/2} \right. \\ & \left. + (\mu+\nu)^2 (\log q)^2q^{\frac{20}{27}(\mu+\nu)+\rho/2} \right. \\ & \left. + q^{\mu+\nu+\rho/2-\gamma_P(\frac{1}{3}(\mu + \nu), \, \mu+\nu)} (\mu+\nu)^2 (\log q)^2  \right).
\end{align*}

Nous choisissons $\rho = \frac{P(\mu+\nu+1)}{2}+ \gamma_P(\frac{1}{3}(\mu + \nu), \, \mu+\nu) $, et par \eqref{Condition sur P SI} et la Remarque \ref{Remarque sur la fonction gammaP},  %$\gamma_P(\frac{1}{3}(\mu + \nu), \, \mu+\nu) \leq \frac{\mu+\nu}{6}$ et ainsi, comme $\rho \leq \frac{3}{2}\gamma_P(\frac{1}{3}(\mu + \nu), \, \mu+\nu)$,
 \eqref{Condition sur rho SI} est v\'erifi\'ee. Le choix de $\rho$ fournit %En effet, par \eqref{Condition sur P SI} : \begin{align*}
%P(\mu+\nu+1) \leq \gamma_P\left(\frac{1}{3}(\mu+\nu),\mu+\nu\right) \leq \rho.
%\end{align*}

 \begin{align*}
S_I(\vartheta) & \ll q^{\mu+\nu}(\mu+\nu)^3(\log q)^3\left(q^{\frac{P(\mu+\nu+1)}{4}- \frac{1}{2}\gamma_P(\frac{1}{3}(\mu + \nu), \, \mu+\nu)}+q^{\frac{3}{4}\gamma_P(\frac{1}{3}(\mu + \nu), \, \mu+\nu)-\frac{7}{27}(\mu+\nu)} \right)
\\ & \ll q^{\mu+\nu}(\mu+\nu)^3(\log q)^3\left(q^{- \frac{1}{4}\gamma_P(\frac{1}{3}(\mu + \nu), \, \mu+\nu)}+q^{\frac{3}{4}\gamma_P(\frac{1}{3}(\mu + \nu), \, \mu+\nu)-\frac{7}{27}(\mu+\nu)} \right).
\end{align*}
 De plus la Remarque \ref{Remarque sur la fonction gammaP} nous donne la majoration
 \begin{equation}\label{Equation Finale SI}
 S_I(\vartheta) \ll (\mu+\nu)^3(\log q)^3q^{\mu+\nu- \frac{1}{4}\gamma_P(\frac{1}{3}(\mu + \nu), \, \mu+\nu)}
 \end{equation}
 qui achève notre preuve.
\end{proof}

\section{Travail pr\'eparatoire pour les sommes de type II}\label{Sec Prepa II}

De m\^eme que pour les sommes de type I, le traitement des sommes de type II n\'ecessite un certain nombre de r\'esultats pr\'eliminaires. La strat\'egie pour contr\^oler les sommes de type II consiste \`a introduire une fonction doublement tronqu\'ee et \`a exploiter sa structure pour mettre en exergue dans nos estimations la propri\'et\'e de Fourier de $e(\alpha a_P(\cdot))$. Les Lemmes \ref{Lemme Olivier Robert}, \ref{Lemme digital premiere troncature SII} et \ref{Lemme digital double troncature}, issus directement de la méthode de Mauduit et Rivat, permettent d'introduire la fonction doublement tronqu\'ee. Le Lemme \ref{Estimation Fourier SII} nous permet de faire appara\^itre la propri\'et\'e de Fourier dans l'estimation des sommes de type II. Pour conclure la Partie \ref{Section SII}, nous serons amen\'es \`a \'evaluer des sommes d'exponentielles pond\'er\'ees de deux types diff\'erents. Les Lemmes \ref{Equa terme principal somme mn SII} et \ref{h2+h0 diff 0} sont les estimations des sommes d'exponentielles que nous serons amen\'es \`a regarder par la suite.

%\begin{lemma}\label{Lemme pour sommes expos}
%Soit $f$ une fonction r\'eelle, $(b_n)_{n \geq 0}$ une suite r\'eelle, $\mu$ et $q$ deux entiers sup\'erieurs ou \'egaux \`a $2$, $q^{\mu-1} \leq M < q^\mu$ et $I_1 \subseteq [M/q,M[$ un intervalle. Alors \begin{align*}
 %\left|\ds\sum_{m \in I_1}e(f(m))b_m\right| \ll \max  & \left(|b_{\max I_1+1}|\left|\ds\sum_{m \in I_1}e(f(m)) \right|, \right. \\& \left. \sup_{M/q < K \leq M}\left|\ds\sum_{l \in I_1\cap[M/q,K)}e(f(l)) \right|\ds\sum_{m \in I_1}\left|b_{m+1}-b_m \right|\right).
%\end{align*}
%\end{lemma}

%\begin{proof}
%Ceci résulte d'une sommation d'Abel.
%\end{proof}

\begin{lemma}\label{Equa terme principal somme mn SII}
Soient $\mu,\nu, M, N, q$ des entiers tels que $2 \mu \leq \nu $, $q^{\mu-1} \leq M < q^\mu$ et $q^{\nu-1} \leq N < q^\nu$ . Soient $k \in \{\mu+\nu-4, \ldots ,\mu+\nu-1\}$ et $M/q \leq m < M$ des entiers. D\'efinissons $$I(k,m) = \left\{\frac{N}{q} \leq n < N : \frac{q^{k}}{m} \leq n < \frac{q^{k+1}}{m}\right\},$$ alors, si $I_1 \subseteq [M/q,M[$ est un intervalle, nous avons :
\begin{equation*}
\left|\ds\sum_{m \in I_1}e\left(m\frac{h_1r}{q^{\mu_2}} \right)\#I(k,m)\right| \ll q^\nu \min\left(q^\mu, \left|\sin \pi \frac{h_1r}{q^{\mu_2}} \right|^{-1} \right).
\end{equation*}
\end{lemma}

\begin{proof}

Commen\c cons par remarquer que 
 \begin{align*}
\#I(k,m) = \left\{ \begin{array}{c c}
0 & \text{si} \quad m \leq \ds\frac{q^k}{N} \quad \text{ou} \quad m \geq \ds\frac{q^{k+2}}{N}
\\ \left \lceil \ds\frac{q^{k+1}}{m}\right \rceil- \left\lceil\ds\frac{N}{q}\right\rceil & \text{si} \quad \ds\frac{q^{k+1}}{N} \leq m < \ds\frac{q^{k+2}}{N}
\\ N- \left\lfloor \ds\frac{q^{k}}{m} \right\rfloor & \text{si} \quad \ds\frac{q^k}{N} \leq m < \ds\frac{q^{k+1}}{N} 
\end{array}. \right.
\end{align*}

%Nous allons d'abord effectuer la preuve du lemme dans le cas $(M,N) = (q^{\mu-1},q^{\nu-1})$ qui est un cas particulier et qui permet de bien comprendre la d\'emonstration : dans le cas g\'en\'eral, seules quelques difficult\'es techniques suppl\'ementaires interviennent.

Commençons par traiter le cas $(M,N) = (q^{\mu-1},q^{\nu-1})$. Sous ces conditions, nous avons $q^{\mu+\nu-4} \leq mn < q^{\mu+\nu-2}$. Ainsi pour $k \geq \mu+\nu-2$, les conditions imposent $\#I(k,m)$ nul ainsi que
\begin{equation}\label{Imu+nu-2}
\#I(\mu+\nu-4,m) =  \left\lceil \frac{q^{\mu+\nu-3}}{m} \right\rceil -q^{\nu-2}
\end{equation} et
\begin{equation}\label{Imu+nu-1}
\#I(\mu+\nu-3,m) =  \quad q^{\nu-1}- \left\lfloor \frac{q^{\mu+\nu-3}}{m} \right\rfloor .
\end{equation}%, car $(M,N) = (q^{\mu-1},q^{\nu-1})$.

La technique de majoration \'etant quasi identique, nous nous contenterons de l'expliciter pour $\#I(\mu+\nu-3,m)$. Une sommation d'Abel fournit \begin{align}\label{Somme Lemme techn 1 chap 2 SII, cas simple}
&\left|\ds\sum_{m \in I_1}e\left(m\frac{h_1r}{q^{\mu_2}} \right)\left |q^{\nu-1}- \left\lfloor\frac{q^{\mu+\nu-3}}{m}\right\rfloor \right| \right| 
\\ \notag & \ll \max \left(\left|q^{\nu-1}-\left\lfloor \ds\frac{q^{\mu+\nu-3}}{\max I_1 + 1} \right\rfloor \right|\left|\ds\sum_{m \in I_1}e\left(m\frac{h_1r}{q^{\mu_2}} \right) \right|, \right.
\\& \notag \left. \qquad \sup_{q^{\mu-2} < K \leq q^{\mu-1}}\left|\ds\sum_{m \in I_1 \cap [q^{\mu-2},K)}e\left(m\frac{h_1r}{q^{\mu_2}} \right) \right|\ds\sum_{m \in I_1}\left|\left\lfloor\frac{q^{\mu+\nu-3}}{m+1}\right\rfloor-\left\lfloor\frac{q^{\mu+\nu-3}}{m}\right\rfloor\right|\right).
\end{align}
Majorons d'abord le premier terme du max. Comme $I_1 \subseteq [q^{\mu-2},q^{\mu-1}[$, nous avons $$\left|q^{\nu-1}-\left\lfloor \ds\frac{q^{\mu+\nu-3}}{\max I_1 + 1} \right\rfloor \right| \ll q^\nu $$ et comme \begin{equation}\label{Equa somme expo sur m}
\left|\ds\sum_{m \in I_1}e(\alpha m) \right| \leq \min \left(|I_1|,\ds\frac{1}{|\sin \pi \alpha|}\right),
\end{equation} nous avons la majoration d\'esir\'ee car $|I_1| \ll q^\mu$. Pour le second terme, on remarque que la somme interne est téléscopique et devient de fait $O(q^\nu)$ %\begin{align*}\ds\sum_{m \in I_1}\left|\left\lfloor\frac{q^{\mu+\nu-3}}{m+1}\right\rfloor-\left\lfloor\frac{q^{\mu+\nu-3}}{m}\right\rfloor\right|
%&= \ds\sum_{m \in I_1}\left(\left\lfloor\frac{q^{\mu+\nu-3}}{m}\right\rfloor-\left\lfloor\frac{q^{\mu+\nu-3}}{m+1}\right\rfloor\right)
%\end{align*} qui est une somme t\'elescopique, donc 
%\begin{align*}
% \ds\sum_{m \in I_1}\left|\left\lfloor\frac{q^{\mu+\nu-3}}{m+1}\right\rfloor-\left\lfloor\frac{q^{\mu+\nu-3}}{m}\right\rfloor\right| &\leq \left\lfloor\frac{q^{\mu+\nu-3}}{\min I_1}\right\rfloor \ll q^\nu 
%\end{align*}  
car $I_1 \subseteq [q^{\mu-2},q^{\mu-1}[$ et nous concluons en utilisant \eqref{Equa somme expo sur m}.

\smallskip

Pour traiter le cas $(M,N) \neq (q^{\mu-1},q^{\nu-1})$, on remarque que si $m<q^{k+1}/N < m+1$, c'est à dire tel que $\#I(k,m)$ ne soit pas de la même forme que $\#I(k,m+1)$, alors $m = \left \lfloor q^{k+1}/N \right \rfloor$. On sépare le reste de la somme selon que $m < q^{k+1}/N$ ou $m\geq q^{k+1}/N$ et nous appliquons le procédé précédent à chaque sous-cas. Nous concluons par le fait que $\#I\left(k,\left\lfloor \frac{q^{k+1}}{N} \right\rfloor\right) \leq \#\{N/q \leq n < N\} \ll q^\nu$.

\end{proof}

\begin{lemma}\label{h2+h0 diff 0}
Soient $\mu,\nu$ des entiers tels que $\frac{1}{4}(\mu+\nu)\leq \mu \leq \nu \leq \frac{3}{4}(\mu+\nu)$. Soient $M$, $m$, $k$ et  $I(k,m)$  d\'efinis dans le Lemme \ref{Equa terme principal somme mn SII}. Alors si $h, h_1, \mu_0, \mu_1, \mu_2, s$ sont des entiers tels que $\mu_0 < \mu_1 < \mu_2 $ et $hq^{\mu_1-\mu_2}s \notin \Z$ , nous avons pour tout $I_1 \subseteq [q^{\mu-2},q^\mu[$ :
\begin{align*}
\left|\ds\sum_{m \in I_1}e\left(\frac{h_1rm}{q^{\mu_2}}\right)\ds\sum_{n \in I(k,m)}e(hsnq^{\mu_1-\mu_2}) \right|\ll \left(shq^{3(\mu_2-\mu_1)} \right)^{1/2}q^{\frac{7}{8}(\mu+\nu)}
\end{align*}
\end{lemma}

\begin{proof}

Soient $a_m$ et $b_m$ des entiers tels que $I(k,m) = [a_m,b_m[$ (les entiers d\'ependent de $k$, mais pour ne pas alourdir les notations nous ne le marquons pas). Comme $ hq^{\mu_1-\mu_2}s \notin \Z$, nous pouvons \'ecrire en explicitant la sommation sur $n$ :
%\begin{align*}
%\ds\sum_{n \in I(k,m)}e(hq^{\mu_1-\mu_2}sn) = e\left(hq^{\mu_1-\mu_2}s\frac{a_m+b_m}{2}\right)\frac{\sin \left(\pi hq^{\mu_1-\mu_2}s(b_m-a_m) \right)}{\sin \left(\pi hq^{\mu_1-\mu_2}s \right)}
%\end{align*} donc 
\begin{align*}
&\left|\ds\sum_{m \in I_1} e\left(\frac{h_1rm}{q^{\mu_1}}\right)\ds\sum_{n \in I(k,m)}e(hq^{\mu_1-\mu_2}sn) \right|
 \\ \notag & \quad \ll \ds\frac{1}{|\sin (\pi hq^{\mu_1-\mu_2}s) |}\left|\ds\sum_{m \in I_1}e\left(\frac{h_1rm}{q^{\mu_1}}\right)e\left(hq^{\mu_1-\mu_2}s\frac{a_m+b_m}{2}\right)\sin \left(\pi hq^{\mu_1-\mu_2}s(b_m-a_m) \right) \right|.
\end{align*} Rappelons que $|\sin (\pi x) | \geq 2|\!|x|\!|_{\Z} $. Comme $q^{\mu_1-\mu_2}hs \notin \Z$, nous pouvons \'ecrire $hs = kq^{\mu_2-\mu_1}+l$ avec $1 \leq l \leq q^{\mu_2-\mu_1}-1$, et donc $|\!|q^{\mu_1-\mu_2}hs |\!|_{\Z}\geq q^{\mu_1-\mu_2}$. Ainsi %\begin{align*}
%&\left|\ds\sum_{m \in I_1} e\left(\frac{h_1rm}{q^{\mu_1}}\right)\ds\sum_{n \in I(k,m)}e(hq^{\mu_1-\mu_2}sn) \right|
% \\ \notag & \quad \ll q^{\mu_2-\mu_1}\left|\ds\sum_{m \in I_1}e\left(\frac{h_1rm}{q^{\mu_1}}\right)e\left(hq^{\mu_1-\mu_2}s\frac{a_m+b_m}{2}\right)\sin \left(\pi hq^{\mu_1-\mu_2}s(b_m-a_m) \right) \right|.
%\end{align*}
\begin{align*}
|\sin (\pi hq^{\mu_1-\mu_2}s) |^{-1} \ll q^{\mu_1-\mu_2}.
\end{align*}
Quitte \`a d\'ecouper l'intervalle, du fait de la forme de $I(k,m)$ (comme le laisse comprendre \eqref{Imu+nu-2}), %comme dans la d\'emonstration du Lemme \ref{Equa terme principal somme mn SII}, 
la somme interne se ramène à estimer : %nous pouvons nous ramener \`a estimer 
\begin{align*}
\ds\sum_{m \in J_1}e\left(\frac{h_1rm}{q^{\mu_1}}+hq^{\mu_1-\mu_2}s\frac{q^{k}}{2m} \right)\sin \left(\pi hq^{\mu_1-\mu_2}s\frac{q^k}{m} \right),
\end{align*} avec $J_1 \subseteq I_1 \subseteq [q^{\mu-2},q^\mu)$.

Une sommation d'Abel fournit \begin{align*}
& \left|\ds\sum_{m \in J_1}e\left(\frac{h_1rm}{q^{\mu_1}}+hq^{\mu_1-\mu_2}s\frac{q^{k}}{2m} \right)\sin \left(\pi hq^{\mu_1-\mu_2}s\frac{q^k}{m} \right)\right|
\\ & \ll \max \left(\left|\ds\sum_{m \in J_1}e\left(\frac{h_1rm}{q^{\mu_1}}+hq^{\mu_1-\mu_2}s\frac{q^{k}}{2m} \right) \right| \right., \\ & \qquad  \sup_{M/q < K \leq M}\left|\ds\sum_{l \in J_1 \cap [M/q,K)}e\left(\frac{h_1rm}{q^{\mu_1}}+hq^{\mu_1-\mu_2}s\frac{q^{k}}{2m} \right) \right|
\\ & \qquad \left. \ds\sum_{m \in I_1}\left|\sin \left(\pi hq^{\mu_1-\mu_2}s\frac{q^k}{m+1} \right) -\sin \left(\pi hq^{\mu_1-\mu_2}s\frac{q^k}{m} \right)  \right|\right).
\end{align*} La fonction $$f(x):= \frac{h_1r}{q^{\mu_2}}x + \ds\frac{1}{2x}hsq^{\mu_1-\mu_2+k}$$ est $\mathcal{C}^{\infty}$ sur $[q^{\mu-2},q^\mu)$ et de plus sa d\'eriv\'ee seconde v\'erifie $$\lambda := \frac{1}{2} shq^{\mu_1-\mu_2+k-3\mu}\leq f''(x) = \ds\frac{shq^{\mu_1-\mu_2+k}}{2x^3}\leq shq^{\mu_1-\mu_2+k-3\mu+3} = \lambda q^3,  $$  nous pouvons alors utiliser \cite[Theorem $2.2$]{GrahamKolesnik} pour obtenir : \begin{align*}
\left|\ds\sum_{l \in J_1 \cap [q^{\mu-1},K)}e\left(\frac{h_1r}{q^{\mu_1}}l + \frac{1}{2l}hsq^{\mu_1-\mu_2+k} \right) \right| \ll q^{\mu}\left(shq^{\mu_1-\mu_2+k-3\mu}\right)^{1/2}+\left(shq^{\mu_1-\mu_2+k-3\mu}\right)^{-1/2}.
\end{align*} En exploitant le fait que $\mu+\nu-4 \leq k \leq \mu + \nu - 1$ dans l'estimation pr\'ec\'edente, nous obtenons : 
\begin{align*}
\left|\ds\sum_{l \in J_1 \cap [q^{\mu-1},K)}e\left(\frac{h_1r}{q^{\mu_1}}l + \frac{1}{2l}hsq^{\mu_1-\mu_2+k'} \right) \right| & \ll q^{\mu}\left(shq^{\mu_1-\mu_2+\nu-2\mu }\right)^{1/2}+\left(shq^{\mu_1-\mu_2+\nu-2\mu }\right)^{-1/2}
\\ & \ll \left(shq^{\mu_1-\mu_2} \right)^{1/2}q^{\nu/2}+ \left(shq^{\mu_1-\mu_2}\right)^{-1/2}q^{\mu-\nu/2}.
\end{align*} En majorant trivialement le terme de la somme sur les sinus, et en utilisant $J_1 \subseteq I_1 \subseteq [q^{\mu-2},q^\mu)$, et donc $|J_1| \ll q^\mu$ nous obtenons la majoration : \begin{align*}
&\notag \left|\ds\sum_{m \in J_1}e\left(\frac{h_1rm}{q^{\mu_1}}+hq^{\mu_1-\mu_2}s\frac{q^{k}}{2m} \right)\sin \left(\pi hq^{\mu_1-\mu_2}s\frac{q^k}{2m} \right)\right|
\\ \label{Equa temporaire lemme sinus} & \qquad \ll \ \left(shq^{\mu_2-\mu_1} \right)^{1/2}q^{\mu+\nu/2}+ \left(shq^{\mu_1-\mu_2}\right)^{-1/2}q^{2\mu-\nu/2}.
\end{align*} Cependant, nous avons $\frac{1}{4}(\mu+\nu)\leq \mu \leq \nu \leq \frac{3}{4}(\mu+\nu) $, ce qui donne %$$2\mu-\frac{\nu}{2} \leq \mu + \nu-\frac{\nu}{2} \leq (\mu+\nu)-\frac{\mu+\nu}{8} = \frac{7}{8}(\mu+\nu) $$ ainsi que $$\mu+\frac{\nu}{2}=\frac{\mu+\nu}{2}+\frac{\mu}{2}\leq \frac{\mu+\nu}{2}+\frac{3(\mu+\nu)}{8} = \frac{7}{8}(\mu+\nu)  $$ ce qui nous donne
 \begin{equation*}
\left|\ds\sum_{m \in J_1}e\left(\frac{h_1rm}{q^{\mu_1}}+hq^{\mu_1-\mu_2}s\frac{q^{k}}{2m} \right)\sin \left(\pi hq^{\mu_1-\mu_2}s\frac{q^k}{2m} \right)\right| \ll \left(shq^{\mu_2-\mu_1} \right)^{1/2}q^{\frac{7}{8}(\mu+\nu)},
\end{equation*} et le lemme est d\'emontr\'e.
\end{proof}

Le lemme ci dessous est le lemme crucial du pr\'esent article. Il dit en quelque sorte que si on ne perturbe pas trop des entiers $m$ et $n$ en $m'$ et $n'$, le produit $m'n'$ aura la même taille que le produit $mn$. Ceci nous permet d'introduire dans $S_{II}$ la notion essentielle de fonction doublement tronqu\'ee. Les explications de la n\'ecessit\'e d'un tel r\'esultat se trouvent dans la Partie \ref{Section SII}.

\begin{lemma}\label{Lemme Olivier Robert}
Soient $\mu,\nu,\rho$ des entiers tels que $2\rho \leq \nu - 1 $. Soient $q^{\mu-2} \leq m < q^\mu$ et $q^{\nu-2} \leq n < q^\nu$ des entiers. Posons $m' = m+q^{\mu -\rho}$ et $n' = n+q^{\rho}$. 
 
 Alors :
 $$\#\{q^{\mu-1}\leq m < q^{\mu},q^{\nu-1}\leq n < q^{\nu} : T_q(mn) \neq T_q(m'n')\} \ll \log(q^{\mu+\nu})q^{\mu+\nu-\rho}.$$ 
\end{lemma}

\begin{proof}
Observons que $m'n' = mn + mq^\rho + nq^{\mu-\rho} + q^\mu$. Par ailleurs \begin{align*}
mq^\rho + nq^{\mu-\rho} + q^\mu &< q^{\mu+\rho}+nq^{\mu-\rho}+q^\mu
\\ & \leq q^{\mu+\rho+1}+nq^{\mu-\rho}
\\ &  \leq q^{\mu+\nu+1-\rho} .
\end{align*} 

Comme $mn \geq q^{\mu+\nu-4}$, pour que la taille de $mn$ soit perturb\'ee par l'ajout d'un terme inf\'erieur \`a $q^{\mu+\nu+1-\rho}$, il faut que les chiffres de $mn$ d'indice entre $\mu+\nu+1-\rho$ et $\mu+\nu-4$ soient \'egaux \`a $q-1$. Soit $\mathcal{M}$ le nombre des $(m,n)$ tels que $T_q(mn) \neq T_q(m'n')$. Si \[
\chi(a) = \left \{
\begin{array}{c c}
    1& \quad \text{si} \quad \epsilon_i(a) = q-1, \quad \mu+\nu+1-\rho \leq i \leq \mu+\nu-4 \\
    0& \qquad \text{sinon}, \\
\end{array}
\right.
\] alors \begin{align*}
\mathcal{M} &\leq \ds\sum_{q^{\mu+\nu+1-\rho} \leq a < q^{\mu+\nu}}\tau(a)\chi(a)
\\ & \leq \ds\sum_{b < q^{\mu+\nu+1-\rho}}\ds\sum_{c < q^4}\tau(b+(q-1)q^{\mu+\nu+1-\rho}+\ldots+(q-1)q^{\mu+\nu-4}+q^{\mu+\nu-3}c).
\end{align*}

Nous pouvons alors appliquer \cite[Lemma $3.5$]{PrimesDigits}, avec $x=q^{\mu+\mu+1-\rho}-1+(q-1)q^{\mu+\nu+1-\rho}+\ldots+(q-1)q^{\mu+\nu-2}+q^{\mu+\nu-1}c \leq q^{\mu+\nu+1}$ et $y=q^{\mu+\nu+1-\rho}$ pour pouvoir dire \begin{align*}
\mathcal{M} \ll q^{\mu+\nu-\rho}\log q^{\mu+\nu}.
\end{align*}
\end{proof}

\begin{lemma}\label{Lemme digital premiere troncature SII}
Soient $(\mu,\nu,\rho) \in \N^3$ avec $2\rho < \nu$ et tels que $P(\mu+\nu) \leq \rho$. L'ensemble $\mathcal{E}$ des couples $(m,n) \in \{q^{\mu-2},\ldots, q^{\mu}-1\}\times \{q^{\nu-2},\ldots,q^{\nu}-1\}$ tels qu'il existe $k < q^{\mu+\rho}$ avec $f_P(mn+k)\overline{f_P(mn)}\neq f_P^{(\mu+2\rho)}(mn+k)\overline{f_P^{(\mu+2\rho)}(mn)} $ satisfait \`a \begin{equation*}
\# \mathcal{E} \ll (\log q)q^{\mu+\nu-\rho+P(\mu+\nu+1)}.
\end{equation*}
\end{lemma}

\begin{proof}
Soit $\mathcal{B}$ l'ensemble des $l < q^{\nu-\rho}$ tels qu'il existe $(k_1,k_2) \in \{0, \ldots, q^{\mu+\rho}-1\}^2 $ avec $$a_P(lq^{\kappa}+k_1+k_2)- a_P(lq^{\kappa}+k_1) \neq a_P^{(\kappa + \rho)}(lq^{\kappa}+k_1+k_2) - a_P^{(\kappa + \rho)}(lq^{\kappa}+k_1).$$ Alors nous utilisons le Lemme \ref{Lemme propa SI} avec $\lambda = \nu-\rho$, $\kappa = \mu+\rho$ et $\rho = \rho$ pour obtenir $\#\mathcal{B} \ll  q^{\nu-2\rho+P(\mu+\nu+1)}$.

 Il suffit donc de compter le nombre $\mathcal{N}$ de couples $(m,n) \in \{q^{\mu-2},\ldots, q^{\mu}-1\}\times \{q^{\nu-2},\ldots,q^{\nu}-1\}$ tels que $mn = k_1 + lq^{\mu+\rho} $ avec $l \in \mathcal{B}$, le $k$ de notre énoncé ici correspondant alors à $k_2$. En utilisant \cite[Lemma $7$]{RudinShapiro} avec $\mu' = \mu+\rho$, nous obtenons  \begin{equation*}
\# \mathcal{E} \ll (q^{\mu+\rho}\log q + q^{\mu}+q^{\mu+\rho-\mu+2})\#\mathcal{B} \ll (\log q) q^{\mu+\nu-\rho+P(\mu+\nu+1)}.
\end{equation*}
\end{proof}

%Lorsque nous tronquons deux fois la fonction $f_P$, des probl\`emes techniques peuvent surgir. Pour les pallier, il nous faut introduire une fenêtre de s\'ecurit\'e qui grandit selon la taille des entiers regard\'es. Le but de ce lemme est de mesurer l'erreur produite en regardant cette fenêtre. Nous notons qu'il ressemble tr\`es fortement au \cite[Lemma $9$]{RudinShapiro}, et la preuve en est d'ailleurs similaire. Seulement les objets sp\'ecifiques \`a cet article imposent une nouvelle d\'emonstration.

\begin{lemma}\label{Lemme digital double troncature}
Soient $(\mu,\nu,\mu_0,\mu_1,\mu_2)\in \N^5$ avec $\mu_0\leq \mu_1 \leq \mu \leq \mu_2 $, $\mu_1-\mu_0 \leq \frac{3}{4}(\mu_2-\mu_0)$ et $2(\mu_2-\mu)\leq \mu_0 $. Supposons \'egalement que $P(\mu_2) \leq \mu_1-\mu_0$. Alors, pour $(a,b,c) \in \N^3 $ l'ensemble $\mathcal{E}(a,b,c)$ des couples $(m,n) \in \{q^{\mu-2},\ldots,q^{\mu}-1\}\times\{q^{\nu-2},\ldots,q^{\nu}-1\}$ tels que \begin{align}
\notag &f_P^{(\mu_2)}\Big(mn+am+bn+c,T_q(mn+am+bn+c)\Big)\\ \notag & \qquad \overline{f_P^{(\mu_1)}\Big(mn+am+bn+c,T_q(mn+am+bn+c)\Big)} 
\\ \label{Equation double troncature lemme} & \neq f_P^{(\mu_2)}\Big(q^{\mu_0}r_{\mu_0,\mu_2}(mn+am+bn+c),T_q(mn+am+bn+c)\Big)\\ \notag & \qquad \overline{f_P^{(\mu_1)}\Big(q^{\mu_0}r_{\mu_0,\mu_2}(mn+am+bn+c),T_q(mn+am+bn+c)\Big)} 
\end{align}  v\'erifie \begin{equation}\label{Equation Lemme digital double troncature}
\# \mathcal{E}(a,b,c) \ll \max (\tau(q),\log q)\mu_2^{\omega(q)}q^{\mu+\nu+\mu_0-\mu_1+P(\mu_2+1)}.
\end{equation}
\end{lemma}

\begin{proof}
Soit $\mathcal{B}$ l'ensemble des $l \in \{0,\ldots,q^{\mu_2-\mu_0}-1\}$ tel qu'il existe $(k_1,k_2) \in \{0,\ldots,q^{\mu_0}-1\}^2$ avec $$f_P^{(\mu_2)}(q^{\mu_0}l+k_1+k_2)\overline{f_P^{(\mu_2)}(q^{\mu_0}l+k_1)} \neq f_P^{(\mu_1)}(q^{\mu_0}l+k_1+k_2)\overline{f_P^{(\mu_1)}(q^{\mu_0}l+k_1)}. $$

 Pour $0 \leq l \leq q^{\mu_2-\mu_0}-2$, nous avons $0 \leq q^{\mu_0}l+k_1+k_2 \leq q^{\mu_2}-2 $ et ainsi $$f_P^{(\mu_2)}(q^{\mu_0}l+k_1+k_2)\overline{f_P^{(\mu_2)}(q^{\mu_0}l+k_1)} = f_P(q^{\mu_0}l+k_1+k_2)\overline{f_P(q^{\mu_0}l+k_1)}$$ sauf possiblement si $l = q^{\mu_2-\mu_0}-1$. Comme $\mu_1-\mu_0 \leq 3/4(\mu_2-\mu_0)$, nous pouvons utiliser le Lemme \ref{Lemme propa SI} avec $\lambda =\mu_2-\mu_0$, $\kappa = \mu_0$, $\rho = \mu_1-\mu_0$, et donc \begin{equation}\label{Equation card B}
\# \mathcal{B} = O(q^{\mu_2-\mu_0-(\mu_1-\mu_0)+P(\mu_2-\mu_0+\mu_0+1)})=O(q^{\mu_2-\mu_1+P(\mu_2+1)}).
\end{equation}  Le reste de la preuve est identique à \cite[Lemma $9$]{RudinShapiro}.

\end{proof}

 L'objet du lemme suivant est de dire que la fonction doublement tronqu\'ee a une transform\'ee de Fourier qui d\'ecro\^it. Pour l'expliciter nous introduisons un objet interm\'ediaire.

Soit $G_{\mu_0,\lambda}(t,k)= \ds\frac{1}{q^{\lambda}}\ds\sum_{0 \leq u < q^{\lambda}}f_P^{(\mu_1,\mu_2)}(uq^{\mu_0},k)e\left(-\frac{ut}{q^{\lambda}} \right)$.

\begin{lemma}\label{Estimation Fourier SII}
Uniform\'ement pour $\lambda \in \N$ tel que \begin{equation}\label{Encadrement lambda}
\frac{1}{3}(\mu_2-\mu_0)\leq \lambda \leq \frac{4}{5}(\mu_2-\mu_0),
\end{equation} et tout entier $k$ tel que \begin{equation}\label{Equation P \'enonc\'e}
 P(k) \leq \frac{1}{3}(\mu_1-\mu_0)
\end{equation} et $t \in \R$, nous avons : \begin{equation}
\ds\sum_{0 \leq h < q^{\mu_2-\mu_0-\lambda}}|G_{\mu_0,\mu_2-\mu_0}(h+t,k)|^2 \ll q^{1/2(\mu_1-\mu_0-\gamma_P(\lambda,k))+3P(k)/4}(\log q^{\mu_2-\mu_1})^2.
\end{equation}
\end{lemma}
%La preuve est similaire \`a \cite{RudinShapiro}, mais les modifications des \'el\'ements entrant en compte dans la d\'emonstration imposent de reprendre la preuve. Le fil de la preuve consiste \`a se ramener \`a une \'ecriture de la forme de \eqref{G_1 ligne 1}.
\begin{proof}
De même que \cite[Lemma $11$]{RudinShapiro}, en observant que les conditions imposées permettent de se passer de la troncature en $\mu_2$ et ainsi d'en introduire une autre, nous pouvons écrire :

%Pour commencer modifions l'\'ecriture de $G_{\mu_0,\mu_2-\mu_0}(h+t,k)$. Pour $0 \leq \lambda \leq \mu_2-\mu_0$ et $t\in \R$, nous pouvons \'ecrire \begin{equation*}
%G_{\mu_0,\mu_2-\mu_0}(t,k) = \frac{1}{q^{\mu_2-\mu_0}}\ds\sum_{0 \leq u < q^{\lambda}}\ds\sum_{0 \leq v < q^{\mu_2-\mu_0-\lambda}}f_P^{(\mu_1,\mu_2)}(q^{\mu_0}(u+vq^\lambda),k)e\left(- \frac{(u+vq^\lambda)t}{q^{\mu_2-\mu_0}} \right).
%\end{equation*}
%Mais si $\mu_1-\mu_0 \leq \lambda \leq \mu_2-\mu_0$, alors $0 \leq u+vq^{\lambda} < q^{\mu_2-\mu_0}$ et $(u+vq^\lambda)q^{\mu_0} \equiv uq^{\mu_0} \bmod q^{\mu_1} $, et pour $0 \leq u < q^\lambda$ et $0 \leq v < q^{\mu_2-\mu_0-\lambda}$, on a : \begin{align*}
%f_P^{(\mu_1,\mu_2)}(q^{\mu_0}(u+vq^{\lambda}),k) &= f_P^{(\mu_2)}(q^{\mu_0}(u+vq^{\lambda}),k)\overline{f_P^{(\mu_1)}(q^{\mu_0}(u+vq^{\lambda}),k)}
%\\ & = f_P(q^{\mu_0}(u+vq^{\lambda}),k)\overline{f_P^{(\mu_1)}(q^{\mu_0}u,k)}
%\end{align*}
%ce qui m\`ene \`a 
%\begin{align*}
% G_{\mu_0,\mu_2-\mu_0}&(t,k) =  \frac{1}{q^{\mu_2-\mu_0-\lambda}}\ds\sum_{0 \leq v < q^{\mu_2-\mu_0-\lambda}}f_P(vq^{\mu_0+\lambda},k)e\left(- \frac{vq^\lambda t}{q^{\mu_2-\mu_0}} \right)
%\\ & \quad \ds\frac{1}{q^\lambda}\ds\sum_{0 \leq u < q^{\lambda}}f_P(q^{\mu_0}(u+vq^{\lambda}),k)\overline{f_P(vq^{\mu_0+\lambda},k)f_P^{(\mu_1)}(uq^{\mu_0},k)}e\left(- \frac{ut}{q^{\mu_2-\mu_0}} \right).
%\end{align*}

%Introduisons une autre troncation. Pour cela, nous posons   %Ceci nous m\`ene \`a consid\'erer pour tout $t \in \R$ la d\'ecomposition suivante : 
$$G_{\mu_0,\mu_2-\mu_0}(t,k) = G_{\mu_0,\mu_2-\mu_0,\lambda,1}(t,k) + G_{\mu_0,\mu_2-\mu_0,\lambda,2}(t,k), $$ avec \begin{align*}
 &G_{\mu_0,\mu_2-\mu_0,\lambda,1}(t,k):= \frac{1}{q^{\mu_2-\mu_0-\lambda}}\ds\sum_{0 \leq v < q^{\mu_2-\mu_0-\lambda}}f_P(vq^{\mu_0+\lambda},k)e\left(- \frac{vq^\lambda t}{q^{\mu_2-\mu_0}} \right)
\\ & \quad \cdot \frac{1}{q^{\lambda}}\ds\sum_{0 \leq u < q^{\lambda}}f_P^{(\mu_0+\lambda+\rho_3)}(q^{\mu_0}(u+vq^{\lambda}),k)\overline{f_P^{(\mu_0+\lambda+\rho_3)}(vq^{\mu_0+\lambda},k)f_P^{(\mu_1)}(uq^{\mu_0},k)}e\left(- \frac{ut}{q^{\mu_2-\mu_0}} \right),
\end{align*}
qui est le terme principal, et 
\begin{align*}
 &G_{\mu_0,\mu_2-\mu_0,\lambda,2}(t,k):= \frac{1}{q^{\mu_2-\mu_0}}\ds\sum_{(u,v) \in \widetilde{\mathcal{W}_\lambda}}f_P(vq^{\mu_0+\lambda},k)f_P^{(\mu_1)}(uq^{\mu_0},k)e\left(- \frac{(u+vq^\lambda) t}{q^{\mu_2-\mu_0}} \right)
\\ & \quad \cdot \bigg(f_P(q^{\mu_0}(u+vq^{\lambda}),k)\overline{f_P(vq^{\mu_0+\lambda},k)}-f_P^{(\mu_0+\lambda+\rho_3)}(q^{\mu_0}(u+vq^{\lambda}),k)\overline{f_P^{(\mu_0+\lambda+\rho_3)}(vq^{\mu_0+\lambda},k)}\bigg),
\end{align*}
qui est le terme d'erreur, où \begin{equation}\label{equation rho3}
1 \leq \rho_3 \leq \mu_2-\mu_0-\lambda.
\end{equation}

Si nous introduisons dans $G_{\mu_0,\mu_2-\mu_0,\lambda,1}(t,k)$ le r\'esidu $w$ de $v \bmod q^{\rho_3}$ dans le but de rendre les variables $u$ et $v$ ind\'ependantes (notons que contrairement aux sommes de type I il n'y a pas de difficult\'es ici : la taille est fix\'ee), nous obtenons en suivant pas à pas la preuve de \cite{RudinShapiro} :
%\begin{align*}
%& G_{\mu_0,\mu_2-\mu_0,\lambda,1}(t,k) 
%\\ &= \ds\sum_{0 \leq w < q^{\rho_3}}\frac{1}{q^{\mu_2-\mu_0-\lambda}}\ds\sum_{0 \leq v < q^{\mu_2-\mu_0-\lambda}}f_P(vq^{\mu_0+\lambda},k)e\left(- \frac{vq^\lambda t}{q^{\mu_2-\mu_0}} \right)\frac{1}{q^{\rho_3}}\ds\sum_{0 \leq l < q^{\rho_3}}e\left(l \frac{v-w}{q^{\rho_3}} \right)
%\\ & \quad \cdot \frac{1}{q^{\lambda}} \ds\sum_{0 \leq u < q^{\lambda}}f_P^{(\mu_0+\lambda+\rho_3)}(q^{\mu_0}(u+wq^{\lambda}),k)\overline{f_P^{(\mu_0+\lambda+\rho_3)}(wq^{\mu_0+\lambda},k)f_P^{(\mu_1)}(uq^{\mu_0},k)}e\left(- \frac{ut}{q^{\mu_2-\mu_0}} \right),
%\end{align*} ce qui nous donne 
\begin{equation*}
G_{\mu_0,\mu_2-\mu_0,\lambda,1}(t,k)  = \ds\sum_{0 \leq l < q^{\rho_3}}\frac{\tilde{c_l}(t)}{q^{\mu_2-\mu_0-\lambda}}\ds\sum_{0 \leq v < q^{\mu_2-\mu_0-\lambda}}f_P(vq^{\mu_0+\lambda},k)e\left( - \frac{vt}{q^{\mu_2-\mu_0-\lambda}}+\frac{vl}{q^{\rho_3}}\right)
\end{equation*} où on a pos\'e \begin{equation*}
\tilde{c_l}(t) = \frac{1}{q^{\rho_3}}\ds\sum_{0 \leq w < q^{\rho_3}}c_\lambda(w,t)e\left(-\frac{wl}{q^{\rho_3}} \right)
\end{equation*} et 
\begin{align*}
& c_\lambda(w,t) = \\ & \frac{1}{q^\lambda}\ds\sum_{0 \leq u < q^\lambda}f_P^{(\mu_0+\lambda+\rho_3)}(q^{\mu_0}(u+wq^{\lambda}),k)\overline{f_P^{(\mu_0+\lambda+\rho_3)}(wq^{\mu_0+\lambda},k)f_P^{(\mu_1)}(uq^{\mu_0},k)}e\left(- \frac{ut}{q^{\mu_2-\mu_0}} \right).
\end{align*} 
Si nous supposons \begin{equation}\label{Equation restriction sur P pour transfo Fourier}
P(k) \leq \lambda + \rho_3,
\end{equation} une démonstration analogue au Lemme \ref{Lemme premiere estimation Fourier} nous donne l'estimation uniforme suivante $$\left|\frac{1}{q^{\lambda+\rho_3}}\ds\sum_{0 \leq u' < q^{\lambda+\rho_3}}f_P(u'q^{\mu_0},k)e\left(-\frac{u't}{q^{\mu_2-\mu_0}}-\frac{l'u'}{q^{\lambda+\rho_3}} \right)\right| \leq q^{-\gamma_P(\lambda+\rho_3,k)}.$$ 
%ce qui nous donne \begin{equation*}
%|c_\lambda(w,t)| \ll \frac{q^{-\gamma_P(\lambda+\rho_3,k)}}{q^{\lambda - \mu_1+\mu_0}}\ds\sum_{0 \leq l' < q^{\lambda+\rho_3}} \min \left(q^{\lambda-\mu_1+\mu_0},\left| \sin \pi \frac{l'q^{\mu_1-\mu_0}}{q^{\lambda+\rho_3}} \right|^{-1} \right).
%\end{equation*} 
%Comme \eqref{Forme gamma deuxieme prop fourier} nous informe que $\gamma_P(l,k)$ est strictement croissante en $l$, en répétant les arguments de \cite{RudinShapiro}, on obtient
Suivant \cite[p $2626$]{RudinShapiro}, nous rattachons cette estimation à $c_{\lambda}(w,t)$ en effectuant des jeux d'écriture par rapport à la fonction tronquée. Ces calculs sont assez lourds à écrire, et, du fait qu'ici la taille ne change pas, ne présentent aucune nouveauté par rapport à \cite{RudinShapiro} Suivant toujours le fil de la preuve, l'identité de Cauchy-Schwarz, la croissance de $\gamma_P(l,k)$ en $l$ assurée par \eqref{Forme gamma deuxieme prop fourier} et l'orthogonalité des caractères permettent alors de conclure à
% et en utilisant le Lemme \ref{Lemme estimation somme geometrique premiere}, et en remarquant qu'il y a $q^{\mu_1-\mu_0}$ p\'eriodes modulo $q^{\lambda+\rho_3-\mu_1+\mu_0}$ et que $q^{\lambda-\mu_1+\mu_0} \leq q^{\lambda+\rho_3} $, nous obtenons \begin{align*}
%|c_\lambda(w,t)| & \ll \frac{q^{-\gamma_P(\lambda+\rho_3,k)}}{q^{\lambda - \mu_1+\mu_0}}q^{\lambda+\rho_3}\log q^{\lambda+\rho_3-\mu_1+\mu_0}
%\\& \ll q^{\rho_3+\mu_1-\mu_0-\gamma_P(\lambda,k)}\log q^{\lambda+\rho_3-\mu_1+\mu_0}.
%\end{align*} 
%Nous en d\'eduisons, par \eqref{equation rho3} et \eqref{Equation cl a cw }, que $$\ds\sum_{0 \leq l < q^{\rho_3}}|\tilde{c_l}(t)|^2 \ll q^{2 \rho_3+2(\mu_1-\mu_0)-2\gamma_P(\lambda,k)}(\log q^{\mu_2-\mu_1})^2, $$ et \begin{align*}
% \ds\sum_{0 \leq h < q^{\mu_2-\mu_0-\lambda}}&|G_{\mu_0,\mu_2-\mu_0,\lambda,1}(h+t,k)|^2 \\ &  \ll q^{2 \rho_3+2(\mu_1-\mu_0)-2\gamma_P(\lambda,k)}(\log q^{\mu_2-\mu_1})^2 
%\\ & \quad \cdot \ds\sum_{0 \leq l < q^{\rho_3}} \frac{1}{q^{2(\mu_2-\mu_0-\lambda)}}\ds\sum_{0 \leq v,v' < q^{\mu_2-\mu_0-\lambda}}f_P(vq^{\mu_0+\lambda},k)\overline{f_P(v'q^{\mu_0+\lambda},k)}\\ & \quad \cdot e\left( - \frac{(v-v')t}{q^{\mu_2-\mu_0-\lambda}}+\frac{(v-v')l}{q^{\rho_3}}\right)\ds\sum_{0 \leq h < q^{\mu_2-\mu_0-\lambda}}e\left(- \frac{(v-v')h}{q^{\mu_2-\mu_0-\lambda}} \right) ,
%\end{align*} c'est-\`a-dire 
\begin{equation}\label{Equation gros theoreme Fourier, type I}
 \ds\sum_{0 \leq h < q^{\mu_2-\mu_0-\lambda}}|G_{\mu_0,\mu_2-\mu_0,\lambda,1}(h+t,k)|^2 \ll q^{3 \rho_3+2(\mu_1-\mu_0)-2\gamma_P(\lambda,k)}(\log q^{\mu_2-\mu_1})^2.
\end{equation} 

Pour le terme d'erreur, 
% Nous rappelons que \begin{align*}
%& G_{\mu_0,\mu_2-\mu_0,\lambda,2}(t,k) 
%\\ &= \frac{1}{q^{\mu_2-\mu_0}}\ds\sum_{(u,v) \in \widetilde{\mathcal{W}_\lambda}}f_P(vq^{\mu_0+\lambda})f_P^{(\mu_1)}(uq^{\mu_0},k)e\left(- \frac{(u+vq^\lambda) t}{q^{\mu_2-\mu_0}} \right)
%\\ & \quad \cdot\bigg(f_P(q^{\mu_0}(u+vq^{\lambda}),k)\overline{f_P(vq^{\mu_0+\lambda},k)}
%\\ & \qquad -f_P^{(\mu_0+\lambda+\rho_3)}(q^{\mu_0}(u+vq^{\lambda}),k)\overline{f_P^{(\mu_0+\lambda+\rho_3)}(vq^{\mu_0+\lambda},k)}\bigg)
%\\ & = \frac{1}{q^{\mu_2-\mu_0}}\ds\sum_{w < q^{\mu_2-\mu_0}}c_{\lambda,k}'(w)e\left(-\frac{wt}{q^{\mu_2-\mu_0}} \right),
%\end{align*} 
%où $|c_{\lambda,k}'(w)| \leq 2$ et $c_{\lambda,k}'(w) = 0$ si $w \notin \mathcal{W}_\lambda$, où $\mathcal{W}_\lambda$ est l'ensemble des $w= u + vq^\lambda, (u,v) \in \widetilde{\mathcal{W}_\lambda}$. Il vient donc %\begin{align*}
 %\ds\sum_{0 \leq h < q^{\mu_2-\mu_0}}|G_{\mu_0,\mu_2-\mu_0,\lambda,2}(h+t,k)|^2 &= \ds\frac{1}{q^{2(\mu_2-\mu_0)}}\ds\sum_{0 \leq w,w' < q^{\mu_2-\mu_0}}c_{\lambda,k}'(w) \overline{c_{\lambda,k}'(w')}e\left( -\frac{(w-w')t}{q^{\mu_2-\mu_0}}\right) \\ & \qquad \cdot \ds\sum_{0 \leq h < q^{\mu_2-\mu_0}}e\left(- \frac{(w-w')h}{q^{\mu_2-\mu_0}} \right)
 %\\ & = \ds\frac{1}{q^{\mu_2-\mu_0}}\ds\sum_{w<q^{\mu_2-\mu_0}}|c_{\lambda,k}'(w)|^2.
%\end{align*}
par des résultats classiques sur les propagations des retenues (là encore, voir \cite{RudinShapiro} ou \cite{Hanna1}), le nombre de $v \in \{0, \ldots , q^{\mu_2-\mu_0-\lambda}\}$ tels qu'il existe $u \in \{0,\ldots,q^{\lambda}-1\}$ pour lequel \begin{equation}\label{Equation pour lemme G}
f_P(uq^{\mu_0}+vq^{\mu_0+\lambda},k)\overline{f_P(vq^{\mu_0+\lambda},k)} \neq f_P^{(\mu_0+\lambda+\rho_3)}(uq^{\mu_0}+vq^{\mu_0+\lambda},k)\overline{f_P^{(\mu_0+\lambda+\rho_3)}(vq^{\mu_0+\lambda},k)}
\end{equation} est $O\left(q^{\mu_2-\mu_0-\lambda-\rho_3+P(k)}\right)$. Ainsi, en sommant sur $u$, l'ensemble $\widetilde{\mathcal{W}_\lambda}$ des couples $(u,v)$ satisfaisant \`a \eqref{Equation pour lemme G} v\'erifie \begin{equation}\label{Equation cardinal estim fourier double tronquee}
\# \widetilde{\mathcal{W}_\lambda} \ll q^{\mu_2-\mu_0-\rho_3+P(k)}.
\end{equation}En développant les modules au carré et en utilisant l'orthogonalité des caractères, en suivant le fil que la preuve de \cite{RudinShapiro}, on peut montrer que : 
 \begin{equation}\label{Equation gros theoreme Fourier type II}
 \ds\sum_{0 \leq h < q^{\mu_2-\mu_0-\lambda}}|G_{\mu_0,\mu_2-\mu_0,\lambda,2}(h+t,k)|^2  \ll \ds\frac{\#\widetilde{\mathcal{W}_\lambda}}{q^{\mu_2-\mu_0}} \ll q^{-\rho_3+P(k)}.
\end{equation}
En rassemblant \eqref{Equation gros theoreme Fourier, type I} et \eqref{Equation gros theoreme Fourier type II}, nous obtenons \begin{equation*}
 \ds\sum_{0 \leq h < q^{\mu_2-\mu_0-\lambda}}|G_{\mu_0,\mu_2-\mu_0}(h+t,k)|^2 \ll q^{-\rho_3+P(k)} + q^{3 \rho_3+2(\mu_1-\mu_0)-2\gamma_P(\lambda,k)}(\log q^{\mu_2-\mu_1})^2,
\end{equation*} et en posant $\rho_3 = \max(1,\frac{1}{2}(\gamma_P(\lambda,k)-\mu_1+\mu_0)+P(k)/4)$, nous avons le r\'esultat demand\'e. 
Par ailleurs \begin{align*}
& \frac{1}{2}(\gamma_P(\lambda,\kappa)-\mu_1+\mu_0)+\frac{P(k)}{4} \leq \mu_2-\mu_0-\lambda
\\ & \qquad \Leftrightarrow \lambda + \frac{\gamma_P(\lambda,k)}{2} \leq (\mu_2-\mu_0)+\frac{1}{2}(\mu_1-\mu_0)-\frac{P(k)}{4},
\end{align*} or, comme $\gamma_P(\lambda,k) \leq \lambda /2$, par \eqref{Encadrement lambda} nous avons $\lambda + \frac{\gamma_P(\lambda,k)}{2} \leq (\mu_2-\mu_0)$. La condition \eqref{Equation P \'enonc\'e} permet de conclure \`a \eqref{equation rho3}.

Enfin, toujours par \eqref{Equation P \'enonc\'e}, nous avons 
\begin{align*}
P(k) \leq \frac{1}{3}(\mu_1-\mu_0) \leq \frac{1}{3}(\mu_2-\mu_0) \leq \lambda, 
\end{align*} o\`u la derni\`ere in\'egalit\'e r\'esulte de \eqref{Encadrement lambda}, et donc \eqref{Equation restriction sur P pour transfo Fourier} est v\'erifi\'ee.
\end{proof}

\section{Sommes de type II}\label{Section SII}

Soient $M$ et $N$ des entiers tels que $1 \leq M \leq N$, nous notons, tout comme dans la Partie \ref{Section SI}, $q^{\mu-1} \leq M < q^\mu$ et $q^{\nu-1} \leq N < q^\nu$. Nous supposons ici en outre \begin{equation}\label{Condis sur mu et nu SII}
\frac{1}{4}(\mu+\nu) \leq \mu \leq \nu \leq \frac{3}{4}(\mu+\nu)
\end{equation}

Soit $\vartheta \in \R, a_m \in \C, b_n \in \C$ avec $|a_m| \leq 1, |b_n| \leq 1$. Nous \'ecrivons alors

\begin{equation}\label{def SII}
S_{II}(\vartheta) := \ds\sum_{\frac{M}{q} < m \leq M}\ds\sum_{\frac{N}{q} < n \leq N}a_mb_nf_P(mn)e(\vartheta mn).
\end{equation} 

Nous allons prouver le
\begin{theorem}
Soit $P : \N \rightarrow \N$ telle que \begin{equation}\label{Grosse restriction sur P}
P(\mu+\nu+1) \leq \frac{2}{3}\bigg\lfloor \frac{1}{16} \gamma_P\left(\frac{\mu+\nu}{640}, \, \mu+\nu-2\right) \bigg\rfloor.
\end{equation}
Alors nous avons uniform\'ement en $\vartheta \in \R$ \begin{equation}\label{equation resulta SII}
\left|S_{II}(\vartheta)\right| \ll (\max\big(\tau(q),(\log q)^3\big)(\mu+\nu)^{\omega(q)+3}\big)^{1/4}q^{\mu+\nu-\frac{1}{64}\gamma_P( \frac{\mu+\nu}{640},\, \mu+\nu-2)},
\end{equation}
où $\gamma_P(k,l)$ est d\'efinie en \eqref{Forme gamma deuxieme prop fourier}
\end{theorem}

\begin{proof} Pareillement que \cite{RudinShapiro},
si $\rho$ est un entier tel que \begin{equation}\label{Equation restricition rho}
1 \leq 7\rho\leq \mu,
\end{equation}
en posant $R=q^{\rho}$ de sorte que $1\leq R \ll N$, on obtient par l'inégalité de van der Corput : 
\begin{equation}\label{Premier VdC}\left|S_{II}(\vartheta)\right|^2 \ll \ds\frac{M^2N^2}{R}+\ds\frac{MN}{R}\ds\sum_{1\leq r < R}\left(1-\frac{r}{R}\right)\Re (S_1(r)),
\end{equation} avec $$S_1(r) := \ds\sum_m\ds\sum_{n \in I_1(N,r)} b_{n+r}\overline{b_n}f_P(mn+mr)\overline{f_P(mn)}e(\vartheta mr),$$ et $I_1(N,r) = (N/q, N-r]$.  

Si de plus $\rho$ v\'erifie \begin{equation}\label{Equation determinant premiere condition troncature SII}
P(\mu+\nu) \leq \rho
\end{equation} on peut utiliser le Lemme \ref{Lemme digital premiere troncature SII} pour obtenir  \begin{equation}\label{SII premiere erreur} S_1(r) = S_1'(r) + O( (\log q)q^{\mu+\nu-\rho+P(\mu+\nu+1)}),
\end{equation}  avec $$S_1'(r) = \ds\sum_m\ds\sum_{n \in I_1(N,r)} b_{n+r}\overline{b_n}f_P^{(\mu+2\rho)}(mn+mr)\overline{f_P^{(\mu + 2\rho)}(mn)}e(\vartheta mr),$$ et nous posons \begin{equation}\label{mu2}
\mu_2 := \mu+2\rho.
\end{equation}
En  appliquant l'in\'egalit\'e de Cauchy-Schwarz et en \'etendant la sommation sur $n$ \`a l'intervalle $(N/q,N]$, on trouve : $$|S_1'(r)|^2 \ll N\ds\sum_n\left|\ds\sum_mf_P^{(\mu_2)}(mn+mr)\overline{f_P^{(\mu_2)}(mn)}e(\vartheta mr) \right|^2, $$ et une utilisation du \cite[Lemme $4$]{MR:gelfond} avec des param\`etres $\mu_1$ et $S$ tels que \begin{equation}\label{condition de parametrisation pour mu1}
1 \leq q^{\mu_1}S \ll M
\end{equation} donne :
\begin{equation}\label{Equation SII deuxieme erreur}
 \ds\sum_{1 \leq r < R}|S_1'(r)|^2 \ll \ds\frac{M^2N^2R}{S}+\ds\frac{MN}{S}\Re(S_2), 
\end{equation}
o\`u \begin{equation}\label{Equation S2(r,s)} S_2 := \ds\sum_{1 \leq r < R}\ds\sum_{1 \leq s < S}\left( 1-\frac{s}{S}\right)  S_2'(r,s) e\left(\vartheta q^{\mu_1}rs\right) 
\end{equation} 
et $$S_2'(r,s) = \ds\sum_{m \in I_2(M,s)}\ds\sum_n f_P^{(\mu_2)}((m+sq^{\mu_1})(n+r))\overline{f_P^{(\mu_2)}(m(n+r))f_P^{(\mu_2)}((m+sq^{\mu_1})n)}f_P^{(\mu_2)}(mn),$$ avec $I_2(M,s) = (M/q, M-sq^{\mu_1}]$.

%La suite logique ici consiste \`a remplacer chaque fonction $f_P^{(\mu_2)}$ par la fonction doublement tronqu\'ee $f_P^{(\mu_1,\mu_2)}:=f_P^{(\mu_2)}\overline{f_P^{(\mu_1)}}$ par l'interm\'ediaire du jeu d'\'ecriture $$f_P^{(\mu_2)} = f_P^{(\mu_2)}\overline{f_P^{(\mu_1)}}f_P^{(\mu_1)}.$$

%Il est possible de le faire dans \cite{RudinShapiro} car les fonctions tronqu\'ees permettent le passage aux transform\'ees de Fourier discr\`etes. En effet, le raisonnement repose sur l'\'ecriture suivante : \begin{align*}
%\notag S_2'(r,s) = &\ds\sum_{m \in I_2(M,s)}  \ds\sum_n  f_P^{(\mu_1,\mu_2)}((m+sq^{\mu_1})(n+r))
%\\& \notag \cdot \overline{f_P^{(\mu_1,\mu_2)}(m(n+r))f_P^{(\mu_1,\mu_2)}((m+sq^{\mu_1})n)}f_P^{(\mu_1,\mu_2)}(mn)
%\\  & \cdot f_P^{(\mu_1)}((m+sq^{\mu_1})(n+r))\overline{f_P^{(\mu_1)}(m(n+r))f_P^{(\mu_1)}((m+sq^{\mu_1})n)}f_P^{(\mu_1)}(mn),
%\end{align*} et il n'est pas clair que pour nous la derni\`ere ligne vaille $1$. On pourrait tr\`es bien imaginer $P(T_q((m+sq^{\mu_1})(n+r))) \neq P(T_q(m(n+r)))$ auquel cas il n'y aurait aucune raison de trouver l'identit\'e esp\'er\'ee. 

Si nous posons $S = q^{2\rho}$ et $\mu_1 = \mu-3\rho$, et
\begin{align*}
 S_2''(r,s) :=  \ds\sum_{m \in I_2(M,s)}&\ds\sum_n  f_P^{(\mu_1,\mu_2)}((m+sq^{\mu_1})(n+r))
 \\& \cdot \overline{f_P^{(\mu_1,\mu_2)}(m(n+r))f_P^{(\mu_1,\mu_2)}((m+sq^{\mu_1})n)}f_P^{(\mu_1,\mu_2)}(mn),
\end{align*}
 nous avons $S_2'(r,s) = S_2''(r,s) + S_2'''(r,s)$ où $S_2'''(r,s)$ satisfait à $$|S_2'''(r,s)| \leq \#\{(m,n) : T_q(mn) \neq T_q((m+Sq^{\mu_1})(n+q^\rho))\}.$$ En effet, pour les $m$ et $n$ tels que $T_q(mn) = T_q((m+Sq^{\mu_1})(n+q^\rho))$, on peut remplacer directement la fonction tronquée par la fonction doublement tronquée car $m + sq^{\mu_1} \equiv m \bmod q^{\mu_1}$, et donc $f_P^{(\mu_1)}((m+sq^{\mu_1})n) = f_P^{(\mu_1)}(mn) $, et de même pour $m(n+r)$ et $(m+sq^{\mu_1})(n+r)$. Il suffit alors d'écrire $f_P^{(\mu_2)} = \overline{f_P^{(\mu_1)}}f_P^{(\mu_2)}f_P^{(\mu_1)}$.
 
\smallskip 
 
 %En effet, si $T_q(mn)  = T_q((m+Sq^{\mu_1})(n+q^\rho))$, on a en particulier $T_q(mn) = T_q((m+Sq^{\mu_1})n)$ et $T_q(m(n+q^\rho)) = T_q((m+Sq^{\mu_1})(n+q^\rho))$ et ces derni\`eres identit\'es demeurent valables pour $s \leq S$ et $r \leq q^\rho$ en lieu et place de $S$ et $q^\rho$. Mais alors $m + sq^{\mu_1} \equiv m \bmod q^{\mu_1}$, et donc $f_P^{(\mu_1)}((m+sq^{\mu_1})n) = f_P^{(\mu_1)}(mn) $, et de même pour $m(n+r)$ et $(m+sq^{\mu_1})(n+r)$.
 
En imposant \begin{equation}\label{equa rho nu}
 \rho \leq \frac{\nu-1}{2}
\end{equation}  de sorte \`a pouvoir utiliser le Lemme \ref{Lemme Olivier Robert}. Nous obtenons alors : 

\begin{equation}\label{Equation erreur double troncature OR}
S_2'(r,s) = S_2''(r,s) + O((\mu+\nu)q^{\mu+\nu-\rho}).
\end{equation} 

\smallskip 

%Pour des raisons techniques, il est plus int\'eressant pour nous, au lieu de regarder $r_{\mu_1,\mu_2}(n)$, de regarder $r_{\mu_0,\mu_2}(n)$ avec un $\mu_0 \leq \mu_1$ bien contr\^ol\'e. C'est l'objet des calculs suivants.

Soit à présent $\rho'$ tel que \begin{equation}\label{definition rho'}
\rho' \leq \rho,
\end{equation} nous posons
\begin{equation}\label{definition mu0}
\mu_0 = \mu_1-2\rho',
\end{equation} de sorte que $\mu_1-\mu_0 = 2\rho' \leq 3/4(5\rho + 2\rho') = 3/4(\mu_2-\mu_0)$.

Si nous avons \begin{equation}\label{Deuxieme condition sur P en SII}
P(\mu+2\rho) = P(\mu_2)\leq \mu_1-\mu_0 = 2\rho'
\end{equation}

par le Lemme \ref{Lemme digital double troncature} :

\begin{equation}\label{Equation erreur double troncature grand phi}
S_2''(r,s) = S_3(r,s) + O\left(\max (\tau(q),\log q)\mu_2^{\omega(q)}q^{\mu+\nu+\mu_0-\mu_1+P(\mu_2+1)}\right)
\end{equation}

avec \begin{align*}S_3(r,s) = \ds\sum_m \ds\sum_n & \varphi_P\Big(r_{\mu_0,\mu_2}((m+sq^{\mu_1})(n+r)),T_q((m+sq^{\mu_1})(n+r))\Big)
\\ & \cdot\overline{\varphi_P\Big((r_{\mu_0,\mu_2}m(n+r)),T_q(m(n+r))\Big)}\\ & \cdot \overline{\varphi_P\Big(r_{\mu_0,\mu_2}((m+sq^{\mu_1})n),T_q((m+sq^{\mu_1})n)\Big)}\varphi_P\Big(r_{\mu_0,\mu_2}(mn),T_q(mn)\Big),
\end{align*}

où on a pos\'e $$\varphi_P(x,y) := e\left(\alpha \ds\sum_{i=\mu_1-P(y)}^{\mu_2-P(y)}\epsilon_{i+P(y)}(q^{\mu_0}x)\cdots\epsilon_i(q^{\mu_0}x) \right). $$

Nous rappelons que nous pouvons supposer $T_q(mn) = T_q((m+sq^{\mu_1})(n+r))$; mais comme %\begin{align*}
%q^{T_q(m)} \leq m < q^{T_q(m)+1}, \quad q^{T_q(n)} \leq n < q^{T_q(n)+1} \Rightarrow q^{T_q(m)+T_q(n)} \leq mn < q^{T_q(m)+T_q(n)+2},
%\end{align*} ceci veut dire 
$T_q(m)T_q(n) \leq T_q(mn) \leq T_q(m)T_q(n)+1$ et que $m$ et $n$ sont pris de sorte que \begin{align*}
q^{\mu-2} \leq M/q \leq m < M < q^\mu \quad \text{et} \quad q^{\nu-2} \leq N/q \leq n < N < q^\nu,
\end{align*}

% $$q^{\mu-2} \leq M/q \leq m < M < q^\mu$$ et $$q^{\nu-2} \leq N/q \leq n < N < q^\nu,$$ 
nous avons $\mu+\nu-4 \leq T_q(mn) \leq \mu+\nu-1,$ ce qui nous donne la réécriture suivante :

\begin{align*}S_3(r,s) = \ds\sum_{k = \mu+\nu-4}^{\mu+\nu-1}\ds\sum_m \ds\sum_n & \varphi_P\Big(r_{\mu_0,\mu_2}((m+sq^{\mu_1})(n+r)),k\Big)
\overline{\varphi_P\Big((r_{\mu_0,\mu_2}m(n+r)),k\Big)}\\ & \cdot \overline{\varphi_P\Big(r_{\mu_0,\mu_2}((m+sq^{\mu_1})n),k\Big)}\varphi_P\Big(r_{\mu_0,\mu_2}(mn),k\Big),
\end{align*}

%Nous pouvons \`a pr\'esent rendre les entiers consid\'er\'es ``ind\'ependants''. Pour ce faire nous nous int\'eressons aux diff\'erents restes. On remarque  que $\varphi_P(x,y)$ est $q^{\mu_2 - \mu_0}$ p\'eriodique en $x$. Ainsi 
En posant $u_0 := r_{\mu_0,\mu_2}(mn) $ et $u_1:=r_{\mu_0,\mu_2}(mn+mr)$, on a   \begin{equation*}
r_{\mu_0,\mu_2}(mn+q^{\mu_1}sn) = r_{\mu_2-\mu_0}(u_0+q^{\mu_1-\mu_0}sn) = u_0+q^{\mu_1-\mu_0}sn \bmod q^{\mu_2-\mu_0}
\end{equation*}et 
\begin{align*}
r_{\mu_0,\mu_2}(mn+mr+q^{\mu_1}sn+q^{\mu_1}sr) = u_1+q^{\mu_1-\mu_0}sn+q^{\mu_1-\mu_0}sr \bmod q^{\mu_2-\mu_0}.
\end{align*}
Et donc, comme $\varphi_P(x,y) = \varphi_P(x \bmod q^{\mu_2-\mu_0},y) $, nous avons : \begin{align*}
\varphi_P\Big(r_{\mu_0,\mu_2}((m+sq^{\mu_1})n),k\Big)  = \varphi_P\Big(u_0+q^{\mu_1-\mu_0}sn,k\Big)
\end{align*} et \begin{align*}
\varphi_P\Big(r_{\mu_0,\mu_2}&(m(n+r)+nsq^{\mu_1}+rsq^{\mu_1}),k\Big) = \varphi_P\Big(u_1+q^{\mu_1-\mu_0}sn+q^{\mu_1-\mu_0}sr,k\Big).
\end{align*}
Nous rappelons que (\cite[equation $(11)$]{RudinShapiro}) $$r_{\mu_0,\mu_2}(n)= u \Leftrightarrow \frac{n}{q^{\mu_2}} \in \left[\frac{u}{q^{\mu_2-\mu_0}},\frac{u+1}{q^{\mu_2-\mu_0}}\right) + \mathbb{Z}$$ et que $\chi_\alpha(x)$ d\'esigne la fonction caract\'eristique de l'intervalle $[0,\alpha)$ translat\'e dans $\Z$, si bien que
\begin{align*}
S_3(r,s) = &\ds\sum_{m \in I_2(M,s)}\ds\sum_n  \ds\sum_{0\leq u_0, u_1 < q^{\mu_2-\mu_0}}\chi_{q^{\mu_0-\mu_2}}\left(\frac{mn}{q^{\mu_2}}-\frac{u_0}{q^{\mu_2-\mu_0}}\right)\chi_{q^{\mu_0-\mu_2}}\left(\frac{mn+mr}{q^{\mu_2}}-\frac{u_1}{q^{\mu_2-\mu_0}}\right)\\& \cdot\ds\sum_{k = \mu+\nu-4}^{\mu+\nu-1}\varphi_P\Big(u_1+q^{\mu_1-\mu_0}sn+q^{\mu_1-\mu_0}sr,k\Big)\overline{\varphi_P\Big(u_1,k\Big)}\overline{\varphi_P\Big(u_0+q^{\mu_1-\mu_0}sn,k\Big)}\varphi_P\Big(u_0,k\Big).
\end{align*} 
En refaisant sans autre modification que celle induite par la notation les calculs de \cite[Sections $6.1$-$6.2$]{RudinShapiro}, on peut dire à l'aide de \cite[equation $(64)$]{RudinShapiro} et par \cite[Lemma $2$]{RudinShapiro} que
 
\begin{equation}\label{Terme d'erreur Vaaler SII}
S_3(r,s) = S_4(r,s) + O(\max(\log q^{\mu_0},\tau(q^{\mu_0}))q^{\mu+\nu-2\rho}),
\end{equation} 
avec
\begin{align*}
S_4(r,s) =&\ds\sum_k\ds\sum_{m \in I_2(M,s)}\ds\sum_n  \ds\sum_{ 0 \leq u_0, u_1 < q^{\mu_2-\mu_0}}\ds\sum_{|h_0|\leq H}a_{h_0}\left(q^{\mu_0-\mu_2},H\right)e\left(h_0\frac{mn}{q^{\mu_2}}-h_0\frac{u_0}{q^{\mu_2-\mu_0}}\right)
\\ & \cdot\ds\sum_{|h_1|\leq H}a_{h_1}(q^{\mu_0-\mu_2},H)e\left(h_1\frac{mn+mr}{q^{\mu_2}}-h_1\frac{u_1}{q^{\mu_2-\mu_0}}\right)\varphi_P\Big(u_0,k\Big)\\& \cdot\varphi_P\Big(u_1+q^{\mu_1-\mu_0}sn+q^{\mu_1-\mu_0}sr,k\Big)\overline{\varphi_P\Big(u_1,k\Big)}\overline{\varphi_P\Big(u_0+q^{\mu_1-\mu_0}sn,k\Big)},
\end{align*} avec $H = q^{\mu_2-\mu_0+2\rho}$, mais nous conserverons la notation $H$ tant que sa valeur explicite n'est pas utile, et où les coefficients $a_h$ satisfont à \begin{equation}\label{equa coeff a_h}
a_0(\alpha,H) = \alpha, \quad |a_h(\alpha,H)| \leq \min\left(\alpha, \frac{1}{\pi |h|} \right).
\end{equation}

\smallskip

En \'ecrivant $u_0+q^{\mu_1-\mu_0}sn \equiv u_2 \bmod q^{\mu_2-\mu_0} $ et $u_1+q^{\mu_1-\mu_0}sn+q^{\mu_1-\mu_0}sr \equiv u_3 \bmod q^{\mu_2-\mu_0}, $ nous avons :
\begin{align*}
S_4(r,s) &=  \ds\sum_{m \in I_2(M,s)}\ds\sum_n \, \ds\sum_{0 \leq u_0,u_1 <  q^{\mu_2-\mu_0}} \ds\sum_{|h_0|\leq H}a_{h_0}\left(q^{\mu_0-\mu_2},H\right)e\left(h_0\frac{mn}{q^{\mu_2}}-h_0\frac{u_0}{q^{\mu_2-\mu_0}}\right)
\\ & \cdot\ds\sum_{|h_1|\leq H}a_{h_1}(q^{\mu_0-\mu_2},H)e\left(h_1\frac{mn+mr}{q^{\mu_2}}-h_1\frac{u_1}{q^{\mu_2-\mu_0}}\right)\ds\sum_{0 \leq u_2,u_3 < q^{\mu_2-\mu_0}}\ds\frac{1}{q^{2(\mu_2-\mu_0)}}
\\& \cdot\ds\sum_{0 \leq h_2,h_3 < q^{\mu_2-\mu_0}}e\left(h_2\ds\frac{u_0+q^{\mu_1-\mu_0}sn-u_2}{q^{\mu_2-\mu_0}}\right)e\left(h_3\ds\frac{u_1+q^{\mu_1-\mu_0}sn+q^{\mu_1-\mu_0}sr-u_3}{q^{\mu_2-\mu_0}}\right)
\\ & \qquad \cdot\ds\sum_{k = \mu+\nu-4}^{\mu+\nu-1}\varphi_P\Big(u_3,k\Big)\overline{\varphi_P\Big(u_1,k\Big)}
 \overline{\varphi_P\Big(u_2,k\Big)}\varphi_P\Big(u_0,k\Big).
\end{align*}

 ce qui fait qu'en posant \begin{equation*}\label{Ecriture transfo fourier double tronqu\'ee}
 \widetilde{\varphi_P}(h,y):= \frac{1}{q^{\mu_2-\mu_0}}\ds\sum_{0 \leq u < q^{\mu_2-\mu_0}}\varphi_P(u,y)e\left(\frac{-uh}{q^{\mu_2-\mu_0}}\right),
\end{equation*} nous avons l'\'ecriture suivante, toujours si $m \in I_2(M,s)$ :
\begin{align*}
S_4(r,s) = & q^{2(\mu_2-\mu_0)}\ds\sum_{k=\mu+\nu-4}^{\mu+\nu-1} \ds\sum_{|h_0|\leq H}a_{h_0}\left(q^{\mu_0-\mu_2},H\right) \ds\sum_{|h_1|\leq H}a_{h_1}(q^{\mu_0-\mu_2},H)
\\ & \cdot\ds\sum_{0 \leq h_2,h_3 < q^{\mu_2-\mu_0}}e\left(\frac{h_3sr}{q^{\mu_1-\mu_0}}\right) \widetilde{\varphi_P}\big(h_3,k\big)\overline{\widetilde{\varphi_P}(h_3-h_1,k)}\overline{\widetilde{\varphi_P}(-h_2,k)}\widetilde{\varphi_P}(h_0-h_2,k)
\\ & \cdot\ds\sum_{\substack{m,n \\ T_q(mn)=k}}e\left(h_0\frac{mn}{q^{\mu_2}}+h_1\frac{mn+mr}{q^{\mu_2}}+h_2\ds\frac{q^{\mu_1-\mu_0}sn}{q^{\mu_2-\mu_0}}+h_3\ds\frac{q^{\mu_1-\mu_0}sn}{q^{\mu_2-\mu_0}}\right).
\end{align*}

Nous pouvons remarquer que les variables de $\widetilde{\varphi_P}$ sont ``ind\'ependantes''. Nous s\'eparons ici la double somme sur $h_0,h_1$ en deux sommes : $S_4'(r,s)$ qui correspond au cas où $h_0+h_1 = 0$, qui consiste en la contribution principale (du fait que $e(0) = 1$) et $S_4''(r,s)$ qui rassemble les autres termes. %Dans \cite{RudinShapiro}, Mauduit et Rivat traitent le cas $S_4'(r,s)$ en sommant d'abord sur $n$ puis sur $m$ et utilisent que dans ce cas les deux sommations sont ind\'ependantes (puisque $h_0+h_1$, le terme les reliant, est nul). Nous ne pouvons pas appliquer leur m\'ethode textuellement, car même dans ce cas $m$ et $n$ sont reli\'es par la taille.

%Plus pr\'ecis\'ement, nous avons \begin{equation}\label{Traduction somme reli\'ee par la taille}
%\ds\sum_{M/q \leq m < M}\ds\sum_{\substack{N/q \leq n < N \\  T_q(mn) = k}} = \ds\sum_{M/q \leq m < M}\ds\sum_{\substack{N/q \leq n < N \\  q^{k}/m\leq n < q^{k+1}/m }} = \ds\sum_{N/q \leq n < N }\ds\sum_{\substack{M/q \leq m < M \\ q^{k}/n \leq m < q^{k+1}/n}} .
%\end{equation}

%Nous allons commencer par estimer $S_4''(r,s)$, dont le contr\^ole est le plus proche de \cite{RudinShapiro}.

\subsection{Estimation de $S_4''(r,s)$}
%Le fait que $h_0+h_1 \neq 0$ entraîne que la sommation sur $n$ fera apparaître un terme en sinus, et donc la sommation sera naturellement petite. Nous pouvons nous permettre de contr\^oler le reste de mani\`ere relativement triviale.

%Passant la valeur absolue \`a l'int\'erieur, nous avons, si la sommation sur $m$ se fait sur $I_2(M,s)$,
%\begin{align*}
%S_4''(r,s) \ll &\ds\sum_{k=\mu+\nu-4}^{\mu+\nu-1}q^{2(\mu_2-\mu_0)}\ds\sum_{|h_0|\leq H}\ds\sum_{h_1 \neq -h_0} \left|a_{h_0}\left(q^{\mu_0-\mu_2},H\right) \right|\left|a_{h_1}\left(q^{\mu_0-\mu_2},H\right) \right|
%\\ &  \cdot\ds\sum_{0 \leq h_3 < q^{\mu_2-\mu_0}} \Big|  \widetilde{\varphi_P}\big(h_3,k\big)\widetilde{\varphi_P}(h_3-h_1,k)\Big|\ds\sum_{0 \leq h_2 < q^{\mu_2-\mu_0}}\Big|\widetilde{\varphi_P}(-h_2,k)\widetilde{\varphi_P}(h_0-h_2,k)\Big|
%\\ &  \cdot\left|\ds\sum_{\substack{m,n \\ T_q(mn)=k}} e\left(h_0\frac{mn}{q^{\mu_2}}+h_1\frac{mn+mr}{q^{\mu_2}}+h_2\ds\frac{q^{\mu_1-\mu_0}sn}{q^{\mu_2-\mu_0}}+h_3\ds\frac{q^{\mu_1-\mu_0}sn}{q^{\mu_2-\mu_0}}\right)\right| .
%\end{align*}

Soit $k \in \{\mu+\nu-4,\mu+\nu-1\}$ fix\'e, nous traitons d'abord la sommation sur $m$ et $n$ comme suit :
 \begin{align}
\label{equa S4'' debut} &\left|\ds\sum_{m \in I_2(M,s)}\ds\sum_{\substack{N/q \leq n < N \\ T_q(mn)=k}}\right.  \left.e\left(h_0\frac{mn}{q^{\mu_2}}+h_1\frac{mn+mr}{q^{\mu_2}}+h_2\ds\frac{q^{\mu_1-\mu_0}sn}{q^{\mu_2-\mu_0}}+h_3\ds\frac{q^{\mu_1-\mu_0}sn}{q^{\mu_2-\mu_0}}\right)\right|  
\\ \notag &= \left| \ds\sum_{N/q \leq n < N}e\left((h_2+h_3)q^{\mu_1-\mu_2}sn \right)\ds\sum_{m \in I_2(M,s) \cap \big[ \frac{q^{k}}{n},\frac{q^{k+1}}{n}\big)}e\left(m\left[\frac{n(h_0+h_1)}{q^{\mu_2}}+\frac{h_1r}{q^{\mu_2}}\right]\right) \right|
\end{align} et comme $I_2(M,s) \subseteq [M/q,M) \subseteq [q^{\mu-2},q^\mu)$, la derni\`ere ligne est
\begin{align*}
 & \leq \ds\sum_{N/q \leq n < N}\min \left( \left|[ M/q,M) \cap \Big[ \frac{q^{k}}{n},\frac{q^{k+1}}{n}\Big) \right|, \left| \sin \pi \frac{(h_0+h_1)n+h_1r}{q^{\mu_2}} \right|^{-1} \right)
\\ \notag & \leq \ds\sum_{q^{\nu-1} \leq n < q^{\nu}}\min \left( q^{\mu}, \left| \sin \pi \frac{(h_0+h_1)n+h_1r}{q^{\mu_2}} \right|^{-1} \right) \ll \lceil q^{\nu-\mu_2} \rceil ((h_0+h_1,q^{\mu_2})q^{\mu}+q^{\mu_2}\log q^{\mu_2}),
\end{align*}
par \cite[Lemme $6$]{MR:gelfond}. Mais $|h_0+h_1| \leq 2H $, par ailleurs, en ayant choisi $H = q^{\mu_2-\mu_0+2\rho} $, nous avons $Hq^{\mu} \geq q^{\mu+\mu_2-\mu_0} \geq q^{\mu_2},$ et donc

\begin{equation*}
\eqref{equa S4'' debut} \ll \lceil q^{\nu-\mu_2} \rceil Hq^{\mu}\log q^{\mu_2}.
\end{equation*}

Par orthogonalit\'e des caract\`eres, nous avons pour tout $k$ et $y$ :

\begin{equation}\label{Parseval dans notre cas}
\ds\sum_{0 \leq h < q^{\mu_2-\mu_0}}\left| \widetilde{\varphi_P}(h+y,k) \right|^2 =  1
\end{equation}
et nous obtenons donc par l'in\'egalit\'e de Cauchy-Schwarz : 
\begin{align*}
 &\ds\sum_{0 \leq h_3 < q^{\mu_2-\mu_0}}\Big| \widetilde{\varphi_P}\big(h_3,k\big) \widetilde{\varphi_P}(h_3-h_1,k)\Big| \leq 1  \text{ et }  \ds\sum_{0 \leq h_2 < q^{\mu_2-\mu_0}}\Big|\widetilde{\varphi_P}(-h_2,k)\widetilde{\varphi_P}(h_0-h_2,k)\Big| \leq 1.
\end{align*}

%\begin{align*}
% &\ds\sum_{0 \leq h_3 < q^{\mu_2-\mu_0}}\Big| \widetilde{\varphi_P}\big(h_3,k\big) \widetilde{\varphi_P}(h_3-h_1,k)\Big|\\ & \; \leq \left(\ds\sum_{0 \leq h_3 < q^{\mu_2-\mu_0}} \Big| \widetilde{\varphi_P}\big(h_3,k\big)\Big|^2\right)^{1/2} \left(\ds\sum_{0 \leq h_3 < q^{\mu_2-\mu_0}} \Big|\widetilde{\varphi_P}(h_3-h_1,k)\Big|^2\right)^{1/2} \leq 1
%\end{align*}

%et de mani\`ere similaire $$ \ds\sum_{0 \leq h_2 < q^{\mu_2-\mu_0}}\Big|\widetilde{\varphi_P}(-h_2,k)\widetilde{\varphi_P}(h_0-h_2,k)\Big| \leq 1. $$

De plus, d'apr\`es \eqref{equa coeff a_h} et le choix $H = q^{\mu_2-\mu_0+2\rho}$ : $$\ds\sum_{|h| \leq H}|a_h(q^{\mu_0-\mu_2},H)| \leq \ds\sum_{|h| \leq q^{\mu_2-\mu_0}}\frac{1}{q^{\mu_2-\mu_0}}+\ds\sum_{q^{\mu_2-\mu_0} < |h| \leq H}\frac{1}{\pi |h|}\ll \log (H/q^{\mu_2-\mu_0})=\log q^\rho.$$

Ainsi nous avons $$|S_4''(r,s)| \ll  (\log q)^2\rho^2q^{2(\mu_2-\mu_0)}\lceil q^{\nu-\mu_2} \rceil H q^{\mu}\log q^{\mu_2}, $$ et avec le choix $H=q^{\mu_2-\mu_0+2\rho}$, et comme $\lceil x \rceil \leq x+1$, \begin{equation}\label{S4 seconde}
|S_4''(r,s)| \ll (\log q)^3(\mu+\nu)^3q^{\mu+\nu+3(\mu_2-\mu_0)+2\rho}(q^{-\mu_2}+q^{-\nu}).
\end{equation}

\subsection{Estimation de $S_4'(r,s)$} %Nous avons, par d\'efinition, puisque $h_0+h_1 = 0$ : \begin{align*}
%S_4'(r,s) = & q^{2(\mu_2-\mu_0)}\ds\sum_{k=\mu+\nu-4}^{\mu+\nu-1} \ds\sum_{|h_1|\leq H}a_{h_1}\left(q^{\mu_0-\mu_2},H\right) a_{-h_1}\left(q^{\mu_0-\mu_2},H\right)
%\\ & \cdot\ds\sum_{0 \leq h_2,h_3 < q^{\mu_2-\mu_0}}e\left(\frac{h_3sr}{q^{\mu_2-\mu_1}}\right) \widetilde{\varphi_P}\big(h_3,k\big)\overline{\widetilde{\varphi_P}(h_3-h_1,k)}\overline{\widetilde{\varphi_P}(-h_2,k)}\widetilde{\varphi_P}(-h_1-h_2,k)
%\\ & \cdot\ds\sum_{m \in I_2(M,s)}\ds\sum_{\substack{N/q \leq n < N \\ T_q(mn)=k}}e\left(h_1\frac{mr}{q^{\mu_2}}+h_2\ds\frac{q^{\mu_1-\mu_0}sn}{q^{\mu_2-\mu_0}}+h_3\ds\frac{q^{\mu_1-\mu_0}sn}{q^{\mu_2-\mu_0}}\right).
%\end{align*}
Posons $h = h_2+h_3$, de sorte \`a unifier les termes en $n$. Ceci nous donne  
\begin{align*}
S_4'(r,s) \ll &  q^{2(\mu_2-\mu_0)}\ds\sum_{k=\mu+\nu-4}^{\mu+\nu-1} \ds\sum_{|h_1|\leq H}\left|a_{h_1}\left(q^{\mu_0-\mu_2},H\right)\right|^2
\\ & \ds\sum_{0 \leq h < 2q^{\mu_2-\mu_0}}\ds\sum_{0 \leq h_3 < q^{\mu_2-\mu_0}}\left| \widetilde{\varphi_P}\big(h_3,k\big)\widetilde{\varphi_P}(h_3-h_1,k)\widetilde{\varphi_P}(h_3-h,k)\widetilde{\varphi_P}(h_3-h_1-h,k)\right|
\\ & \left|\ds\sum_{m \in I_2(M,s)}\ds\sum_{\substack{q^{\nu-1} \leq n < q^\nu \\ T_q(mn)=k}}e\left(m\frac{h_1r}{q^{\mu_2}}+n\frac{sh}{q^{\mu_2-\mu_1}}\right)\right|.
\end{align*}

 Il s'agit \`a pr\'esent de rendre les lignes ind\'ependantes les unes des autres. Si nous posons \begin{align*}
 S_6(h,h_1,k) = \ds\sum_{0 \leq h_3 < q^{\mu_2-\mu_0}}\left| \widetilde{\varphi_P}\big(h_3,k\big)\widetilde{\varphi_P}(h_3-h_1,k)\widetilde{\varphi_P}(h_3-h,k)\widetilde{\varphi_P}(h_3-h_1-h,k)\right|
 \end{align*} par l'in\'egalit\'e de Cauchy-Schwarz et par $q^{\mu_2-\mu_0}$ p\'eriodicit\'e, nous avons : 
\begin{align*}
|S_6(h,h_1,k)|  \leq \ds\sum_{0 \leq h_3 < q^{\mu_2-\mu_0}}\left| \widetilde{\varphi_P}\big(h_3,k\big)\widetilde{\varphi_P}(h_3-h_1,k)\right|^2=:S_6'(h_1,k).
\end{align*} Ce qui donne : \begin{align*}
S_4'(r,s) \ll  q^{2(\mu_2-\mu_0)}&\ds\sum_{k = \mu+\nu-4}^{\mu+\nu-1} \ds\sum_{|h_1|\leq H}\left|a_{h_1}\left(q^{\mu_0-\mu_2},H\right)\right|^2 S_6'(h_1,k)
\\ & \quad  \cdot \ds\sum_{1 \leq h < 2q^{\mu_2-\mu_0}}\left|\ds\sum_{m \in I_2(M,s)}\ds\sum_{\substack{N/q \leq n < N \\ T_q(mn)=k}}e\left(m\frac{h_1r}{q^{\mu_2}}+n\frac{sh}{q^{\mu_2-\mu_1}}\right)\right|.
\end{align*} 
Nous constatons qu'il y a essentiellement deux termes :
\begin{equation}\label{Transition S4 S5}
\ds\sum_{1 \leq s < S}|S_4'(r,s)|  \ll  \ds\sum_{\substack{ 1 \leq s < S \\ 0 \leq h < 2q^{\mu_2-\mu_0} \\ q^{\mu_2-\mu_1} | hs}}S_5(r,h) + \ds\sum_{\substack{ 1 \leq s < S \\ 0 \leq h < 2q^{\mu_2-\mu_0} \\ q^{\mu_2-\mu_1} \nmid  hs}}S_5'(r,h,s),
\end{equation} avec
\begin{align*}
 S_5(r,h) :=  q^{2(\mu_2-\mu_0)}\ds\sum_{k = \mu+\nu-4}^{\mu+\nu-1} \ds\sum_{|h_1|\leq H}\left|a_{h_1}\left(q^{\mu_0-\mu_2},H\right)\right|^2 S_6'(h_1,k)  \left|\ds\sum_{\substack{m,n \\ T_q(mn)=k}}e\left(m\frac{h_1r}{q^{\mu_2}}\right)\right| ,
\end{align*}
et \begin{align*}
& S_5'(r,h,s) := 
\\ & q^{2(\mu_2-\mu_0)}\ds\sum_{k = \mu+\nu-4}^{\mu+\nu-1} \ds\sum_{|h_1|\leq H}\left|a_{h_1}\left(q^{\mu_0-\mu_2},H\right)\right|^2 S_6'(h_1,k) \left|\ds\sum_{\substack{m,n \\ T_q(mn)=k}}e\left(m\frac{h_1r}{q^{\mu_2}}+n\frac{hs}{q^{\mu_2-\mu_1}}\right)\right|
\end{align*} o\`u la sommation sur $m$ se fait sur l'intervalle $I_2(M,s)$.

Les termes où l'argument concernant $n$ est altéré sont les plus faciles \`a traiter. D'apr\`es le Lemme \ref{h2+h0 diff 0}, nous avons \begin{align*}
 &S_5'(r,h,s)  \\ &   \ll q^{2(\mu_2-\mu_0)}\ds\sum_{k = \mu+\nu-4}^{\mu+\nu-1} \ds\sum_{|h_1|\leq H}\left|a_{h_1}\left(q^{\mu_0-\mu_2},H\right)\right|^2 S_6'(h_1,k)q^{\mu_2-\mu_1}\left(sq^{\mu_1-\mu_0}q^{\mu_2-\mu_1} \right)^{1/2}q^{\frac{7}{8}(\mu+\nu)}.
\end{align*} 

Cependant \begin{equation}\label{Parseval double}
\ds\sum_{0 \leq h_1 < q^{\mu_2-\mu_0}}S_6'(h_1,k) =
\ds\sum_{0 \leq h_1 < q^{\mu_2-\mu_0}}\ds\sum_{0 \leq h_3 < q^{\mu_2-\mu_0}}\left| \widetilde{\varphi_P}\big(h_3,k\big)\widetilde{\varphi_P}(h_3-h_1,k)\right|^2 = 1.
\end{equation}  Nous s\'eparons donc $S_5'(r,h,s)$ selon $|h_1| \leq q^{\mu_2-\mu_0}$, qu'on nomme $S_{7,1}(r,h,s)$ et selon $q^{\mu_2-\mu_0} < |h_1| \leq H $ qu'on nomme $S_{7,1}'(r,h,s)$. Cette séparation est justifiée par la contribution des coefficients $a_h$. Par \eqref{Parseval double}, en utilisant $|a_{h_1}(q^{\mu_0-\mu_2},H)|\leq q^{\mu_0-\mu_2} $, nous avons l'estimation \begin{equation}\label{S7,1}
S_{7,1}(r,h,s) \ll q^{\mu_2-\mu_1}\left(sq^{\mu_1-\mu_0}q^{\mu_2-\mu_1} \right)^{1/2}q^{\frac{7}{8}(\mu+\nu)} \ll q^{\frac{7}{8}(\mu+\nu)+\frac{3}{2}(\mu_2-\mu_1)+\frac{1}{2}(\mu_2-\mu_0)}s^{1/2}.
\end{equation}

 À présent, si $q^{\mu_2-\mu_0} < |h_1| \leq H $, alors $|a_{h_1}(q^{\mu_0-\mu_2},H)|\leq \frac{1}{\pi|h_1|}$, et donc \begin{align*}
S_{7,1}'(r,h,s) \ll q^{2(\mu_2-\mu_0)}\ds\sum_{k = \mu+\nu-4}^{\mu+\nu-1} \ds\sum_{q^{\mu_2-\mu_0} \leq |h_1|\leq H}\frac{S_6'(h_1,k)}{|h_1|^2} q^{\frac{7}{8}(\mu+\nu)}q^{\mu_2-\mu_1}\left(sq^{\mu_1-\mu_0}q^{\mu_2-\mu_1} \right)^{1/2}.
\end{align*}
Mais $S_6'(h_1,k)$ est $q^{\mu_2-\mu_0}$ p\'eriodique en $h_1$. Nous pouvons s\'eparer la sommation en $jq^{\mu_2-\mu_0} \leq |h| \leq (j+1)q^{\mu_2-\mu_0} $ avec $1 \leq j \leq H/q^{\mu_2-\mu_0}$ et borner $|h_1|^{-2}$ par $j^{-2}q^{2(\mu_0-\mu_2)} $. Ceci nous m\`ene \`a dire que 
\begin{align*}
S_{7,1}'&(r,h,s) \\ & \ll q^{\frac{7}{8}(\mu+\nu)+\frac{3}{2}(\mu_2-\mu_1)+2(\mu_2-\mu_0)}\left(sq^{\mu_1-\mu_0}\right)^{1/2}\ds\sum_{k = \mu+\nu-4}^{\mu+\nu-1}\ds\sum_{j \geq 1}\frac{1}{j^2q^{2(\mu_2-\mu_0)}} \ds\sum_{0 \leq |h_1|< q^{\mu_2-\mu_0}}S_6'(h_1,k)
\\ & \ll q^{\frac{7}{8}(\mu+\nu)+\frac{3}{2}(\mu_2-\mu_1)+\frac{1}{2}(\mu_2-\mu_0)}s^{1/2},
\end{align*} et par \eqref{S7,1}, nous obtenons  \begin{equation}\label{S5'}
\ds\sum_{\substack{ 1 \leq s < S \\ 0 \leq h < q^{\mu_2-\mu_0} \\ q^{\mu_2-\mu_1} \nmid  hs}} S_5'(r,h,s)  \ll q^{\frac{7}{8}(\mu+\nu)+\frac{3}{2}(\mu_2-\mu_1)+\frac{3}{2}(\mu_2-\mu_0)}S^{3/2}.
\end{equation}

\`A pr\'esent contr\^olons $S_5(r,h)$.
Nous utilisons le Lemme \ref{Equa terme principal somme mn SII} pour pouvoir dire \begin{equation}
 \left|\ds\sum_{m \in I_2(M,s)}\ds\sum_{\substack{N/q \leq n < N \\ T_q(mn)=k}}e\left(h_1\frac{mr}{q^{\mu_2}}\right)\right| \ll q^\nu \min\left(q^\mu, \left|\sin \pi \frac{h_1r}{q^{\mu_2}} \right|^{-1} \right),
\end{equation} ce qui am\`ene \`a \begin{align*}
 S_5(r,h,s) & \ll  q^{2(\mu_2-\mu_0)}\ds\sum_{k = \mu+\nu-4}^{\mu+\nu-1} \ds\sum_{|h_1|\leq H}\left|a_{h_1}\left(q^{\mu_0-\mu_2},H\right)\right|^2 S_6'(h_1,k)  q^\nu \min\left(q^\mu, \left|\sin \pi \frac{h_1r}{q^{\mu_2}} \right|^{-1} \right).
\end{align*} Notons que $|h_1r| \leq HR = q^{\mu_2-\mu_0+3\rho}=q^{\mu_2-\mu+6\rho+2\rho'} \leq q^{\mu_2-\mu+8\rho}$. Si l'on suppose \begin{equation}\label{Restriction rho}
8\rho < \mu,
\end{equation} nous avons $|h_1r| < q^{\mu_2}$ et donc $\left|\sin \pi \frac{h_1r}{q^{\mu_2}} \right|^{-1} \leq \ds\frac{q^{\mu_2}}{r|h_1|}$, ce qui donne \begin{align*}
 S_5(r,h) & \ll  q^{2(\mu_2-\mu_0)}\ds\sum_{k = \mu+\nu-4}^{\mu+\nu-1} \ds\sum_{|h_1|\leq H}\left|a_{h_1}\left(q^{\mu_0-\mu_2},H\right)\right|^2 S_6'(h_1,k)  q^\nu \min\left(q^\mu, \frac{q^{\mu_2}}{r|h_1|} \right).
\end{align*}

Nous s\'eparons cette derni\`ere somme en $S_{7,2}(r,h), \, S_{7,2}'(r,h)$ et $ S_{7,2}''(r,h) $ selon $|h_1| \leq q^{2\rho}$, $q^{2\rho} < |h_1| \leq q^{\mu_2-\mu_0} $ et $q^{\mu_2-\mu_0} < |h_1| \leq H $. En effet, si $|h_1| < q^\rho$, $\min\left(q^\mu, \frac{q^{\mu_2}}{r|h_1|} \right) = q^\mu$, alors que $\min\left(q^\mu, \frac{q^{\mu_2}}{r|h_1|} \right) = \frac{q^{\mu_2}}{r|h_1|}$ sinon. Notons ici un enjeu dans le fait que la sommation sur $n$ soit maximale.

En prenant en compte cette particularité, en utilisant \eqref{Parseval double} pour $S_{7,2}'$ et en effectuant le même découpage que $S_{7,1}'$ et la majoration $|a_{h_1}(q^{\mu_0-\mu_2},H)|\leq \frac{1}{\pi|h_1|}$ pour $S_{7,2}''$, on peut démontrer que  \begin{align*}
S_{7,2}'(r,h)+S_{7,2}''(r,h) \ll q^{\mu+\nu}/r.
\end{align*} Par exemple \begin{align*}
S_{7,2}'(r,h) &= q^{\nu+2(\mu_2-\mu_0)}\ds\sum_{k = \mu+\nu-4}^{\mu+\nu-1}\ds\sum_{q^{2\rho} < |h_1| \leq q^{\mu_2-\mu_0}}|a_{h_1}|^2S_6'(h_1,k)\frac{q^{\mu_2}}{r|h_1|}
\\& \ll q^{\nu+\mu_2-2\rho}\ds\sum_{k = \mu+\nu-4}^{\mu+\nu-1}\ds\sum_{0 \leq h_1 < q^{\mu_2-\mu_0}}S_6'(h_1,k),
\end{align*} et le même genre de calculs, adaptés avec le découpage de $S_{7,1}'$ s'applique pour obtenir la majoration de $S_{7,2}''(r,h)$.

Nous pouvons alors \'ecrire 
%\begin{align*}
%\ds\frac{1}{RS}\ds\sum_{1 \leq r < R}\ds\sum_{\substack{ 1 \leq s < S \\ 0 \leq h < q^{\mu_2-\mu_0} \\ q^{\mu_2-\mu_1} |  hs}}&\left(S_{7,2}'(r,h) + S_{7,2}''(r,h) \right)   \ll \ds\frac{1}{RS}\ds\sum_{1 \leq r < R}\ds\sum_{\substack{ 1 \leq s < S \\ 0 \leq h < q^{\mu_2-\mu_0} \\ q^{\mu_2-\mu_1} | hs}}\frac{q^{\nu + \mu}}{r}
%\\ & \ll q^{\mu+\nu}\ds\frac{\log R}{R} \ds\frac{\#\{1 \leq s < S , 0 \leq h < q^{\mu_2-\mu_0} : q^{\mu_2-\mu_1} |  hs \}}{S}
%\\ & \ll q^{\mu+\nu}\ds\frac{\log R}{R} \ds\frac{\#\{0 \leq k < q^{\mu_2-\mu_0}S  : q^{\mu_2-\mu_1} | k \}}{S}
%\\ & \ll q^{\mu+\nu}\ds\frac{\log R}{R} \ds\frac{\#\{0 \leq q^{\mu_2-\mu_1}k' < q^{\mu_2-\mu_0}S  \}}{S}
%\\ & \ll q^{\mu+\nu}\ds\frac{\log R}{R} \ds\frac{q^{\mu_2-\mu_0-\mu_2+\mu_1}S}{S},
%\end{align*} 
\begin{align*}
\ds\frac{1}{RS}\ds\sum_{1 \leq r < R}\ds\sum_{\substack{ 1 \leq s < S \\ 0 \leq h < q^{\mu_2-\mu_0} \\ q^{\mu_2-\mu_1} |  hs}}&\left(S_{7,2}'(r,h) + S_{7,2}''(r,h) \right)   \ll \ds\frac{1}{RS}\ds\sum_{1 \leq r < R}\ds\sum_{\substack{ 1 \leq s < S \\ 0 \leq h < q^{\mu_2-\mu_0} \\ q^{\mu_2-\mu_1} | hs}}\frac{q^{\nu + \mu}}{r}
\\ & \ll q^{\mu+\nu}\ds\frac{\log R}{R} \ds\frac{\#\{1 \leq s < S , 0 \leq h < q^{\mu_2-\mu_0} : q^{\mu_2-\mu_1} |  hs \}}{S}
\\ & \ll q^{\mu+\nu}\ds\frac{\log R}{R} \ds\frac{q^{\mu_2-\mu_0-\mu_2+\mu_1}S}{S},
\end{align*} 
ce qui revient \`a dire \begin{equation}\label{Valeur finale S72 et S72'}
\ds\frac{1}{RS}\ds\sum_{1 \leq r < R}\ds\sum_{\substack{ 1 \leq s < S \\ 0 \leq h < q^{\mu_2-\mu_0} \\ q^{\mu_2-\mu_1} |  hs}}\left(S_{7,2}'(r,h) + S_{7,2}''(r,h) \right)\ll q^{\mu+\nu+\mu_1-\mu_0}\ds\frac{\log R}{R}.
\end{equation} Traitons maintenant le cas $S_{7,2}(h,r)$. Nous rappelons que \begin{align*}
S_{7,2}(h,r) & :=q^{2(\mu_2-\mu_0)}\ds\sum_{k = \mu+\nu-4}^{\mu+\nu-1} \ds\sum_{|h_1|\leq q^{2\rho}}\left|a_{h_1}\left(q^{\mu_0-\mu_2},H\right)\right|^2 S_6'(h_1,k)  q^\nu \min\left(q^\mu, \frac{q^{\mu_2}}{r|h_1|} \right)
\\ & \ll q^{\mu+\nu}\ds\sum_{k = \mu+\nu-4}^{\mu+\nu-1} \ds\sum_{|h_1|\leq q^{2\rho}} S_6'(h_1,k)
\\ & \ll q^{\mu+\nu}\ds\sum_{k = \mu+\nu-4}^{\mu+\nu-1} \ds\sum_{|h_1|\leq q^{2\rho}}\ds\sum_{0 \leq h_3 < q^{\mu_2-\mu_0}}\left| \widetilde{\varphi_P}\big(h_3,k\big)\widetilde{\varphi_P}(h_3-h_1,k)\right|^2.
\end{align*}Supposons \begin{equation}\label{Estimation P(mu+nu+1) < 1/3}
P(\mu+\nu+1) \leq \frac{1}{3}(\mu_1-\mu_0)
\end{equation} Nous utilisons ici le Lemme \ref{Estimation Fourier SII} avec $\lambda = \mu_2-\mu_0-2\rho $ pour pouvoir dire \begin{align*}
\ds\sum_{|h_1|\leq q^{2\rho}}\ds\sum_{0 \leq h_3 < q^{\mu_2-\mu_0}}\left| \widetilde{\varphi_P}\big(h_3,k\big)\right.&\left.\widetilde{\varphi_P}(h_3-h_1,k)\right|^2
\\ & \ll q^{\frac{1}{2}(\mu_1-\mu_0-\gamma_P(\mu_2-\mu_0-2\rho,\mu+\nu-2)+\frac{3P(k)}{4}}(\log q^{\mu_2-\mu_1})^2, 
\end{align*} donc nous obtenons par croissance de $P$ : \begin{align*}
 \ds\frac{1}{RS}\ds\sum_{1 \leq r < R}\ds\sum_{\substack{ 1 \leq s < S \\ 0 \leq h < q^{\mu_2-\mu_0} \\ q^{\mu_2-\mu_1} \nmid  hs}}S_{7,2}(h,r) & \ll q^{(\mu+\nu)+\frac{1}{2}(\mu_1-\mu_0-\gamma_P(\mu_2-\mu_0-2\rho,\mu+\nu-2)+\frac{3}{4}P(\mu+\nu)}(\log q^{\mu_2-\mu_1})^2
 \\ & \quad \cdot \ds\frac{\#\{1 \leq s < S , 0 \leq h < q^{\mu_2-\mu_0} : q^{\mu_2-\mu_1} |  hs \}}{S},
\end{align*} 
ce qui nous permet de dire \begin{equation}\label{Equa 27}
\ds\frac{1}{RS}\ds\sum_{1 \leq r < R}\ds\sum_{\substack{ 1 \leq s < S \\ 0 \leq h < q^{\mu_2-\mu_0} \\ q^{\mu_2-\mu_1} |  hs}}S_7(h,r) \ll (\log q^{\mu_2-\mu_1})^2q^{(\mu+\nu)+\frac{3}{2}(\mu_1-\mu_0)-\frac{1}{2}\gamma_P(\mu_2-\mu_0-2\rho,\mu+\nu-2)+\frac{3}{4}P(\mu+\nu)}.
\end{equation} En rassemblant les \'equations \eqref{Premier VdC}, \eqref{SII premiere erreur}, \eqref{Equation SII deuxieme erreur}, \eqref{Equation S2(r,s)}, \eqref{Equation erreur double troncature OR}, \eqref{Equation erreur double troncature grand phi}, \eqref{Terme d'erreur Vaaler SII}, \eqref{S4 seconde}, \eqref{Transition S4 S5}, \eqref{S5'}, \eqref{Valeur finale S72 et S72'} et \eqref{Equa 27}, ainsi qu'en posant $S = q^{2\rho}$ nous avons :
%\begin{align*}
%\ds\frac{1}{RS}\ds\sum_{1 \leq r < R}\ds\sum_{1 \leq s < S}|S_4'(r,s)| \ll & (\log q^{\mu_2-\mu_1})^2q^{(\mu+\nu)+\frac{3}{2}(\mu_1-\mu_0)-\frac{1}{2}\gamma_P(\mu_2-\mu_0-2\rho,\mu+\nu-2)+\frac{3}{4}P(\mu+\nu)}
%\\ & + q^{\mu+\nu+\mu_1-\mu_0-\rho}\log q^\rho
%\\ & + q^{\frac{7}{8}(\mu+\nu)+\frac{3}{2}(\mu_2-\mu_1)+\frac{3}{2}(\mu_2-\mu_0)+\rho}.
%\end{align*} Ceci nous donne en utilisant \eqref{Equation S2(r,s)}, \eqref{Equation erreur double troncature OR}, \eqref{Equation erreur double troncature grand phi}, \eqref{Terme d'erreur Vaaler SII} et \eqref{S4 seconde}
 % :  
%\begin{align*}
%\ds\frac{1}{RS}\ds\sum_{1 \leq r < R}\ds\sum_{1 \leq s < S}|S_2'(r,s)| & \ll  (\log q^{\mu_2-\mu_1})^2q^{(\mu+\nu)+\frac{3}{2}(\mu_1-\mu_0)-\frac{1}{2}\gamma_P(\mu_2-\mu_0-2\rho,\mu+\nu-2)+\frac{3}{4}P(\mu+\nu)}
%\\ & \quad  + q^{\mu+\nu+\mu_1-\mu_0-\rho}\log q^\rho
%\\ & \quad + q^{\frac{7}{8}(\mu+\nu)+\frac{3}{2}(\mu_2-\mu_1)+\frac{3}{2}(\mu_2-\mu_0)+\rho}
%\\ & \quad + (\log q)^3(\mu+\nu)^3q^{\mu+\nu+3(\mu_2-\mu_0)+2\rho}(q^{-\mu_2}+q^{-\nu})
%\\ & \quad + \max(\log q^{\mu_0},\tau(q^{\mu_0}))q^{\mu+\nu-2\rho}
%\\ & \quad + \max (\tau(q),\log q)\mu_2^{\omega(q)}q^{\mu+\nu+\mu_0-\mu_1+P(\mu_2+1)}
%\\ & \quad + q^{\mu+\nu-\rho}(\mu+\nu).
%\end{align*} 
%Puis  par \eqref{Premier VdC}, \eqref{SII premiere erreur} et \eqref{Equation SII deuxieme erreur} :

\begin{align*}
\left|S_{II}(\vartheta)\right|^4 & \ll  (\log q^{\mu_2-\mu_1})^2q^{4(\mu+\nu)+\frac{3}{2}(\mu_1-\mu_0)-\frac{1}{2}\gamma_P(\mu_2-\mu_0-2\rho,\mu+\nu-2)+\frac{3}{4}P(\mu+\nu)}
\\ & \quad  + q^{4(\mu+\nu)+\mu_1-\mu_0-\rho}\log q^\rho
\\ & \quad + q^{(4-\frac{1}{8})(\mu+\nu)+\frac{3}{2}(\mu_2-\mu_1)+\frac{3}{2}(\mu_2-\mu_0)+\rho}
\\ & \quad + (\log q)^3(\mu+\nu)^3q^{4(\mu+\nu)+3(\mu_2-\mu_0)+2\rho}(q^{-\mu_2}+q^{-\nu})
\\ & \quad + \max(\log q^{\mu_0},\tau(q^{\mu_0}))q^{4(\mu+\nu)-2\rho}
\\ & \quad + \max (\tau(q),\log q)\mu_2^{\omega(q)}q^{4(\mu+\nu)+\mu_0-\mu_1+P(\mu_2+1)}
\\ & \quad + q^{4(\mu+\nu)-\rho}(\mu+\nu) + q^{4(\mu+\nu)-2\rho}+q^{4(\mu+\nu)-2\rho}.
\end{align*}

Nous faisons quelques r\'eductions \'el\'ementaires pour pouvoir dire 

\begin{align*}
\left|S_{II}(\vartheta)\right|^4 & \ll  (\log q^{\mu_2-\mu_1})^2q^{4(\mu+\nu)+\frac{3}{2}(\mu_1-\mu_0)-\frac{1}{2}\gamma_P(\mu_2-\mu_0-2\rho,\mu+\nu-2)+\frac{3}{4}P(\mu+\nu)}
\\ & \quad + q^{(4-\frac{1}{8})(\mu+\nu)+\frac{3}{2}(\mu_2-\mu_1)+\frac{3}{2}(\mu_2-\mu_0)+\rho}
\\ & \quad + (\log q)^3(\mu+\nu)^3q^{4(\mu+\nu)+3(\mu_2-\mu_0)+2\rho}(q^{-\mu_2}+q^{-\nu})
\\ & \quad + \max(\log q^\rho,\log q^{\mu_0},\tau(q^{\mu_0}),(\mu+\nu))q^{4(\mu+\nu)+\mu_1-\mu_0-\rho}
\\ & \quad + \max (\tau(q),\log q)\mu_2^{\omega(q)}q^{4(\mu+\nu)+\mu_0-\mu_1+P(\mu_2+1)} .
\end{align*}

Nous rappelons \`a pr\'esent que nous avons pos\'e $\mu_2=\mu+2\rho$, $ \mu_1 = \mu-3\rho$ et  $\mu_0 = \mu_1-2\rho' = \mu-3\rho-2\rho'$. Quitte \`a choisir $2\rho < \nu$, en remarquant que par \eqref{Estimation P(mu+nu+1) < 1/3}, nous avons $P(\mu+\nu+1) \leq 2\rho'/3$ et que par ailleurs, $\rho' \leq \rho \leq \mu$, et $\gamma_P(l,k)$ est croissante en $l$ et décroissante en $k$ (pour voir ce dernier point, on remarque que l'application 
$x \mapsto \log (x-K)/\log x
$ a sa dérivée positive si $x \geq K$ du fait de la croissance de $x \log x$), une fois posé
\begin{equation}\label{equa rho pour Anne}
P(\mu+\nu) \leq \gamma_P(3\rho, \mu+\nu)/3,
\end{equation} nous pouvons écrire :

\begin{align*}
\left|S_{II}(\vartheta)\right|^4 & \ll q^{4(\mu+\nu)}\left[ (\log q^{5\rho})^2q^{3\rho'-\frac{1}{4}\gamma_P(3\rho, \, \mu+\nu-2)} \right.
\\& \quad + q^{-\frac{1}{8}(\mu+\nu)+19\rho}
\\ & \quad + (\log q)^3(\mu+\nu)^3q^{23\rho}(q^{-\mu-2\rho}+q^{-\nu})
\\ & \quad + \max\big(\tau(q^{\mu-3\rho-2\rho'}),\mu+\nu \big)q^{2\rho'-\rho}
\\ & \quad \left. + \max (\tau(q),\log q)(\mu+2\rho)^{\omega(q)}q^{-\frac{4}{3}\rho'}\right].
\end{align*}
Nous faisons alors le choix de prendre $\rho = \lfloor \frac{\mu}{160} \rfloor \leq \frac{\mu}{160} \leq \frac{3}{4*160}(\mu+\nu)$ et $\rho' = \lfloor \frac{1}{16} \gamma_P(3\rho, \, \mu+\nu-2)) \rfloor \leq \frac{3}{32}\rho $. On voit ais\'ement que $\rho$ v\'erifie \eqref{equa rho nu} et d'apr\`es la Remarque \ref{Remarque sur la fonction gammaP} et par \eqref{Grosse restriction sur P}, nous sommes certains que \eqref{Estimation P(mu+nu+1) < 1/3} et \eqref{equa rho pour Anne} sont v\'erifi\'ees. Ces choix, avec la condition \eqref{Condis sur mu et nu SII} et les estimations classiques des parties enti\`eres, nous permettent d'écrire : 

\begin{align*}
\left|S_{II}(\vartheta)\right|^4 & \ll q^{4(\mu+\nu)}\left[ \max \big(\tau(q),\log q)(\mu+\nu)^{\omega(q)},(\log q^{5\rho})^2 \big)q^{-\frac{1}{16}\gamma_P(\frac{\mu+\nu}{640},\, \mu+\nu-2)}  \right.
\\ & \quad + q^{-\frac{184}{5120}(\mu+\nu)}+ (\log q)^3(\mu+\nu)^3q^{-\frac{139}{4*160} (\mu+\nu)}
\\ & \quad + (\log q)^3(\mu+\nu)^3q^{-\frac{137}{4*160} (\nu+\mu)}
\\ & \quad \left. + q^{\frac{26}{32}}\max\big( \tau(q^{\mu-3\rho-2\rho'}),\mu+\nu \big)q^{-\frac{26}{32*4*160} (\mu+\nu) }\right]
\\& \ll q^{4(\mu+\nu)}q^{\frac{26}{32}}\left(q^{-\frac{26}{20480} (\mu+\nu) }+q^{-\frac{1}{16}\gamma_P(\frac{\mu+\nu}{640},\, \mu+\nu-2))})\right)
\\& \qquad \max\bigg(\tau(q)(\mu+\nu)^{\omega(q)},\log q(\mu+2\rho)^{\omega(q)}
\\ & \qquad \qquad ,(\log q^{5\rho})^2  \tau(q^{\mu-3\rho-2\rho'}),(\log q)^3(\mu+\nu)^3 \bigg) .
\end{align*} 
Cependant, par la Remarque \ref{Remarque sur la fonction gammaP} : \begin{equation*}
\frac{1}{16} \gamma_P\left(\frac{\mu+\nu}{640},\, \mu+\nu-2\right)  \leq  \frac{\mu+\nu}{160*32}  \leq \frac{26}{20480} (\mu+\nu)
\end{equation*} %et donc, si on suppose $\mu \geq 3*160$, nous avons $\lfloor \frac{\mu}{160}\rfloor \geq \frac{\mu}{3*160}$ et du fait que $\gamma_P(l,k)$ est croissante en $l$ : \begin{align*}
%S_{II}(\vartheta) & \ll q^{\frac{26}{128}}\max\big(\tau(q),(\log q)^3\big)(\mu+\nu)^{\omega(q)+3}\big)^{1/4}q^{\mu+\nu-\frac{1}{64}\gamma_P(3\lfloor \frac{\mu}{160} \rfloor,\, \mu+\nu-2))})
%\\ & \ll q^{\frac{26}{128}}\max\big(\tau(q),(\log q)^3\big)(\mu+\nu)^{\omega(q)+3}\big)^{1/4}q^{\mu+\nu-\frac{1}{64}\gamma_P( \frac{\mu+\nu}{640},\, \mu+\nu-2))})
%\end{align*} 
et donc finalement:
\begin{equation}\label{Equation Finale SII}
S_{II}(\vartheta) \ll q^{\frac{26}{128}}\left(\max\big(\tau(q),(\log q)^3\big)(\mu+\nu)^{\omega(q)+3}\right)^{1/4}q^{\mu+\nu-\frac{1}{64}\gamma_P( \frac{\mu+\nu}{640},\, \mu+\nu-2)}.
\end{equation}
\end{proof}

\section{Fin de l'estimation}\label{Sec Fin estim}

%Nous avons par la Partie \ref{Section SI}
%\begin{align*}
%S_I(\vartheta) \ll (\mu+\nu)^3(\log q)^3q^{\mu+\nu- \frac{1}{4}\gamma_P(\frac{1}{3}(\mu + \nu), \, \mu+\nu)},
%\end{align*}

%et par la Partie \ref{Section SII}
Nous utilisons \cite[Lemme $1$]{MR:gelfond} et son pendant que l'on peut retrouver par \cite[equation ($13.40$)]{Iwaniec} avec nos estimations \eqref{Equation Finale SI}
 et \eqref{Equation Finale SII}.
%$$S_{II}(\vartheta) \ll q^{\frac{26}{128}}\left(\max\big(\tau(q),(\log q)^3\big)(\mu+\nu)^{\omega(q)+3}\right)^{1/4}q^{\mu+\nu-\frac{1}{64}\gamma_P( \frac{\mu+\nu}{640},\, \mu+\nu-2)}.$$
En rappelant que
$\gamma_P(l,k)$,
 est croissante en $l$ et d\'ecroissante en $k$ et en posant $$c_1(q) = q^{26/128}\max\left((\log q)^3,\tau(q)^{1/4}\right)(\log q)^{-3-\frac{\omega(q)}{4}}$$ et en rappelant que $x = q^{\mu+\nu}$, nous obtenons :
%On rappelle que $\gamma_P(l,k) = l\left(1-\ds\frac{\log \left(q^{P(k)}-8\left(\sin \frac{\pi |\!|\alpha |\!|}{4}\right)^2\right)}{P(k)\log q} \right) $ est croissante en $l$ et d\'ecroissante en $k$ du fait de la Proposition \ref{Croissance des fonctions}. Ainsi nous pouvons \'ecrire \begin{align*}
%S_I(\vartheta) & \ll (\mu+\nu)^3(\log q)^3q^{\mu+\nu- \frac{1}{4}\gamma_P(\frac{1}{3}(\mu + \nu), \, \mu+\nu)}
%\\ & \ll (\mu+\nu)^3(\log q)^3q^{\mu+\nu- \frac{1}{4}\gamma_P(\frac{\mu+\nu}{640}, \, \mu+\nu)}
%\\ & \ll q^{26/128}(\mu+\nu)^{3+\frac{\omega(q)}{4}}\max\left((\log q)^3,\tau(q)^{1/4}\right)q^{\mu+\nu- \frac{1}{64}\gamma_P(\frac{\mu+\nu}{640}, \, \mu+\nu)}
%\end{align*} et de même 
%\begin{align*}
%S_{II}(\vartheta) \ll q^{26/128}(\mu+\nu)^{3+\frac{\omega(q)}{4}}\max\left((\log q)^3,\tau(q)^{1/4}\right)q^{\mu+\nu- \frac{1}{64}\gamma_P(\frac{\mu+\nu}{640}, \, \mu+\nu)}.
%\end{align*} 
%Posons $c_1(q) = q^{26/128}\max\left((\log q)^3,\tau(q)^{1/4}\right)(\log q)^{-3-\frac{\omega(q)}{4}}$. En rappelant que $x = q^{\mu+\nu}$, nous pouvons alors utiliser les Lemmes \ref{VM SI et SII} et \ref{mu SI et SII} pour obtenir 
\begin{align*}
\left |\ds\sum_{x/q < n \leq x}\Lambda(n)f_P(n)e(\vartheta n) \right| \ll c_1(q)(\log x)^{3+\frac{\omega(q)}{4}}xq^{-\frac{1}{64}\gamma_P(\lfloor \frac{\log x}{\log q}\rfloor \frac{1}{160} , \, \lfloor \frac{\log x}{\log q}\rfloor)}
\end{align*} et 
la même équation en remplaçant $\Lambda$ par $\mu$.

Il reste \`a remplacer $x$ par $xq^{-k}$ et sommer sur $k$. Soit $K \in \N$ tel que $q^K \leq x^{1/4} < q^{K+1}$. Comme $\gamma_P(l,k)$ est croissante en $l$ et d\'ecroissante en $k$, nous avons 
\begin{align*}
\ds\sum_{k \leq K}xq^{-k}q^{-\frac{1}{64}\gamma_P(\lfloor \frac{\log xq^{-k}}{\log q}\rfloor \frac{1}{160} , \, \lfloor \frac{\log xq^{-k}}{\log q}\rfloor)}  & \leq \ds\sum_{k \leq K}xq^{-k}q^{-\frac{1}{64}\gamma_P(\lfloor \frac{\log xq^{-k}}{\log q}\rfloor \frac{1}{160} , \, \lfloor \frac{\log x}{\log q}\rfloor)}
\\ & \leq q^{-\frac{1}{64}\gamma_P(\lfloor \frac{\log x^{3/4}}{\log q}\rfloor \frac{1}{160} , \, \lfloor \frac{\log x}{\log q}\rfloor)}\ds\sum_{k \leq K}xq^{-k}
\\ & \ll xq^{-\frac{1}{64}\gamma_P(\lfloor \frac{\log x}{\log q}\rfloor \frac{1}{120} , \, \lfloor \frac{\log x}{\log q}\rfloor)},
\end{align*}
tandis que \begin{align*}
\ds\sum_{k > K}\frac{x}{q^k}q^{-\frac{1}{64}\gamma_P(\lfloor \frac{\log xq^{-k}}{\log q}\rfloor \frac{1}{160} , \, \lfloor \frac{\log xq^{-k}}{\log q}\rfloor)} & \leq \ds\sum_{k > K}\frac{x}{q^k}
\\ & \ll \frac{x}{q^{K+1}} \ll x^{3/4} \ll  xq^{-\frac{1}{64}\gamma_P(\lfloor \frac{\log x}{\log q}\rfloor \frac{1}{120} , \, \lfloor \frac{\log x}{\log q}\rfloor)}
\end{align*} 

et finalement \begin{equation}\label{Equation finale lambda}
\left |\ds\sum_{ n \leq x}\Lambda(n)f_P(n)e(\vartheta n) \right| \ll c_1(q)(\log x)^{3+\frac{\omega(q)}{4}}xq^{-\frac{1}{64}\gamma_P(\lfloor \frac{\log x}{\log q}\rfloor \frac{1}{120} , \, \lfloor \frac{\log x}{\log q}\rfloor)}
\end{equation} et \begin{equation}\label{Equation finale mu}
\left |\ds\sum_{ n \leq x}\mu(n)f_P(n)e(\vartheta n) \right| \ll c_1(q)(\log x)^{3+\frac{\omega(q)}{4}}xq^{-\frac{1}{64}\gamma_P(\lfloor \frac{\log x}{\log q}\rfloor \frac{1}{120} , \, \lfloor \frac{\log x}{\log q}\rfloor)}.
\end{equation}

\section{Conditions sur P}\label{Sec Choix de P}

Nous devons \`a pr\'esent choisir $P$ de sorte que \eqref{Condition sur P SI} et \eqref{Grosse restriction sur P} soient satisfaites.
Comme la fonction $P$ est croissante, que $\lfloor x \rfloor \geq x/2$, et que $\rho \leq \frac{3}{4*160}(\mu+\nu)$, les deux \'equations sont impliqu\'ees par \begin{equation}
P(2(\mu+\nu))\leq \frac{1}{48}\gamma_P\left(\frac{\mu+\nu}{640},\mu+\nu\right)
\end{equation}
%Ceci nous donne $P(\mu+\nu+1) \leq \frac{2}{3}\Big\lfloor \frac{1}{16} \gamma_P\left(\frac{\mu+\nu}{640}, \, \mu+\nu-2\right) \Big\rfloor$ et $P(\mu+\nu) \leq \gamma_P(\frac{1}{3}(\mu + \nu), \, \mu+\nu)$.

Constatons que %\begin{align*}
%\gamma_P(l,k) &= l\left(1-\ds\frac{\log \left(q^{P(k)}-8\left(\sin \frac{\pi |\!|\alpha |\!|}{4}\right)^2\right)}{P(k)\log q} \right)= -l\ds\frac{\log \left(1-\ds\frac{8\left(\sin \frac{\pi |\!|\alpha |\!|}{4}\right)^2}{q^{P(k)}}\right)}{P(k)\log q},
%\end{align*} et que 
pour tout $0 \leq x < 1$, nous avons $\log(1-x) \leq -x$, ce qui implique que \begin{align*}
\gamma_P(l,k) \geq l\frac{8\left(\sin \frac{\pi |\!|\alpha |\!|}{4}\right)^2}{q^{P(k)}P(k)\log q},
\end{align*} et donc il suffit que la fonction $P$ v\'erifie %\begin{equation*}
%P(2(\mu+\nu))\leq \frac{\mu+\nu}{640*48}\frac{8\left(\sin \frac{\pi |\!|\alpha |\!|}{4}\right)^2}{q^{P(\mu+\nu)}P(\mu+\nu)\log q}
%\end{equation*}
%ou encore 
\begin{equation}\label{Estimation restriction pour croissance de P}
P(2x)q^{P(x)}P(x)\log q \leq \frac{x}{640*48}8\left(\sin \frac{\pi |\!|\alpha |\!|}{4}\right)^2.
\end{equation}
 et nous constatons que pour tout $0 < c < 1/\log q$, toute fonction $P(x)$ v\'erifiant $P(x) \leq c \log x$ v\'erifie \eqref{Estimation restriction pour croissance de P}.

\smallskip

Pour que les \'equations \eqref{Equation finale lambda} et \eqref{Equation finale mu} ne soient pas triviales, il faut que \begin{align}\label{Condi P finale infini}
& c_1(q)(\log x)^{3+\frac{\omega(q)}{4}}xq^{-\frac{1}{64}\gamma_P(\lfloor \frac{\log x}{\log q}\rfloor \frac{1}{120} , \, \lfloor \frac{\log x}{\log q}\rfloor)} = o(x)
\end{align} %ce qui revient \`a dire
%\begin{align*}
 %c_1(q)(\log x)^{3+\frac{\omega(q)}{4}}q^{-\frac{1}{64}\gamma_P(\lfloor \frac{\log x}{\log q}\rfloor \frac{1}{120} , \, \lfloor \frac{\log x}{\log q}\rfloor)} =o(1)
%\end{align*}
%ou encore 
%\begin{align*}
%\frac{12+\omega(q)}{4 \log q }\log \log x - \frac{1}{64*120} \Big\lfloor \frac{\log x}{\log q} \Big\rfloor\left( 1- \frac{\log \left(q^{P(\lfloor \frac{\log x}{\log q}\rfloor)}-8\left(\sin \frac{\pi |\!|\alpha |\!|}{4}\right)^2 \right)}{P(\lfloor \frac{\log x}{\log q}\rfloor)\log q}\right) \underset{x \rightarrow \infty}{\rightarrow} - \infty,
%\end{align*} 
Comme $0 < c < 1/\log q$ et $P(x) < c \log x$, nous avons  \begin{align*}
 \frac{12+\omega(q)}{4 \log q }&\log \log x  - \frac{1}{64*120} \Big\lfloor \frac{\log x}{\log q} \Big\rfloor \cdot \frac{8\left(\sin \frac{\pi |\!|\alpha |\!|}{4}\right)^2}{q^{P(\lfloor \frac{\log x}{\log q}\rfloor)}P(\lfloor \frac{\log x}{\log q}\rfloor)\log q} 
\\ & \notag \leq 
 \frac{12+\omega(q)}{4 \log q }\log \log x - \frac{\left(8\sin \frac{\pi |\!|\alpha |\!|}{4}\right)^2}{64*120c\log q} \cdot \frac{1}{\log(\lfloor \frac{\log x}{\log q}\rfloor)} \cdot\Big\lfloor \frac{\log x}{\log q} \Big\rfloor^{1-c\log q}\underset{x \rightarrow \infty}{\rightarrow} - \infty, 
\end{align*} de sorte que \eqref{Condi P finale infini} est satisfaite.
\section{Non automaticit\'e}\label{Annexe 4}

Dans cette partie, nous montrons que les suites $((-1)^{a_P(n)})_{n \geq 0}$ \'etudi\'ees ne sont pas automatiques dans le cas $q=2$. Ceci nous permet d'affirmer que les r\'esultats de cet article ne sont pas recouverts par \cite{MuellnerAutomate}. La preuve que nous pr\'esentons, plus élégante que notre preuve initiale est due \`a Julien Cassaigne, que nous remercions pour nous avoir autorisé à la reproduire ici.

Soit $P : \N \rightarrow \N$ une fonction croissante majorée par $\log n / \log 2$.
Nous allons montrer que l'ensemble $E$ des entiers possédant un nombre pair de blocs de
taille $P(T_2(n))$ "$1$" cons\'ecutifs en base $2$ n'est reconnaissable par aucun
automate fini.

Supposons que $E$ est reconnu par un automate \`a $k$ \'etats. Parmi les $k+1$ mots
$x_l = 0^{k-l} 1^l$, pour $l$ de $0$ \`a $k$, il y en a donc deux qui conduisent au
m\^eme \'etat, disons $x_i$ et $x_j, i<j$. Soit $m$ un entier suffisamment grand pour
que $P(m)>k$ et $m>P(m)+k$. Soit $y = 1^{P(m)-j} 0^{m-P(m)+j-k}$. Alors les mots
$x_i\cdot y$ et $x_j \cdot y$ sont deux mots de longueur $m$ qui conduisent au m\^eme \'etat,
et pourtant le premier code un élément de $E$ (il ne contient aucun bloc de
taille $P(m)$) et pas le second (il contient exactement un bloc de taille $P(m)$). Le choix de $m$ est valable puisque $\min (P(m), m-P(m))$ est non major\'e dans notre cas.

\section{Optimalité de la méthode}\label{Optim}

%\section{Cas de blocs de grande taille}

Dans cette partie, nous allons montrer que si $P(x) > 2c \log x$ alors les résultats obtenus dans la Partie\ref{Sec Fin estim} se trouvent être mis en défaut.

Les r\'esultats obtenus dans \cite{Hanna1} (tout comme ceux obtenus dans \cite{RudinShapiro}) sont tr\`es d\'ependants de ce qui est nomm\'e la propri\'et\'e de Fourier, qui implique notamment qu'il existe une application $\gamma$ tendant vers l'infini \`a l'infini telle que, si $k$ est un entier sup\'erieur ou \'egal \`a $1$, \begin{align}\label{Prop Fourier simplifiee}
\ds\frac{1}{2^N}\left|\ds\sum_{n < 2^N}e\left(\alpha \ds\sum_{i \geq 0}\epsilon_{i}(n)\cdots\epsilon_{i+k}(n)\right) \right| \leq 2^{-\gamma(N)}.
\end{align} On se convainc aisément que ce problème est quasi identique à celui rencontré dans cet article pour le contrôle de la transformée de Fourier. L'adapter à notre cas n'entrainerait qu'un alourdissement des notations. Par ailleurs c'est la propriété de propagation qui a posé problème dans cet article, et qui a poussé à adapter la méthode.

  En prenant $\alpha = 1/2$, %ce qui correspond \`a la question de Kalai, 
nous cherchons donc \`a montrer que \begin{align*}
\ds\frac{1}{2^N}\left|E_k(2^N)-I_k(2^N) \right| \leq 2^{-\gamma(N)},
\end{align*} o\`u $E_k(m)$ d\'esigne le nombre de $u < m$ tels que $u$ a un nombre pair de blocs de taille $k+1$ de $`1$' dans son \'ecriture binaire, et $I_k(m)$, un nombre impair. El\'ementairement : 
\begin{align}\label{Transfo}
\left|E_k(2^N)-I_k(2^N) \right| = 2^N - 2\min (E_k(2^N), I_k(2^N)) = 2\max (E_k(2^N), I_k(2^N)) - 2^N.
\end{align} 

Nous allons montrer %dans la Partie \ref{Section un resultat proba} l
le r\'esultat n\'egatif suivant :

\begin{proposition}\label{Prop les pairs sont trop gros}
Soient $N$ et $k$ deux entiers, et $E_k(2^N)$ et $I_k(2^N)$ d\'efinis comme pr\'ec\'edemment .

Nous avons pour tout $A>1$, et pour tout entier $k \geq A \log_2(N)$: 
\begin{equation*}
\ds\frac{E_k(2^N)}{2^N} \geq  1-2N^{1-A} + o(N^{1-A})  \qquad \text{et donc} \qquad \ds\frac{I_k(2^N)}{2^N} \leq 2N^{1-A} + o(N^{1-A}).
\end{equation*}
\end{proposition}

Par \eqref{Transfo}, ceci veut dire que sous les conditions de la proposition, on a pour $N$ assez grand \begin{align*}
\ds\frac{1}{2^N}\left|\ds\sum_{u < 2^N}e\left(\frac{1}{2} \ds\sum_{i \geq 0}\epsilon_{i+k}(u)\cdots\epsilon_i(u)\right) \right| &= \frac{2\max (E_k(2^N), I_k(2^N)) - 2^N}{2^N}
\\ & \geq 2(1-2N^{1-A})+o(N^{1-A})-1 
\\ & = 1-4N^{1-A}+o(N^{1-A})
\end{align*} et en particulier \eqref{Prop Fourier simplifiee} n'a aucune chance de se r\'ealiser. Le Th\'eor\`eme suivant d\'ecoule facilement de la Proposition \ref{Prop les pairs sont trop gros}.
\begin{theorem}\label{Thm principal chap 3}
Soient $f : \N \rightarrow \N$ une application, $A>1$ un r\'eel, $N$ un entier positif, et $$P_N(X_1,\ldots ,X_N) = \ds\sum_{i = 1}^{N-k}X_i\ldots X_{i+k-1}$$ un polyn\^ome, avec  $k$ un entier sup\'erieur ou \'egal \`a $ A\log N /\log 2$. Alors, si pour tout entier $n$ inf\'erieur \`a $2^N$, nous notons $P_N(n) := P_N(\epsilon_0(n),\ldots , \epsilon_{N-1}(n))$ o\`u les $\epsilon_i$ d\'esignent les chiffres binaires de $n$, nous avons $$\ds\sum_{n < 2^N}f(n)(-1)^{P_N(n)} = \ds\sum_{n < 2^N}f(n) + \varepsilon(N),$$ avec $|\varepsilon(N)| \leq 2^{N+1}N^{1-A}\sup_{n < 2^N}|f(n)|(1+o(1))$.
\end{theorem} 
\begin{proof}
Soient $E_k$ et $I_k$ d\'efinis comme pr\'ec\'edemment. Sous les conditions de la Proposition \ref{Prop les pairs sont trop gros}, nous avons pour $N \gg_A 1$ : \begin{align}\label{Equa I}
I_k(2^N) = \min(E_k(2^N),I_k(2^N)) = 2^N - \max(E_k(2^N),I_k(2^N)) 
\end{align} et \begin{align*}
\ds\sum_{n < 2^N}f(n)(-1)^{P_N(n)} &= \ds\sum_{\substack{n < 2^N \\ 2 | P_N(n)}}f(n) - \ds\sum_{\substack{n < 2^N \\ 2 \nmid P_N(n)}}f(n)
\\ & = \ds\sum_{n < 2^N}f(n) - 2\ds\sum_{\substack{n < 2^N \\ 2 \nmid P_N(n)}}f(n).
\end{align*} \`A pr\'esent remarquons que 
\begin{align*}
\left|\ds\sum_{\substack{n < 2^N \\ 2 \nmid P_N(n)}}f(n)\right| & \leq \sup_{n < 2^N}|f(n)| \#\{n<2^N : P_N(n) \not\equiv 0 \bmod 2\}
\\ & = \sup_{n < 2^N}|f(n)| I_k(2^N)
\end{align*} et la Proposition \ref{Prop les pairs sont trop gros} permet de conclure.
\end{proof}

En appliquant la proposition avec $f(n) = \mu(n)$ et en utilisant le th\'eor\`eme des nombres premiers sous la forme $\sum_{n \leq x}\mu(n) = o(x)$, nous obtenons que pour tout $A > 1$, \begin{align*}
\ds\sum_{n < 2N}\mu(n)(-1)^{P_N(n)} = o(2^N)
\end{align*} avec %$$P_N(X_1,\ldots ,X_N) = \ds\sum_{i = 1}^{N-k}X_i\cdots X_{i+k-1}$$ 
le polyn\^ome regardé dans l'énoncé où $k \geq A\log N /\log 2$ et $P_N(n) = P_N(\epsilon_0(n), \ldots, \epsilon_{N-1}(n))$. De m\^eme, avec cette d\'efinition de $P_N(n)$, le th\'eor\`eme des nombres premiers permet de conclure que pour tout $A > 2$ et $k \geq A \log N / \log 2$, \begin{align*}
\ds\sum_{n < 2N}\Lambda(n)(-1)^{P_N(n)} \sim 2^N.
\end{align*} En effet, $\sup_{n < 2^N}\Lambda (n) \leq N\log 2 < N$ et $|\varepsilon(N)| \leq N^{1-A}2^{N+1}\sup_{n < 2^N}\Lambda (n)(1+o(1)) = o(2^N)$, car $A > 2$.  Ceci montre en particulier que toute méthode utilisant l'identité de Vaughan est vouée à l'échec dans ce cas précis.

%\section{Un r\'esultat probabiliste}\label{Section un resultat proba}
La preuve de la Proposition \ref{Prop les pairs sont trop gros} repose sur un r\'esultat probabiliste. Ce n'est pas si \'etonnant si on tient compte du fait que, pour un entier $n$ pris al\'eatoirement entre $0$ et $2^N -1$, en \'ecrivant $$n = \ds\sum_{i = 0}^{N-1}\epsilon_i(n)2^i,$$ o\`u $\epsilon_{N-1}(n)$ peut \^etre \'egal \`a $0$, les $\epsilon_i$ suivent des lois de Bernoulli et sont ind\'ependantes et uniform\'emement distribu\'ees. 

Introduisons des notations : soit $\mathcal{A} := \{0,1\}$ l'alphabet \`a deux \'el\'ements, $\mathcal{A}^N$ les mots sur $\mathcal{A}$ de taille $N$. Nous identifions totalement $\{0,\ldots, 2^N-1\}$ \`a $\mathcal{A}^N$. Pour $\omega$ un mot de taille $N$, nous notons $\mathcal{N}(\omega)$ son nombre de blocs constitu\'es du m\^eme chiffre, pour chaque $i$-i\`eme bloc, $X_i(\omega)$ sa taille. Nous notons enfin $\mathcal{M}(\omega)$ la taille du plus grand bloc constitu\'e du m\^eme chiffre. 

Ainsi $\omega := 111000011$ est constitu\'e de trois blocs : `$111$' de taille $3$, `$0000$' de taille $4$ et `$11$' de taille $2$. Donc dans ce cas $\mathcal{N}(\omega) = 3$, $X_1(\omega) = 3, X_2(\omega) = 4, X_3(\omega) = 2$ et $\mathcal{M}(\omega) = 4$.  

Nous commen\c cons cette partie en g\'en\'erant un mot al\'eatoire de taille $N$. Cette construction est standard et peut \^etre retrouv\'ee dans \cite{Regine}.

\begin{construction}
Soient $(Z_i)_{i \geq 0}$ des lois g\'eom\'etriques de param\`etre $1/2$ d\'efinies sur un espace de probabilit\'es $(\Omega, \mathcal{F}, \mathbb{P})$. Supposons-les ind\'ependantes et identiquement distribu\'ees. Soit $\epsilon$ une loi de Bernoulli de param\`etre $1/2$ d\'efinie sur $\Omega$ et ind\'ependante des $Z_i$. Alors nous construisons une suite infinie de $0$ et de $1$ de la mani\`ere suivante : 
\begin{itemize}
\item si $\epsilon = 1$ nous \'ecrivons $Z_1$ $ `0$', puis $Z_2$ $`1$' puis $Z_3 `0$' etc.
\item si $\epsilon = 0$ nous \'ecrivons $Z_1$ $`1$', puis $Z_2$ $`0$' puis $Z_3 `1$' etc.
\end{itemize} Gardant les premiers symboles, nous obtenons un mot al\'eatoire de taille $N$. Avec ces notations, nous avons \begin{align*}
\mathcal{N}(\omega) = \inf \left\{k \in \N, \quad \ds\sum_{i = 1}^k Z_i(\omega) \geq N \right\}
\end{align*} et \begin{equation}
\forall i \in \{1,\ldots, \mathcal{N}(\omega)-1\} : X_i(\omega) = Z_i \quad \text{ et } \quad X_{\mathcal{N}(\omega)}(\omega) = N-\ds\sum_{i = 1}^{\mathcal{N}(\omega)-1}Z_i \leq Z_{\mathcal{N}(\omega)}.
\end{equation}
\end{construction}

Nous avons alors le r\'esultat suivant : 
\begin{proposition}\label{Res proba}
Pour tout $1<A<2$, tout $N \geq 2$ et tout $\omega \in \mathcal{A}^l$ : \begin{align*}
\mathbb{P}(\mathcal{M}(\omega) \geq A \log N / \log 2 ) \leq 2N^{1-A}.
\end{align*}
\end{proposition}

\begin{proof}
Soit $y > 0$, comme pour tout $i$, $X_i \leq Z_i$ et que $\mathcal{N}(\omega) \leq N$, nous avons \begin{align*}
\mathbb{P}(\mathcal{M}(\omega) < y) &= \mathbb{P}(\forall 1 \leq i \leq \mathcal{N}(\omega) : X_i(\omega) < y )
\\ & \geq \mathbb{P}(\forall 1 \leq i \leq \mathcal{N}(\omega) : Z_i < y )
\\ & \geq \mathbb{P}(\forall 1 \leq i \leq N : Z_i < y )
\\ & \geq (1-2^{-\lfloor y \rfloor})^N.
\end{align*} En appliquant ce r\'esultat \`a $y = A \log N / \log 2$, ceci montre que \begin{align*}
\mathbb{P}(\mathcal{M}(\omega) \geq A \log N / \log 2 ) &= 1 - \mathbb{P}(\mathcal{M}(\omega) < A \log N / \log 2 )
\\ & \leq 1-(1-2^{-\lfloor A \log N / \log 2 \rfloor})^N
\\ & \leq N2^{-\lfloor A \log N / \log 2 \rfloor} + o(N2^{-\lfloor A \log N / \log 2 \rfloor})
\\ & \leq N2^{1-A \log N / \log 2}+o(N2^{-\lfloor A \log N / \log 2 \rfloor}) 
\\ &= 2N^{1-A} + o(N^{1-A}).
\end{align*}
\end{proof}

Voyons maintenant comment la Proposition \ref{Res proba} nous donne la Proposition \ref{Prop les pairs sont trop gros}.

\begin{proof}[D\'emonstration de la Proposition \ref{Prop les pairs sont trop gros}]
Nous avons clairement, comme $0$ est pair : 
\begin{align*}
E_k(2^N) &\geq  2^l\cdot \mathbb{P}(\mathcal{M}(\omega) < k)
\\& \geq 2^l\cdot \mathbb{P}(\mathcal{M}(\omega) < A\log N / \log 2)
\\& = 2^l \cdot (1-\mathbb{P}(\mathcal{M}(\omega) \geq A\log N / \log 2))
\\ & \geq 2^N \cdot \left(1-2N^{1-A}\right) + o(2^NN^{1-A}),
\end{align*} ce qui conclut notre preuve.
\end{proof}

\end{document}